\documentclass[11pt, a4paper, oneside]{amsart}

\usepackage[english]{babel}
\usepackage{amsmath, amsthm, amsfonts, mathrsfs, amssymb}
\usepackage{mathtools}
\mathtoolsset{centercolon}
\usepackage{booktabs}
\usepackage[shortlabels]{enumitem}
\setlist[itemize]{leftmargin=20pt}
\usepackage[colorlinks, citecolor = blue]{hyperref}
\usepackage[hmargin=3cm,vmargin=3cm]{geometry}
\usepackage[color=green!40]{todonotes}
\usepackage{comment}

\makeatletter
\DeclareFontFamily{OMX}{MnSymbolE}{}
\DeclareSymbolFont{MnLargeSymbols}{OMX}{MnSymbolE}{m}{n}
\SetSymbolFont{MnLargeSymbols}{bold}{OMX}{MnSymbolE}{b}{n}
\DeclareFontShape{OMX}{MnSymbolE}{m}{n}{
    <-6>  MnSymbolE5
   <6-7>  MnSymbolE6
   <7-8>  MnSymbolE7
   <8-9>  MnSymbolE8
   <9-10> MnSymbolE9
  <10-12> MnSymbolE10
  <12->   MnSymbolE12
}{}
\DeclareFontShape{OMX}{MnSymbolE}{b}{n}{
    <-6>  MnSymbolE-Bold5
   <6-7>  MnSymbolE-Bold6
   <7-8>  MnSymbolE-Bold7
   <8-9>  MnSymbolE-Bold8
   <9-10> MnSymbolE-Bold9
  <10-12> MnSymbolE-Bold10
  <12->   MnSymbolE-Bold12
}{}

\let\llangle\@undefined
\let\rrangle\@undefined
\DeclareMathDelimiter{\llangle}{\mathopen}%
                     {MnLargeSymbols}{'164}{MnLargeSymbols}{'164}
\DeclareMathDelimiter{\rrangle}{\mathclose}%
                     {MnLargeSymbols}{'171}{MnLargeSymbols}{'171}
\makeatother

\usepackage{color}
\usepackage{graphicx}
\usepackage{tikz-cd}
\usetikzlibrary{arrows}

\newcommand{\F}{\ensuremath{\mathbf{F}}}

\newcommand{\N}{\ensuremath{\mathbf{N}}}
\newcommand{\Z}{\ensuremath{\mathbf{Z}}}
\newcommand{\Q}{\ensuremath{\mathbf{Q}}}
\newcommand{\R}{\ensuremath{\mathbf{R}}}
\newcommand{\C}{\ensuremath{\mathbf{C}}}

\newcommand{\mb}{\mathbf}
\newcommand{\mc}{\mathcal}

\DeclareMathOperator{\ind}{\mathbf{1}}

\DeclareMathOperator*{\esssup}{ess\,sup}

\newcommand{\1}{\mathbf{1}} 

\renewcommand{\emptyset}{\varnothing}
\def\avint_#1{\mathchoice{\mathop{\kern 0.2em\vrule width 0.6em height 0.69678ex depth -0.58065ex \kern -0.8em \intop}\nolimits_{\kern -0.4em#1}}{\mathop{\kern 0.1em\vrule width 0.5em height 0.69678ex depth -0.60387ex \kern -0.6em \intop}\nolimits_{#1}} {\mathop{\kern 0.1em\vrule width 0.5em height 0.69678ex depth -0.60387ex \kern -0.6em \intop}\nolimits_{#1}} {\mathop{\kern 0.1em\vrule width 0.5em height 0.69678ex depth -0.60387ex \kern -0.6em \intop}\nolimits_{#1}}}

\newtheorem{TheoremLetter}{Theorem}
{}
\newtheorem{theorem}{Theorem}
\newtheorem{corollary}[theorem]{Corollary}
\newtheorem{lemma}[theorem]{Lemma}
\newtheorem{proposition}[theorem]{Proposition}

{}

{}

\theoremstyle{remark}
\newtheorem{remark}[theorem]{Remark}
\newtheorem{example}[theorem]{Example}

\theoremstyle{definition}
\newtheorem{definition}[theorem]{Definition}

\numberwithin{theorem}{section}
\numberwithin{equation}{section}

\title{Multilinear matrix weights}

\author{Spyridon~Kakaroumpas}
\address{Spyridon~Kakaroumpas (he/him), Institut f\"{u}r Mathematik, Julius-Maximilians-Universit\"{a}t W\"{u}rzburg, Emil-Fischer-Stra{\ss}e 41, 97074 W\"{u}rzburg, Germany}
\email{spyridon.kakaroumpas@uni-wuerzburg.de}

\author{Zoe Nieraeth}
\thanks{Z. N. is supported by the Basque government through project GV IT1615-22}
\address{Zoe Nieraeth (she/her), UPV/EHU\textendash University of the Basque country, Leioa, Spain}
\email{zoe.nieraeth@gmail.com}
\allowdisplaybreaks

\begin{document}
\begin{abstract}
In this work we fully characterize the classes of matrix weights for which multilinear Calder\'on--Zygmund operators extend to bounded operators on matrix weighted Lebesgue spaces. To this end, we develop the theory of multilinear singular integrals taking values in tensor products of finite dimensional Hilbert spaces. On the one hand, we establish quantitative bounds in terms of multilinear Muckenhoupt matrix weight characteristics and scalar Fujii--Wilson conditions of a tensor product analogue of the convex body sparse operator, of a convex-set valued tensor product analogue of the Hardy--Littlewood maximal operator, and of a multilinear analogue of the Christ--Goldberg maximal operator. These bounds recover the sharpest known bounds in the linear case. Moreover, we define a notion of directional nondegeneracy for multilinear Calder\'{o}n--Zygmund operators, which is new even in the scalar case. The noncommutavity of matrix multiplication, the absence of duality, and the natural presence of quasinorms in the multilinear setting present several new difficulties in comparison to previous works in the scalar or in the linear case. To overcome them, we use techniques inspired from convex combinatorics and differential geometry.
\end{abstract}




\maketitle

\section{Introduction}

\subsection{The Hunt--Muckenhoupt--Wheeden theorem}
In the study of singular integrals, the problem of bounding Calder\'on--Zygmund operators on spaces of functions such as the weighted Lebesgue space
\[
\|f\|_{L^p_w(\R^d)}:=\Big(\int_{\R^d}\!|w(x)f(x)|^p\,\mathrm{d}x\Big)^{\frac{1}{p}}
\]
is of fundamental interest. In \cite{Mu72}, Muckenhoupt showed that the Hardy--Littlewood maximal operator
\[
Mf(x):=\sup_Q\Big(\avint_Q\!|f(x)|\,\mathrm{d}x\Big)\ind_Q(x),
\]
where the supremum is taken over all cubes $Q$ in $\R^d$,
is bounded in $L^p_w(\R^d)$ for $1<p\leq\infty$ if and only if $w$ satisfies the Muckenhoupt $A_p$ condition $w\in A_p$:
\[
[w]_p:=\sup_Q\Big(\avint_Q\!w(x)^p\,\mathrm{d}x\Big)^{\frac{1}{p}}\Big(\avint_Q\!w(x)^{-p'}\,\mathrm{d}x\Big)^{\frac{1}{p'}}<\infty.
\]
Here $\tfrac{1}{p'}:=1-\frac{1}{p}$, and the average is interpreted as an essential supremum when the exponent is infinite. This result was extended by Hunt, Muckenhoupt, and Wheeden in \cite{HMW73}, who showed that all Calder\'on--Zygmund operators $T$ are bounded in $L^p_w(\R^d)$ if and only if $1<p<\infty$ and $w\in A_p$.

In the `90s, a certain vector-valued extension of the Hunt--Muckenhoupt--Wheeden theorem was sought after by Nazarov, Treil, and Volberg, see \cite{Tr89, NT96, TV97a, TV97b, Vo97}. If $f:\R^d\to\C^n$, a matrix weighted analogue of this problem can be formulated as follows: suppose $W:\R^d\to\C^{n\times n}$ is a.e. a Hermitian positive definite matrix and set
\[
\|f\|_{L^p_W(\R^d;\C^n)}:=\Big(\int_{\R^d}\!|W(x)f(x)|^p\,\mathrm{d}x\Big)^{\frac{1}{p}}.
\]
Then one can ask when the pointwise extension $\widetilde{T}$ of a Calder\'on--Zygmund operator $T$, defined by
\begin{equation}\label{eq:introcomponentwiseextension}
\widetilde{T}f(x):=(Tf_1(x),\ldots,Tf_n(x)),
\end{equation}
where $f=(f_1,\ldots,f_n)$ with $f_k:\R^d\to\C$, is bounded in $L^p_W(\R^d;\C^n)$. Nazarov, Treil, and Volberg showed that this is the case precisely when $1<p<\infty$, and $W$ satisfies the matrix Muckenhoupt $A_p$ condition $W\in A_p$: $W^{-1}\in L^{p'}_{\text{loc}}(\R^d;\C^{n\times n})$, and there is a $C\geq 1$ such that for all cubes $Q\subseteq\R^d$ and all $u\in\C^n$ there is a $v\in\C^n$ such that
\[
\Big(\avint_Q\!|W(x)u|^p\,\mathrm{d}x\Big)^{\frac{1}{p}}\Big(\avint_Q\!|W(x)^{-1}v|^{p'}\,\mathrm{d}x\Big)^{\frac{1}{p'}}\leq C|\langle u, v\rangle_{\C^n}|,
\]
where $\langle u, v\rangle_{\C^n}=\sum_{k=1}^n u_k\overline{v_k}$. We denote the smallest possible $C$ by $[W]_p$. This is equivalent to the uniform boundedness in $L^p_W(\R^d;\C^n)$ of the family of averaging operators
\[
T_Qf:=\ind_Q\avint_Q\!f(x)\,\mathrm{d}x,
\]
indexed over all cubes $Q$ in $\R^d$, in which case we have
\[
[W]_p=\sup_Q\|T_Q\|_{L^p_W(\R^d;C^n)\to L^p_W(\R^d;C^n)}.
\]
To obtain a characterization of $A_p$ in terms of a Hardy--Littlewood type maximal operator, in \cite{CG01}, Christ and Goldberg defined the operator
\[
M_Wf(x):=\sup_Q\Big(\avint_Q\!|W(x)W(y)^{-1}f(y)|\,\mathrm{d}y\Big)\ind_Q(x),
\]
and showed that $W\in A_p$ precisely when $M_W:L^p(\R^d;\C^n)\to L^p(\R^d)$. In order to prove the Rubio de Francia extrapolation theorem for matrix weights, Bownik and Cruz-Uribe in \cite{BC23} needed an intrinsically defined maximal operator that maps a space to itself. To this end, they considered convex-set valued measurable mappings $F:\R^d\to\mc{K}(\R^n)$, where $\mc{K}(\R^n)$ denotes the collection of compact symmetric convex sets in $\R^n$. They then defined $M^{\mc{K}}F(x)$ to be the smallest set in $\mc{K}(\R^n)$ containing the set
\[
\bigcup_Q\,\langle F\rangle_Q\ind_Q(x),
\]
where the average is defined through the so-called Aumann integral
\[
\langle F\rangle_Q:=\Big\{\avint_Q\!f\,\mathrm{d}x:f\in L^0(\R^d;\R^n),\, f(x)\in F(x)\, a.e.\Big\}.
\]
This is well-defined, assuming $F$ is \emph{locally integrably bounded}, i.e., for each cube $Q$ there is a constant $C$ such that for all measurable $f:\R^d\to\R^n$ with $f(x)\in F(x)$ a.e. we have
\[
\|f\|_{L^1(Q)}\leq C.
\]
Defining $L^p_W(\R^d;\mc{K}(\R^n))$ as the measurable mappings $F:\R^d\to\mc{K}(\R^n)$ for which 
\[
h(x):=\sup_{u\in F(x)}|W(x)u|
\]
satisfies $h\in L^p(\R^d)$, they showed that $W\in A_p$ precisely when $1<p\leq \infty$ and
\[
M^{\mc{K}}:L^p_W(\R^d;\mc{K}(\R^n))\to L^p_W(\R^d;\mc{K}(\R^n)).
\]

An alternative characterization of the matrix Muckenhoupt condition was found by Roudenko \cite{Ro03}, who showed that
\begin{equation}\label{eq:roudenkointro}
[W]_p\eqsim\sup_Q\Big(\avint_Q\Big(\avint_Q\!\|W(x)W(y)^{-1}\|^{p'}\,\mathrm{d}y\Big)^{\frac{p}{p'}}\,\mathrm{d}x\Big)^{\frac{1}{p}},
\end{equation}
where $\|A\|$ denotes the operator norm of a matrix $A\in\C^{n\times n}$. An essential ingredient in establishing this equivalence is the John ellipsoid theorem, which says that if $\rho:\C^n\to[0,\infty)$ is a norm, then there is a Hermitian positive definite matrix $A\in\C^{n\times n}$ for which
\[
\rho(u)\leq|Au|\leq n^{\frac{1}{2}}\rho(u).
\]
One can then apply this result to the norms
\[
u\mapsto \Big(\avint_Q\!|W(x)u|^p\,\mathrm{d}x\Big)^{\frac{1}{p}},\quad u\mapsto \Big(\avint_Q\!|W(x)^{-1}u|^{p'}\,\mathrm{d}x\Big)^{\frac{1}{p'}}
\]
to obtain the so-called reducing matrices $A_{W,Q,p}$, $A_{W^{-1},Q,p'}$, which yield another characterization of $A_p$ through
\[
[W]_p\eqsim\sup_Q\|A_{W,Q,p} A_{W^{-1},Q,p'}\|.
\]
From this, one can deduce \eqref{eq:roudenkointro}.

A different way of extending the Hunt--Muckenhoupt--Wheeden theorem is by considered \emph{multilinear} Calder\'on--Zygmund operators $T$, which are singular integrals acting on $m$ functions $f_1,\ldots,f_m$. It was shown by Lerner, Ombrosi, P\'erez, Torres, and Trujillo-Gonz\'alez in \cite{LOPTT09} that such maps satisfy
\[
T:L^p_{w_1}(\R^d)\times\cdots L^p_{w_m}(\R^d)\to L^p_w(\R^d),
\]
where
\[
w:=\prod_{j=1}^m w_j,\quad\frac{1}{p}=\sum_{j=1}^m\frac{1}{p_j},
\]
when $p_1,\ldots,p_m\in(1,\infty]$, $p<\infty$, and $\vec{w}\in A_{\vec{p}}$:
\[
[\vec{w}]_{\vec{p}}:=\sup_Q\Big(\avint_Q\!w(x)^p\,\mathrm{d}x\Big)^{\frac{1}{p}}\prod_{j=1}^m\Big(\avint_Q\!w_j(x)^{-p_j'}\,\mathrm{d}x\Big)^{\frac{1}{p_j'}}<\infty.
\]
Moreover, they showed that the multilinear Hardy--Littlewood maximal operator
\[
M(f_1,\ldots,f_m)(x):=\sup_Q\prod_{j=1}^m\Big(\avint_Q\!|f_j(x)|\,\mathrm{d}x\Big)\ind_Q(x)
\]
satisfies
\[
M:L^p_{w_1}(\R^d)\times\cdots\times L^p_{w_m}(\R^d)\to L^p_w(\R^d)
\]
precisely when $p_1,\ldots,p_m\in(1,\infty]$ and $w\in A_{\vec{p}}$.

\subsection{Main result} 
The goal of this paper is to fully establish a multilinear matrix weighted analogue of the Hunt--Muckenhoupt--Wheeden theorem, unifying and extending the theory for both the multilinear and matrix weighted settings.

First, we need to establish what we mean by extending a multilinear operator $T$ to $m$-tuples of vector-valued functions $f_1,\ldots,f_m$ with $f_j:\R^d\to \C^{n_j}$. To exemplify our approach, we first restrict to the case $m=2$, $n_1=n_2=2$. Given two functions $f,g:\R^d\to\C^2$, each have component functions $f=(f_1,f_2)$ and $g=(g_1,g_2)$. As $T$ is bilinear, we want our extension to consider all possible pairings of the component functions:
\[
T(f_1,g_1),\quad T(f_1,g_2),\quad T(f_2,g_1),\quad T(f_2,g_2).
\]
Arranging these functions in a matrix yields a mapping $\widetilde{T}(f,g):\R^d\to\C^{2\times 2}\cong\C^{4}$.

In the general case, this extension can be conveniently expressed through the use of the tensor product 
\[
\bigotimes_{j=1}^m\C^{n_j}\cong \C^n,\quad n:=\prod_{j=1}^m n_j.
\]
Suppose $T$ is an $m$-linear operator in $L^0(\R^d)$ defined on $m$-tuples of functions in $L^\infty_c(\R^d)$. Given vector-valued functions $\vec{f}=(f_1,\ldots,f_m)$ with $f_j\in L^\infty_c(\R^d;\C^{n_j})$, we define 
\[
\widetilde{T}(\vec{f}):\R^d\to\bigotimes_{j=1}^m\C^{n_j}
\]
by applying $T$ ``component-wise'' to $\bigotimes_{j=1}^m f_j:\R^d\to \bigotimes_{j=1}^m\C^{n_j}$. More precisely, denote the canonical basis of $\C^{n_j}$ by $(e_k)_{k=1}^{n_j}$ so that each $f_j$ can be written as
\[
f_j(x)=\sum_{k=1}^{n_j} f_{j,k}(x)e_k.
\]
Then our definition is
\begin{equation}
\label{eq:vector_valued_extension}
\widetilde{T}(\vec{f})(x):=\sum_{k_1=1}^{n_1}\cdots\sum_{k_m=1}^{n_m} T(f_{1,k_1},\ldots,f_{m,k_m})(x)\bigotimes_{j=1}^m e_{k_j}.
\end{equation}
A standard check shows that this definition does not depend on the choice of orthonormal basis. As a matter of fact, in this work we will replace the $\C^{n_j}$ by general $n_j$-dimensional Hilbert spaces $\mc{H}_j$, in which we work with this same extensions, even when no orthonormal basis is specified. 

When $T$ is the product map $T(h_1,\ldots,h_m)(x)=\prod_{j=1}^m h_j(x)$, our extension coincides with the tensor product map
\[
\widetilde{T}(\vec{f})(x)=\sum_{k_1=1}^{n_1}\cdots\sum_{k_m=1}^{n_m} \prod_{j=1}^mf_{j,k_j}(x)\bigotimes_{j=1}^m e_{k_j}=\bigotimes_{j=1}^m f_j(x).
\]
Moreover, when $m=1$, our extension reduces back to the component-wise extension \eqref{eq:introcomponentwiseextension}. 

We emphasize here that the proposed extension $\widetilde{T}$ of $T$ is the natural one obtained from the universal property of the tensor product.

Next, we need to establish what the appropriate multilinear Muckenhoupt condition is for matrix weights. Let $p_1,\ldots,p_m\in(1,\infty]$ and let $W_j:\R^d\to\C^{n_j\times n_j}$ be matrix weights for which $W_j^{-1}\in L^{p_j'}_{\text{loc}}(\R^d;\C^{n_j\times n_j})$ for $j=1,\ldots,m$. We set
\[
\mb{W}:=\bigotimes_{j=1}^m W_j,\quad\frac{1}{p}=\sum_{j=1}^m\frac{1}{p_j},
\]
where for a.e. $x\in\R^d$, the map $\bigotimes_{j=1}^m W_j(x)$ is the unique map satisfying the property that
\[
\Big(\bigotimes_{j=1}^m W_j(x)\Big)\Big(\bigotimes_{j=1}^m u_j\Big)=\bigotimes_{j=1}^m W_j(x)u_j
\]
for all $u_j\in\C^{n_j}$, $j=1,\ldots,m$. Writing $\vec{W}=(W_1,\ldots,W_m)$ and $\vec{p}=(p_1,\ldots,p_m)$, we say that $\vec{W}$ satisfies the $\vec{p}$ matrix Muckenhoupt condition $\vec{W}\in A_{\vec{p}}$ if for all cubes $Q$ the mapping
\[
T_Q(f_1,\ldots,f_m):=\ind_Q\bigotimes_{j=1}^m\avint_Q\!f_j(x)\,\mathrm{d}x
\]
satisfies 
\[
T_Q:L^{p_1}_{W_1}(\R^d;\C^{n_1})\times\cdots\times L^{p_m}_{W_m}(\R^d;\C^{n_m})\to L^p_{\mb{W}}\Big(\R^d;\bigotimes_{j=1}^m \C^{n_j}\Big),
\]
with
\[
[\vec{W}]_{\vec{p}}:=\sup_Q\|T_Q\|_{L^{p_1}_{W_1}(\R^d;\C^{n_1})\times\cdots\times L^{p_m}_{W_m}(\R^d;\C^{n_m})\to L^p_{\mb{W}}\Big(\R^d;\bigotimes_{j=1}^m \C^{n_j}\Big)}<\infty.
\]
Just like in the scalar case, extra difficulties arise in the multilinear setting due to the appearance of the exponents $p$ in the quasi-Banach range $\tfrac{1}{m}<p<1$. In this work, we develop a way to apply the John ellipsoid theorem to the now quasinorms
\[
\rho_{W,Q,p}(\mb{u}):=\Big(\avint_Q\!|\mb{W}(x)\mb{u}|^p\,\mathrm{d}x\Big)^{\frac{1}{p}}
\]
on $\bigotimes_{j=1}^m\C^{n_j}$. To facilitate this, we make use of Caratheodory's theorem on convex hulls. This result states that if $S$ is a set of points in $\R^n$, then its convex hull consists precisely of points of the form
\[
\sum_{k=1}^{n+1}\theta_ku_k,
\]
where $0\leq \theta_1,\ldots,\theta_k$, $\sum_{k=1}^{n+1}\theta_k=1$, and $u_1,\ldots,u_{n+1}\in S$. In Section~\ref{subsec:quasinorms} below we show how this allows us to obtain a reducing matrix $A_{\mathbf{W},Q,p}:\bigotimes_{j=1}^m\C^{n_j}\to \bigotimes_{j=1}^m\C^{n_j}$ for which
\[
K_p^{-2n}\rho_{\mathbf{W},Q,p}(\mb{u})\leq |A_{\mathbf{W},Q,p}\mb{u}|\leq n^{\frac{1}{2}}\rho_{\mathbf{W},Q,p}(\mb{u})
\]
for all $\mb{u}\in\bigotimes_{j=1}^m\C^{n_j}$, where $K_p=2^{(\frac{1}{p}-1)_+}$ is the quasi-metric constant of the Lebesgue space of exponent $p$. Using this allows us to show that our multilinear matrix $A_{\vec{p}}$ condition can be characterized through the Roudenko-type condition
\[
[\vec{W}]_{\vec{p}}\eqsim\sup_Q\Big(\avint_Q\prod_{j=1}^m\Big(\avint_Q\!\|W_j(x)W_j(y_j)^{-1}\|^{p_j'}\,\mathrm{d}y_j\Big)^{\frac{p}{p_j'}}\,\mathrm{d}x\Big)^{\frac{1}{p}},
\]
see Proposition~\ref{prop:equivalence_averages_operator} below, as well as through the reducing matrix condition
\[
[\vec{W}]_{\vec{p}}\eqsim\sup_Q\Big\|A_{W,Q,p}\Big(\bigotimes_{j=1}^m A_{W_j^{-1},Q,p_j'}\Big)\Big\|,
\]
see Proposition~\ref{prop:reducingmatrixavop} below. The proof of Proposition~\ref{prop:reducingmatrixavop} itself is achieved through an involved induction scheme relying on some elementary tensor algebra. While such techniques are standard in the field of differential geometry, they have not appeared so far in the field of weighted estimates for singular integrals. Thus, we explain them in detail.

We also consider a characterization in terms of the tensor product maximal operator. For locally integrably bounded mappings $F_j:\R^d\to\mc{K}(\C^{n_j})$, $j=1,\ldots,m$, we define $M^{\mc{K}}(F_1,\ldots,F_m)(x)$ as the smallest set in $\mc{K}\big(\bigotimes_{j=1}^m\C^{n_j}\big)$ containing
\[
\bigcup_{Q}\,\bigotimes_{j=1}^m\langle F_j\rangle_Q\ind_Q(x).
\]
Setting
\[
L^{\vec{p}}_{\vec{W}}(\R^d;\C^{\vec{n}}):=L^{p_1}_{W_1}(\R^d;\C^{n_1})\times\cdots\times L^{p_m}_{W_m}(\R^d;\C^{n_m}),
\]
our main theorem is as follows:
\begin{TheoremLetter}\label{thm:A}
Let $p_1,\ldots,p_m\in(1,\infty]$, $\tfrac{1}{p}:=\sum_{j=1}^m\tfrac{1}{p_j}>0$ and let $W_1,\ldots,W_m$ be matrix weights $W_j:\R^d\to\C^{n_j\times n_j}$ with $W_j^{-1}\in L^{p_j'}_{\text{loc}}(\R^d;\C^{n_j\times n_j})$ for $j=1,\ldots,m$. Then the following are equivalent:
\begin{enumerate}[(i)]
    \item\label{it:thmA1}  $\widetilde{T}:L^{\vec{p}}_{\vec{W}}(\R^d;\C^{\vec{n}})\to L^p_{\mb{W}}\Big(\R^d;\bigotimes_{j=1}^m \C^{n_j}\Big)$ for all $m$-linear Calder\'{o}n--Zygmund operators $T$;
    \item\label{it:thmA2} $M^{\mc{K}}:L^{\vec{p}}_{\vec{W}}(\R^d;\C^{\vec{n}})\to L^p_{\mb{W}}\Big(\R^d;\bigotimes_{j=1}^m \C^{n_j}\Big)$;
    \item\label{it:thmA3} $\vec{W}\in A_{\vec{p}}$.
\end{enumerate}
\end{TheoremLetter}
In order to establish this equivalence, we prove a general convex body domination principle for $m$-linear Calder\'on--Zygmund operators with kernels having a modulus of continuity satisfying the Dini condition, the precise definition of which we give in Section~\ref{sec:mczo}. We defer the proof of Theorem~\ref{thm:A} to Section~\ref{sec:proofs}.

\subsection{Quantitative bounds}

Following the Hunt--Muckenhoupt--Wheeden theorem, a natural question is how exactly the bounds of Calder\'on--Zygmund operators and the Hardy--Littlewood maximal operator in $L^p_w(\R^d)$ depend on the Muckenhoupt characteristic $[w]_p$.

The sharp bound for the Hardy-Littlewood maximal operator was found by Buckley in \cite{Bu93}, who showed that
\[
\|M\|_{L^p_w(\R^d)\to L^p_w(\R^d)}\lesssim_{d,p}[w]_p^{p'}\quad\text{for $1<p\leq\infty$.}
\]
As for matrix weights $W$, it was shown in \cite{Isralowitz_Moen_2019} that the Christ--Goldberg maximal operator $M_W$ satisfies its bound with the same exponent $p'$:
\[
\|M_W\|_{L^p(\R^d;\C^n)\to L^p(\R^d)}\lesssim_{d,n,p}[W]_p^{p'}.
\]
Reducing back to this result, it was shown in \cite{BC23} that the convex-set valued maximal operator $M^{\mc{K}}$ also satisfies
\[
\|M^{\mc{K}}\|_{L_W^p(\R^d;\C^n)\to L^p_W(\R^d;\C^n)}\lesssim_{d,n,p}[W]_p^{p'}.
\]

The problem for Calder\'on--Zygmund operators $T$ was a much harder problem, and became known as the $A_2$ conjecture. It was eventually settled by Hyt\"onen in \cite{Hy12}, who proved the validity of the conjectured bound
\begin{equation}\label{eq:A2intro}
\|T\|_{L^p_w(\R^d)\to L^p_w(\R^d)}\lesssim_{d,p,T}[w]_p^{\max\{p,p'\}}.
\end{equation}
The proof was simplified by Lerner in \cite{Le13a}, who showed that for each Calder\'on-Zygmund operator $T$ there is a constant $C_T$ such that for all $f\in L^\infty_c(\R^d)$ there is a \emph{$\tfrac{1}{2\cdot 3^d}$-sparse} collection of cubes $\mc{S}$ in $\R^d$ such that
\begin{equation}\label{eq:sparsedomintro}
|Tf(x)|\leq C_T\sum_{Q\in\mc{S}}\ind_Q(x)\avint_Q\!|f(x)|\,\mathrm{d}x=:C_TA_{\mc{S}}f
\end{equation}
for a.e. $x\in\R^d$. For $0<\eta<1$, a collection of cubes $\mc{S}$ is called $\eta$-sparse if it satisfies the packing condition
\[
\Big|\bigcup_{\substack{Q\in\mc{S}\\ Q\subsetneq Q_0}}Q\Big|\leq(1-\eta)|Q_0|
\]
for all $Q_0\in\mc{S}$. This reduced the proof of \eqref{eq:A2intro} to proving a straightforward of $A_{\mc{S}}f$, uniform in the sparse collection $\mc{S}$.

This technique was not readily adaptable to the case of matrix weights. Indeed, unlike in the scalar case, a pointwise bound like \eqref{eq:sparsedomintro} does not guarantee a bound of $|W(x)\widetilde{T}f(x)|$ for matrix weights $W$. This issue was solved by Nazarov, Petermichl, Treil, and Volberg in the seminal work \cite{NPTV17}. They introduced the notion of \emph{convex body domination}, where they showed the existence of a constant $C_T>0$ such that for each $f\in L^\infty_c(\R^d;\C^n)$ there exists a sparse collection $\mc{S}$ for which
\[
\widetilde{T}f(x)\in C_T\sum_{Q\in\mc{S}}\llangle f\rrangle_Q\ind_Q(x)=:C_TA_{\mc{S}}^{\mc{K}}f(x).
\]
Here the sum is treated as a Minkowski sum, scalar multiples are considered pointwise, and the double angle bracket notation denotes the Aumann integral
\[
\llangle f\rrangle_Q:=\langle\mc{K}(f)\rangle_Q=\Big\{\avint_Q\!h(x)f(x)\,\mathrm{d}x:h\in L^\infty(\R^d),\, \|h\|_{L^\infty(\R^d)}\leq 1\Big\},
\]
where $\mc{K}(f)(x)$ denotes the smallest set in $\mc{K}(\C^n)$ containing $f(x)$. Unlike what happened for the maximal operator, this does not recover the scalar bound \eqref{eq:A2intro}. Instead, for the case $p=2$, they proved that
\[
\|\widetilde{T}\|_{L^2_W(\R^d)\to L^2_W(\R^d)}\lesssim_{d,n}C_T[W]_2^3.
\]
Rather surprisingly, this bound was shown to be sharp for the Hilbert transform $T=H$ by Domelevo, Petermichl, Treil, and Volberg in \cite{DPTV24}. For general exponents, it was shown by Cruz-Uribe, Isralowitz, and Moen in \cite{CIM18} that
\begin{equation}\label{eq:matrixa2sparse}
\|\widetilde{T}\|_{L^p_W(\R^d)\to L^p_W(\R^d)}\lesssim_{d,n,p}C_T[W]_p^{p+p'-1}\quad\text{if $1<p<\infty$.}
\end{equation}
In bounding the convex body operator $A_{\mc{S}}^{\mc{K}}$ in $L^p_W(\R^d;\mc{K}(\C^n))$, a main ingredient is the sharp reverse H\"older inequality of \cite{HP13} for scalar weights $v$ satisfying the Fujii--Wilson condition
\[
[v]_{\text{FW}}:=\sup_Q\frac{1}{v(Q)}\int_Q\!M^{\mc{D}(Q)}(v)(x)\,\mathrm{d}x<\infty,
\]
where $M^{\mc{D}(Q)}$ is the Hardy--Littlewood maximal operator taken over the dyadic subgrid of the cube $Q$. By restricting the averaging operators $T_Q$ to one-dimensional subspaces of $L^p_W(\R^d;\C^n)$, one finds that if $W\in A_p$, then also $|Wu|\in A_p$ for all $u\in\C^n\backslash\{0\}$, with
\[
\sup_{u\in\C^n\backslash\{0\}}[|Wu|]_p\leq[W]_p.
\]
Thus, as $W\in A_p$ precisely if $W^{-1}\in A_{p'}$, the Fujii--Wilson conditions
\[
[W]_{\text{FW}_p}:=\sup_{u\in\C^n\backslash\{0\}}[|Wu|^p]^{\frac{1}{p}}_{\text{FW}}\lesssim_d [W]_p,\quad [W^{-1}]_{\text{FW}_{p'}}:=\sup_{u\in\C^n\backslash\{0\}}[|Wu|^{p'}]^{\frac{1}{p'}}_{\text{FW}}\lesssim_d [W]_p
\]
hold. This yields a uniform reverse H\"older inequality which is used in the proof of \eqref{eq:matrixa2sparse} to deduce that
\[
\|\widetilde{T}\|_{L^p_W(\R^d)\to L^p_W(\R^d)}\lesssim_{d,n,p}C_T[W]^{\frac{p}{p'}}_{\text{FW}_p}[W^{-1}]^{\frac{p'}{p}}_{\text{FW}_{p'}}[W]_p\quad\text{if $1<p<\infty$.}
\]
However, in the multilinear case, these uniform reverse H\"older conditions are not readily available. Indeed, while the restriction technique to one dimensional subspaces still works, this does not yield the required scalar bounds (see Remark~\ref{rem:first_scalar_ap} below). Moreover, the notion of duality yielding $W^{-1}\in A_{p'}$ for $W\in A_p$ is now completely absent.

In the scalar case, the multilinear Muckenhoupt condition is equivalent to a set of individual matrix linear Muckenhoupt conditions on each of the weights $w_j$, $j=1,\ldots,m$ and $w$, see \cite{LOPTT09, LMO20, Ni20}. It turns out that a similar result continues being true in the matrix setting. However, the noncommutative nature of matrix multiplication imposes additional technical difficulties, which we are able to overcome through a careful use of H\"{o}lder's inequality. We refer to Theorem~\ref{thm:multilinear_muckenhoupt_through_linear_muckenhoupt} below for details.

Reverse H\"{o}ler inequalities in the scalar case follow immediately from the individual Muckenhoupt conditions. In the matrix case the individual matrix Muckenhoupt conditions one obtains are intrisically of a ``two-exponent nature'' but at the same time different from the ones that have appeared in other places in the literature such as \cite{Isralowitz_Moen_2019, KNV24}. In order to deduce scalar Muckenhoupt conditions from them we need to test a new type of ``exponent-shifted'' averaging operators on one dimensional subspaces, see Proposition~\ref{prop:muck_matrix_to_scalar} below. Using this, we find that if $\vec{W}\in A_{\vec{p}}$, then
\begin{equation*}
    [\mathbf{W}]_{\mathrm{FW}_p}:=\sup_{\mb{u}\in\Big(\bigotimes_{j=1}^{m}\C^{n_j}\Big)\setminus\{0\}}[|\mb{W}\mb{u}|^{p}]_{\text{FW}}^{\frac{1}{p}}\lesssim_{d,m,\vec{n},p}[\vec{W}]_{\vec{p}}\quad\text{ if }p<\infty,
\end{equation*}
and
\begin{equation*}
        [W_j^{-1}]_{\mathrm{FW}_{p_j'}}:=\sup_{u_j\in\C^{n_j}\setminus\{0\}}[| W_j^{-1}u_j|^{p_j'}]_{\mathrm{FW}}^{\frac{1}{p_j'}}\lesssim_{d,m,\vec{n},p_j}[\vec{W}]_{\vec{p}}\quad\text{if $p_j>1$ for }j=1,\ldots,m,
\end{equation*}
see Corollary~\ref{cor:needed_reverse_Holder} coupled with Proposition~\ref{prop:equivalence_averages_operator} below. Defining
\[
A_{\mc{S}}^{\mc{K}}(\vec{F})(x):=\sum_{Q\in\mc{S}}\mc{K}\Big(\bigotimes_{j=1}^m\langle F_j\rangle_Q\Big)\ind_Q(x),
\]
where for $S\subseteq \bigotimes_{j=1}^m\C^{n_j}$ we let $\mc{K}(S)$ denote the smallest set in $\mc{K}\big(\bigotimes_{j=1}^m\C^{n_j}\big)$ containing $S$, our main quantitative result is:
\begin{TheoremLetter}\label{thm:B}
Let $p_1,\ldots,p_m\in(1,\infty]$, $\tfrac{1}{p}:=\sum_{j=1}^m\tfrac{1}{p_j}$, let $W_1,\ldots,W_m$ be matrix weights $W_j:\R^d\to\C^{n_j\times n_j}$, and let $0<\eta<1$. Then the following assertions hold:
\begin{enumerate}[(a)]
    \item\label{it:thmB1} If $p<\infty$ and $\vec{W}\in A_{\vec{p}}$, then uniformly for all $\eta$-sparse collections $\mc{S}$, \begin{align*}
    \|A_{\mc{S}}^{\mc{K}}\|_{L^{\vec{p}}_{\vec{W}}(\R^d;\C^{\vec{n}})\to L^p_{\mb{W}}(\R^d;\bigotimes_{j=1}^m\C^{n_j})}&\lesssim_{d,m,\vec{n},\vec{p},\eta}[\vec{W}]_{\vec{p}}[\mb{W}]_{\mathrm{FW}_p}^{(p-1)^{+}}\prod_{j=1}^m[W_j^{-1}]_{\mathrm{FW}_{p_j'}}^{\frac{p_j'}{p_j}}\\
    &\lesssim_{d,m,\vec{n},p}[\vec{W}]_{\vec{p}}^{\max(1,p)+\sum_{j=1}^m\frac{p_j'}{p_j}};
    \end{align*}
    \item\label{it:thmB2} If $\vec{W}\in A_{\vec{p}}$, then
    \begin{align*}\displaystyle\|M^{\mc{K}}\|_{L^{\vec{p}}_{\vec{W}}(\R^d;\C^{\vec{n}})\to L^p_{\mb{W}}(\R^d;\bigotimes_{j=1}^m\C^{n_j})}
    &\lesssim_{d,m,\vec{n},\vec{p}}[\vec{W}]_{\vec{p}}\prod_{j=1}^m[W_j^{-1}]_{\mathrm{FW}_{p_j'}}^{\frac{p_j'}{p_j}}\\
    &\lesssim_{d,m,\vec{n}}[\vec{W}]_{\vec{p}}^{1+\sum_{j=1}^m\frac{p_j'}{p_j}}.
    \end{align*}
\end{enumerate}
\end{TheoremLetter}
Part \ref{it:thmB1} is proven as Theorem~\ref{thm:multilinearconvexbodyweightedbounds} below, and part \ref{it:thmB2} as Theorem~\ref{thm:strongboundsmaxop} below.

Interestingly, when $m=1$, the bound in \ref{it:thmB2} recovers Buckley's sharp bound of the scalar result. However, in the multilinear case, the sharp bound for $M$ was shown in \cite{LMS14} (see \cite{Ni19} for the case of infinite exponents) to be
\[
\|M\|_{L^{\vec{p}}_{\vec{w}}(\R^d)\to L^p_w(\R^d)}\lesssim_{d,m,\vec{p}}[\vec{w}]_{\vec{p}}^{\max\{p_1',\ldots,p_m'\}}.
\]
This exponent is smaller than the one we obtain in \ref{it:thmB2}. We leave it as an open problem whether the bound in the matrix case can be improved or not.

We apply Theorem~\ref{thm:B} to prove implication \ref{it:thmA3}$\Rightarrow$\ref{it:thmA1} of Theorem~\ref{thm:A} with the following convex body domination result. Here, we write $A_{\mc{S}}^{\mc{K}}(\vec{f}):=A_{\mc{S}}^{\mc{K}}(\mc{K}(f_1),\ldots,\mc{K}(f_m))$.
\begin{TheoremLetter}\label{thm:C}
    Let $T$ be an $m$-linear Calder\'{o}n--Zygmund operator with a kernel having a modulus of continuity that satisfies the Dini condition. Then there is a constant $C_T>0$ (depending on $d$, $m$, $\vec{n}$, and $T$) such that for any cube $Q_0$ and any $m$-tuple of functions $\vec{f}\in L^{\vec{1}}_{\mathrm{c}}(\R^d;\C^{\vec{n}})$ there exists a $\tfrac{1}{2\cdot 3^{d}}$-sparse family $\mc{S}$ of cubes in $\R^d$ contained in $3Q_0$ such that
    \begin{equation*}
        T(\vec{f})(x)\in C_TA^{\mc{K}}_{\mc{S}}(\vec{f})(x)\quad\text{for a.e. }x\in Q_0,
    \end{equation*}
    where the constant $C_T$ depends only on $d$, $m$, $\vec{n}$ and $T$.

    In particular, if $p_1,\ldots,p_m\in(1,\infty]$ satisfy $\tfrac{1}{p}:=\sum_{j=1}^m\tfrac{1}{p_j}>0$, and $W_1,\ldots,W_m$ are matrix weights $W_j:\R^d\to\C^{n_j\times n_j}$ with $\vec{W}\in A_{\vec{p}}$, then
    \begin{align*}
\|\widetilde{T}\|_{L^{\vec{p}}_{\vec{W}}(\R^d;\C^{\vec{n}})\to L^p_{\mb{W}}(\R^d;\bigotimes_{j=1}^m \C^{n_j})}
&\lesssim_{d,m,\vec{n},\vec{p}} C_T[\vec{W}]_{\vec{p}}[\mb{W}]_{\mathrm{FW}_p}^{(p-1)^{+}}\prod_{j=1}^m[W_j^{-1}]_{\mathrm{FW}_{p_j'}}^{\frac{p_j'}{p_j}}\\
&\lesssim_{d,m,\vec{n},\vec{p}}C_T[\vec{W}]_{\vec{p}}^{\max(1,p)+\sum_{j=1}^m\frac{p_j'}{p_j}}.
\end{align*}
\end{TheoremLetter}
This convex body domination result is proven in Theorem~\ref{thm:sparse_domination_multilinear_CZ} below, and the quantitative bounds in Corollary~\ref{cor:czoquantitativebound} below.

While this result is certainly known, we note that even in the scalar case, sparse domination for $m$-linear Calder\'on-Zygmund operators with kernels having a modulus of continuity that satisfies the Dini condition has not been explicitly written in the literature. In \cite{CR16, LN15}, the less general $\log$-Dini condition is assumed, and in \cite{Damian2018}, only the case $m=2$ is treated. For a complete exposition, we provide details of the general case $m>1$ in this work based on \cite{Damian2018, Li18}. 

\subsection{Organization} This paper is organized as follows:
\begin{itemize}
    \item In Section~\ref{sec:preliminaries} we lay the groundwork for the techniques and results of this paper. On the one hand, we introduce our main notation and collect together some notions and facts on directional quasi-Banach functions spaces, tensor products of Hilbert spaces, equivalent norms on spaces of linear operators and sparse families. Of particular importance is a novel definition of a product of directional Banach functions spaces, each of which is obtained from a Banach function space through a matrix weight. On the other hand, inspired by techniques in convex combinatorics we introduce a completely new notion of reducing operators associated to quasinorms, which forms one of the cornerstones of the paper.

    \item In Section~\ref{sec:averaging_operators} we characterize the boundedness of averaging operators from a cartesian product of directional Banach function spaces into a directional quasi-Banach function space in terms of reducing operators. The new reducing operators for quasinorms are the centerpiece of this characterization. On the technical level, inspired by techniques in differential geometry we use an involved but at the same time rather elementary induction scheme involving some tensor algebra.

    \item In Section~\ref{sec:multilinear_muckenhoupt} we develop a theory of multilinear Muckenhoupt characteristics for matrix weights. Thanks to the new notion of reducing operators for quasinorms, we are able to do so directly in the full generality of three sets of parameters and fully recover the results that were already known in the scalar multilinear case. The main result of this section is a set of reverse H\"{o}lder estimates for each weight individually as well as their tensor product. One of the main tools of the proof is the use of novel ``exponent-shifted'' averaging operators.

    \item In Section~\ref{sec:tensor_maximal} we treat lower and upper bounds for a tensor product analog of the convex body maximal operator. First, we show the equivalence of matrix weighted weak-type estimates for this operator and the multilinear Muckenhoupt condition. Then, we show that in the presence of the multilinear Muckenhoupt condition, the tensor product convex body maximal operator admits matrix weighted, strong type bounds. The proof relies on the reverse H\"{o}lder estimates of Section~\ref{sec:multilinear_muckenhoupt}. Finally, we prove a pointwise sparse domintation result for the tensor product convex body maximal operator.

    \item In Section~\ref{sec:mczo} we consider matrix weighted bounds for multilinear Calder\'{o}n--Zygmund operators. After a quick overview of basic facts about such operators we prove a pointwise sparse domination result for them in terms of a tensor product convex body sparse operator. This is achieved through bootstrapping the sparse domination algorithm of the scalar case. Using our sparse domination and the reverse H\"{o}lder estimates from Section~\ref{sec:multilinear_muckenhoupt}, we prove matrix weighted upper bounds for them in the presence of the multilinear Muckenhoupt condition. Finally, we introduce a novel notion of directionally nondegenerate multilinear Calder\'{o}n--Zygmund operators -- which is new even in the scalar case -- and show that matrix weighted bounds for them imply the multilinear Muckenhoupt condition. As an illustration of our notion of directional nondegeneracy, we study the concrete example of multilinear Riesz transforms.
    
    \item In Section~\ref{sec:proofs} we prove theorem~\ref{thm:A} by collecting together all pieces from the preceding sections.

    \item Finally, for the convenience of the reader, in the Appendix we give detailed proofs of several facts about multilinear Calder\'{o}n--Zygmund operators used in the sparse domination of Section~\ref{sec:mczo}.
\end{itemize}

\section{preliminaries}
\label{sec:preliminaries}
\subsection{Notation} Throughout this work, we let $d$ be a positive integer denoting the dimension of the domain, and let $m$ be a positive integer denoting the multilinearity. We let $\mc{H}=(\mc{H}_1,\ldots,\mc{H}_m)$ be finite dimensional Hilbert spaces that are all defined over the field $\F$, where $\F$ denotes either $\R$ or $\C$. We fix their dimensions $n_j:=\dim\mc{H}_j$, $j=1,\ldots,m$ and set
\[
n:=\prod_{j=1}^m n_j.
\]
Choosing an orthonormal basis of $\mc{H}_j$, we may identify $\mc{H}_j\cong\F^{n_j}$. 

We will use the following definitions and notational conventions:
\begin{itemize}
    \item The components of $\vec{p}$ will always be denoted by $\vec{p}=(p_1,\ldots,p_m)$, and, per convention,
    \[
    \frac{1}{p}:=\sum_{j=1}^m\frac{1}{p_j}.
    \]
    Algebraic operations and inequalities of $\vec{p}$ are considered in the Schur sense, i.e., componentwise. For example, we write $\tfrac{1}{\vec{p}}-\tfrac{1}{\vec{q}}=(\tfrac{1}{p_1}-\frac{1}{q_1},\ldots,\tfrac{1}{p_m}-\frac{1}{q_m})$, or $\vec{p}>\vec{q}$ if $p_j>q_j$ for all $j=1,\ldots,m$.
    \item As a convention, we write
    \[
    \mc{H}:=\bigotimes_{j=1}^m\mc{H}_j,
    \]
    which is a Hilbert space with respect to the unique sesquilinear form on $\mc{H}$ satisfying
    \[
    \Big\langle \bigotimes_{j=1}^m u_j,\bigotimes_{j=1}^m v_j\Big\rangle_{\mc{H}}=\prod_{j=1}^m\langle u_j,v_j\rangle_{\mc{H}_j}
    \]
    for all $\vec{u},\vec{v}\in\vec{\mc{H}}$. Note that $\dim\mc{H}=n$.
    \item The space of bounded linear operators on a Hilbert space $\mc{V}$ is denoted by $\mc{L}(\mc{V})$ and is equipped with the operator norm
    \[
    \|L\|_{\mc{V}\to\mc{V}}:=\sup_{\|u\|_{\mc{V}}\leq 1}\|Lu\|_{\mc{V}}.
    \]
    If $\dim\mc{V}=k\in\N$ and an orthonormal basis for $\mc{V}$ is chosen, we may identify $\mc{L}(\mc{V})$ with the space $\F^{k\times k}$ of $k\times k$ matrices.
    \item An $m$-tuple of mappings $\vec{W}=(W_1,\ldots,W_m)$, $W_j:\R^d\to\mc{L}(\mc{H}_j)$ will be called an $m$-tuple of matrix weights if for a.e. $x\in\R^d$ the operator $W_j(x)$ is self-adjoint and positive for all $j\in\{1,\ldots,m\}$. When identified with a matrix, this means that $W_j(x)$ is Hermitian and positive definite. Moreover, as a convention, we set
    \[
    \mb{W}:=\bigotimes_{j=1}^m W_j:\R^d\to\mc{L}(\mc{H}).
    \]
    \item For a finite dimensional Hilbert space $\mc{V}$ over $\R$, respectively over $\C$, we let $\mc{K}(\mc{V})$ denote the non-empty, compact, real, respectively complex symmetric, convex subsets of $\mc{V}$. Given $0<p\leq \infty$ and a matrix weight $W:\R^d\to\mc{L}(\mc{V})$, the space $L^p_W(\R^d;\mc{K}(\mc{V}))$ consists of the measurable mappings $F:\R^d\to\mc{K}(\mc{V})$ for which $\|W(x)F(x)\|_{\mc{V}}:=\sup_{u\in F(x)}\|W(x)u\|_{\mc{V}}$ lies in $L^p(\R^d)$, with
    \[
    \|F\|_{L^p_W(\R^d;\mc{K}(\mc{H}))}:=\big\|\|WF\|_{\mc{V}}\big\|_{L^p(\R^d)}.
    \]
    \item For an $m$-tuple of exponents $\vec{p}$ and an $m$-tuple of matrix weights $\vec{W}$ we write
    \[
    L^{\vec{p}}_{\vec{W}}(\R^d;\vec{\mc{H}}):= L^{p_1}_{W_1}(\R^d;\mc{H}_1)\times\ldots\times L^{p_m}_{W_m}(\R^d;\mc{H}_m).
    \]
    Thus, we write $\vec{f}\in L^{\vec{p}}_{\vec{W}}(\R^d;\vec{\mc{H}})$ to mean that $f=(f_1,\ldots,f_m)$ with $f_j\in L^{p_j}_{W_j}(\R^d;\mc{H}_j)$ for $j=1,\ldots,m$. We analogously define $L^{\vec{p}}_{\vec{W}}(\R^d;\mc{K}(\vec{\mc{H}}))$, $L^0(\R^d;\vec{\mc{H}})$, $L^{\vec{\infty}}_c(\R^d;\vec{\mc{H}})$, ...
    \item By a cube $Q\subseteq\R^d$, we mean a cube whose sides are parallel to the coordinate axes.
\end{itemize}
We write $A\lesssim_{a_1,a_2,\ldots}B$ to mean that there is a constant $C_{a_1,a_2,\ldots,}$ depending on the parameters $a_1,a_2,\ldots$ for which $A\leq C_{a_1,a_2,\ldots}B$. We define $A\gtrsim_{a_1,a_2,\ldots} B$ analogously, and we write $A\eqsim_{a_1,a_2,\ldots} B$ if both $A\lesssim_{a_1,a_2,\ldots}B$ and $A\gtrsim_{a_1,a_2,\ldots}B$ hold.

\subsection{Directional quasi-Banach function spaces}
In this section we describe the directional quasi-Banach function spaces as defined in \cite{Ni24b}. Let $(\Omega,\mu)$ be a $\sigma$-finite measure space and let $\mc{V}$ be a finite dimensional Hilbert space. We say that $\mb{X}$ is a $\mc{V}$-directional quasi-Banach function space over $\Omega$ if it is a complete quasi-normed subspace of $L^0(\Omega;\mc{V})$ and satisfies the following properties:
\begin{itemize}
    \item \emph{The directional ideal property:} For all $f\in\mb{X}$ and $g\in L^0(\Omega;\mc{V})$ satisfying $g(x)\in \mc{K}(f)(x)$ for a.e. $x\in\Omega$, we have $g\in\mb{X}$ with $\|g\|_{\mb{X}}\leq\|f\|_{\mb{X}}$;
    \item\emph{Non-degeneracy:} If $g\in L^0(\Omega;\mc{V})$ satisfies $\int_\Omega\!\langle f,g\rangle_{\mc{V}}\,\mathrm{d}\mu=0$ for all $f\in\mb{X}$, then $g=0$.
\end{itemize}
We let $K_{\mb{X}}\geq 1$ denote the smallest constant for which
\[
\|f+g\|_{\mb{X}}\leq K_{\mb{X}}(\|f\|_{\mb{X}}+\|g\|_{\mb{X}})
\]
for all $f,g\in\mb{X}$. When $K_{\mb{X}}=1$, we call $\mb{X}$ a $\mc{V}$-directional Banach function space over $\Omega$.

We define $\mb{X}'$ as those $g\in L^0(\Omega;\mc{V})$ for which
\[
\|g\|_{\mb{X}'}:=\sup_{\|f\|_{\mb{X}}=1}\int_{\Omega}\!|\langle f,g\rangle_{\mc{V}}|\,\mathrm{d}\mu<\infty.
\]
As in \cite[Proposition~3.2]{Ni24b}, we have $g\in\mb{X}'$ if and only if there is a $C\geq 0$ such that for all $f\in\mb{X}$ we have
\begin{equation}
\label{eq:Koethe_dual_directional}
\Big|\int_{\Omega}\!\langle f,g\rangle_{\mc{V}}\,\mathrm{d}\mu\Big|\leq C\|f\|_{\mb{X}}.
\end{equation}
In this case, the optimal $C$ satisfies $\|g\|_{\mb{X}'}=C$.

We say that $\mb{X}$ satisfies the \emph{Fatou property}, if for any sequence $(f_k)_{k\geq 1}$ in $\mb{X}$ with $f_k\to f$ a.e. in $\mc{V}$, and $\liminf_{k\to\infty}\|f_k\|_{\mb{X}}<\infty$, we have $f\in\mb{X}$ with
\[
\|f\|_{\mb{X}}\leq\liminf_{k\to\infty}\|f_k\|_{\mb{X}}.
\]

If $\mc{V}=\F$, then we simply call a $\mc{V}$-directional quasi-Banach function space over $\Omega$ a \emph{quasi-Banach function space} over $\Omega$, and a Banach function space over $\Omega$ if the triangle inequality is satisfied. We refer the reader to \cite{LN23b} for a thorough treatment of quasi-Banach function spaces.

Let $\mb{X}$ be a $\mc{V}$-directional quasi-Banach function space. For each $v\in\mc{V}$ we define the space $\mb{X}_v$ consisting of those $h\in L^0(\Omega)$ for which $hv\in\mb{X}$, with $\|h\|_{\mb{X}_v}:=\|hv\|_{\mb{X}}$. We say that $\mb{X}$ has the \emph{component-wise saturation property} if for an orthonormal basis $(e_k)_{k=1}^n$ of $\mc{V}$, the spaces $\mb{X}_{e_k}$ for $k=1,\ldots,n$ are  quasi-Banach function spaces.

A measurable mapping $W:\Omega\to\mc{L}(\mc{V})$ is called a \emph{matrix weight} if $W(x)$ is self-adjoint and positive for a.e. $x\in\Omega$. Given a quasi-Banach function space $X$ over $\Omega$ and a matrix weight $W:\Omega\to\mc{L}(\mc{V})$, we define $X_W$ as the space of functions $f\in L^0(\Omega;\mc{V})$ for which $\|Wf\|_{\mc{V}}\in X$, and set
\[
\|f\|_{X_W}:=\big\|\|Wf\|_{\mc{V}}\big\|_X.
\]
Then $X_W$ is an $\mc{V}$-directional quasi-Banach function space over $\Omega$ with the component-wise saturation property satisfying
\[
(X_W)'=(X')_{W^{-1}},
\]
see \cite[Section~3.2]{Ni24b}. Observe that if $X$ satisfies the Fatou property, then so does $X_{W}$.

\medskip

\begin{definition}
Let $\vec{X}=(X_1,\ldots,X_m)$ be $m$ Banach function spaces over $\Omega$. Their product $X:=\prod_{j=1}^m X_j$ is defined as the quasi-Banach function space over $\Omega$ consisting of those $f\in L^0(\Omega)$ for which there exist $0\leq f_j\in X_j$ such that $|f|\leq\prod_{j=1}^m f_j$, with
\begin{equation}
\label{eq:product_quasi_norm}
\|f\|_{X}:=\inf\Big\{\prod_{j=1}^m\|f_j\|_{X_j}:|f|\leq \prod_{j=1}^m f_j,\,0\leq f_j\in X_j\Big\}.
\end{equation}
We refer to \cite{Sc10} for a thorough treatment of the properties of such product spaces.
\end{definition}

We would like to define the product of a finite number of directional Banach function spaces as well. Although we have not been able to do so in full generality, we give a definition in the particular case of matrix weighted Banach function spaces, including matrix weighted Lebesgue spaces.

\begin{definition}
    Let $\vec{X}=(X_1,\ldots,X_m)$ be $m$ Banach functions spaces over $\Omega$ and let $\vec{W}$ be a $m$-tuple of matrix weights. For each $j=1,\ldots,m$, consider the $\mc{H}_{j}$-directional Banach function space $\mathbf{X}_{j}:=(X_{j})_{W_{j}}$. The product $\mathbf{X}:=\prod_{j=1}^{m}\mathbf{X}_j$ of $\mathbf{X}_1,\ldots,\mathbf{X}_{m}$  is defined to be
    \begin{equation}
        \label{eq:product_directional_weighted}
        \mathbf{X}=X_{\mathbf{W}},
    \end{equation}
    where $X:=\prod_{j=1}^{m}X_{j}$ and, we recall that $\mb{W}=\bigotimes_{j=1}^m W_j$. Observe that $\mathbf{X}$ thus defined is an $\mc{H}$-directional quasi-Banach function space over $\Omega$.
\end{definition}

In the special case that there is $\vec{p}\in[1,\infty]^{m}$ with $X_j=L^{p_j}(\Omega)$ for each $j=1,\ldots,m$, we have $X=L^{p}(\Omega)$, where we recall that $\frac{1}{p}=\sum_{j=1}^{m}\frac{1}{p_j}$. Hence, in this case,
\begin{equation*}
    \prod_{j=1}^{m}L^{p_j}_{W_j}(\Omega;\mc{H}_j)=L^{p}_{\mathbf{W}}(\Omega;\mc{H}).
\end{equation*}

\subsection{Linear operators on a tensor product of Hilbert spaces}

Let $\vec{\mc{H}}$ be a $m$-tuple of finite dimensional Hilbert spaces. We recall that $\mc{H}=\bigotimes_{j=1}^{m}\mc{H}_j$ is also a finite dimensional Hilbert space. Given now $A_j\in\mc{L}(\mc{H}_j)$ for each $j=1,\ldots,m$, there is a unique operator $A$ satisfying
\begin{equation*}
    A\Big(\bigotimes_{j=1}^{m}u_j\Big)=\bigotimes_{j=1}^{m}A_ju_j
\end{equation*}
for all $\vec{u}\in\vec{\mc{H}}$. We denote this operator by $\displaystyle\bigotimes_{j=1}^{m}A_j:=A$. In this case one can directly verify that
\begin{equation*}
    \Big(\bigotimes_{j=1}^{m} A_j\Big)\Big(\bigotimes_{j=1}^{m} B_j\Big)=\bigotimes_{j=1}^{m}(A_jB_j)
\end{equation*}
for all $B_j\in\mc{L}(\mc{H}_j)$, $j=1,\ldots,m$.

Notice that
\begin{equation*}
    A^{*}=\bigotimes_{j=1}^{m}A_j^{*}.
\end{equation*}
In particular, if $A_1,\ldots,A_m$ are Hermitian, then $A$ is also Hermitian. Moreover, using for example the spectral theorem one readily deduces that if $A_1,\ldots,A_m$ are all positive definite, then $A$ is also positive definite.

The following lemma is a standard result, but we include its proof for the reader's convenience.

\begin{lemma}
    Let $A_j\in\mc{L}(\mc{H}_j)$ for each $j=1,\ldots,m$. Set $A:=\displaystyle\bigotimes_{j=1}^{m}A_j$. Then, we have
    \begin{equation*}
        \Vert A\Vert_{\mc{H}\to\mc{H}}=\prod_{j=1}^{m}\Vert A_j\Vert_{\mc{H}_j\to\mc{H}_j}.
    \end{equation*}
\end{lemma}

\begin{proof}
    If all $A_1,\ldots,A_m$ are all positive semidefinite, then the claim follows immediately through an application of the spectral theorem.

    For the general case, observe first that
    \begin{equation*}
        A^\ast A=\bigotimes_{j=1}^{m}(A_j^\ast A_j).
    \end{equation*}
    It follows from the special case that
    \begin{align*}
        \Vert A\Vert_{\mc{H}\to\mc{H}}^2=\Vert A^{*}A\Vert_{\mc{H}\to\mc{H}}=\prod_{j=1}^{m}\Vert A_j^{*}A_j\Vert_{\mc{H}_j\to\mc{H}_j}=\prod_{j=1}^{m}\Vert A_j\Vert_{\mc{H}_j\to\mc{H}_j}^2,
    \end{align*}
    concluding the proof.
\end{proof}

\subsection{Reducing operators for quasinorms}\label{subsec:quasinorms}
The construction of reducing operators in the case of norms is a standard application of the John ellipsoid theorem. However, since in the multilinear setting we are working with Lebesgue spaces with exponent $p<1$, we need to consider the setting of quasinorms as well.

\begin{proposition}\label{prop:reducingmatrixquasinorm}
Let $\mc{V}$ be an $n$-dimensional Hilbert space, and let $q:\mc{V}\to[0,\infty)$ be a lower semicontinuous quasinorm with quasi-triangle inequality constant $K$. Then there is a self-adjoint positive $A\in\mc{L}(\mc{V})$ for which
\[
    K^{-2n}q(u)\leq \|Au\|_{\mc{V}}\leq n^{\frac{1}{2}}q(u),
\]
for all $u\in\mc{V}$.
\end{proposition}
When the underlying field is $\R$, this power $2n$ can be replaced by $n$. The main ingredient for the proof other than the John ellipsoid theorem is Carath\'eodory's theorem on convex hulls. It states that if $S\subseteq\R^n$ is a set, then for any $u$ belonging to the convex hull of $S$ there are $\theta_1,\ldots,\theta_{n+1}\geq 0$ with $\sum_{k=1}^{n+1}\theta_k=1$, and $u_1,\ldots,u_{n+1}\in S$ for which
\[
u=\sum_{k=1}^{n+1}\theta_ku_k.
\]
The fact that the amount of elements in this convex sum has a fixed dimensional upper bound makes the result possible. When the underlying field is $\C$, we can use the canonical identification $\C\cong\R^2$ to obtain this result with $n$ replaced by $2n$.
\begin{proof}[Proof of Proposition~\ref{prop:reducingmatrixquasinorm}] We define
\begin{equation*}
    L:=\text{conv}(\{u\in\mc{V}:q(u)\leq 1\})\subseteq\mc{V}.
\end{equation*}
We claim that this is a compact (complex) symmetric convex set containing $0$ in its interior. By Carath\'eodory's theorem on convex hulls, the convex hull of a compact set is compact, so it suffices to show that the set of $u\in\mc{V}$ with $q(u)\leq 1$ has these properties. To this end, note that there are $c,C>0$ for which
\[
c\|u\|_{\mc{V}}\leq q(u)\leq C\|u\|_{\mc{V}}
\]
for all $u\in\mc{V}$. Boundedness then follows from the first inequality, whereas the containment of ball centered at $0$ follows from the second. For a topological argument proving the existence of these constants we refer the reader to \cite[Theorem 2.5]{Cabello_S_nchez_2021}. Finally, to see that the set is closed, let $(u_k)_{k\ge 1}$ be a sequence for which $q(u_k)\leq 1$ and $u_k\to u$ in $\mc{V}$. Then, by lower semicontinuity of $q$, we have
\[
q(u)\leq\liminf_{k\to\infty}q(u_k)\leq 1,
\]
as desired.
 
Now, we may apply the John ellipsoid theorem to $L$ to find a self-adjoint positive definite mapping $A\in \mc{L}(\mc{V})$ satisfying
\begin{equation}\label{eq:johnellipsoidquasiinclusion}
    A^{-1}\overline{B}_{\mc{V}}\subseteq L\subseteq n^{\frac{1}{2}}A^{-1}\overline{B}_{\mc{V}}.
\end{equation}
By Carath\'eodory's theorem on convex hulls, for each $u\in L$, there are $\theta_1,\ldots,\theta_{2n+1}\geq 0$ with $\sum_{k=1}^{2n+1}\theta_k=1$ and $u_1,\ldots,u_{2n+1}\in\mc{V}$ with $q(u_k)\leq 1$ for $k=1,\ldots,n$ such that $u=\sum_{k=1}^{2n+1}\theta_ku_k$. Hence, by inductively using the quasi-triangle inequality, we have
\[
q(u)\leq K^{2n}\sum_{k=1}^{2n+1}\theta_kq(u_k)\leq K^{2n}.
\]
This proves that
\[
L\subseteq\{u\in\mc{V}:q(u)\leq K^{2n}\}.
\]
Hence, by \eqref{eq:johnellipsoidquasiinclusion}, we find that for any $u\in\mc{V}$ we have
\[
    K^{-2n}q(u)\leq \|Au\|_{\mc{V}}\leq n^{\frac{1}{2}}q(u),
\]
as desired.
\end{proof}

The following application of this result proves the existence of reducing operators in directional quasi-Banach function spaces.
\begin{corollary}\label{cor:reducingmatrix}
Let $\mc{V}$ be an $n$-dimensional Hilbert space, let $\mb{X}$ be a $\mc{V}$-directional quasi-Banach function space over $\Omega$ with the Fatou property, and let $E\subseteq\Omega$ satisfy $0<\mu(E)<\infty$. Suppose that $\mb{X}$ has the property that $\ind_E u\in \mb{X}$ for all $u\in\mc{V}$. Then there is a self-adjoint positive $A_{\mb{X},E}\in\mc{L}(\mc{V})$ satisfying
\begin{equation}\label{eq:reducing_matrix}
    K_{\mb{X}}^{-2n}\|\ind_E u\|_{\mb{X}}\leq\|A_{\mb{X},E}u\|_{\mc{H}}\leq n^{\frac{1}{2}}\|\ind_E u\|_{\mb{X}}
\end{equation}
for all $u\in\mc{V}$.
\end{corollary}
The operator $A_{\mb{X},E}$ is called the \emph{reducing operator} of $E$ in $\mb{X}$.
\begin{proof}[Proof of Corollary~\ref{cor:reducingmatrix}]
By Proposition~\ref{prop:reducingmatrixquasinorm} we need only check that the quasinorm $u\mapsto \|\ind_E u\|_{\mb{X}}$ is lower semicontinuous. Indeed, if $u_k\to u$ in $\mc{V}$, then $\ind_E u_k\to \ind_E u$ a.e. and thus, by the Fatou property of $\mb{X}$, we have
\[
\|\ind_E u\|_{\mb{X}}\leq\liminf_{n\to\infty}\|\ind_E u_k\|_{\mb{X}},
\]
as desired.
\end{proof}

\subsection{Equivalent norms on the space of linear operators}

Let $\mc{V}$ be a $n$-dimensional Hilbert space for an integer $n\geq 1$, and let $(e_k)_{k=1}^n$ be an orthonormal basis of $\mc{V}$. Let $A\in\mc{L}(\mc{V})$. We claim that
\begin{equation}
    \label{eq:norm_of_matrix_columns}
    \Vert A\Vert_{\mc{V}\to\mc{V}}\eqsim_{n}\sum_{k=1}^n\Vert Ae_k\Vert_{\mc{V}}.
\end{equation}
Indeed, we have
\begin{equation*}
    \Vert A\Vert_{\mc{V}\to\mc{V}}\geq\sup_{k=1,\ldots,n}\Vert Ae_k\Vert_{\mc{V}}\geq\frac{1}{n}\sum_{k=1}^n\|Ae_k\|_{\mc{V}}.
\end{equation*}
For the converse inequality, an application of the Cauchy--Schwarz inequality combined with the bound $\|\cdot\|_{\ell^2}\leq\|\cdot\|_{\ell^1}$ yields
\begin{equation*}
    \Vert Av\Vert_{\mc{V}}\leq\Big(\sum_{k=1}^{n}\Vert Ae_k\Vert_{\mc{V}}^2\Big)^{\frac{1}{2}}\Vert v\Vert_{\mc{V}}\leq\Big(\sum_{k=1}^n\Vert Ae_k\Vert_{\mc{V}}\Big)\|v\|_{\mc{V}}
\end{equation*}
for all $v\in\mc{V}$, proving the the claim.

In practice, \eqref{eq:norm_of_matrix_columns} is used for the following result:
\begin{lemma}
    \label{lem:reducing_operator_on_operator}
    Let $X$ be a quasi-Banach function space over $\Omega$ with the Fatou property, let $\mc{V}$ be a $n$-dimensional Hilbert space with $n\geq 1$ an integer, and let $W:\Omega\to\mc{L}(\mc{V})$ be a matrix weight. Set $\mathbf{X}:=X_{W}$. Then, we have
    \begin{align*}
    \Vert\ind_{E}\Vert WA\Vert_{\mc{V}\to\mc{V}}\Vert_{X}&\eqsim_{n,K_{X}}\Vert A_{\mathbf{X},E}A\Vert_{\mc{V}\to\mc{V}},\quad\forall A\in\mc{L}(\mc{V}).
\end{align*}
\end{lemma}

\begin{proof}
    Pick an orthonormal basis $(e_k)_{k=1}^n$ for $\mc{V}$. Then, by \eqref{eq:reducing_matrix} and \eqref{eq:norm_of_matrix_columns},
\begin{align*}
    \Vert\ind_{E}\Vert WA\Vert_{\mc{V}\to\mc{V}}\Vert_{X}&\lesssim_{n,K_{X}}\sum_{k=1}^n\Vert\ind_{E}\Vert WAe\Vert_{\mc{V}}\Vert_{X}=\sum_{k=1}^n\Vert\ind_{E}Ae_k\Vert_{\mathbf{X}}\\
    &\lesssim_n\sum_{k=1}^n\Vert A_{\mathbf{X},E}Ae_k\Vert_{\mc{V}}\eqsim_{n}\Vert A_{\mathbf{X},E}A\Vert_{\mc{V}\to\mc{V}}.
\end{align*}
For the lower bound, note that, by \eqref{eq:reducing_matrix} and \eqref{eq:norm_of_matrix_columns},
\begin{align*}
\Vert\ind_{E}\Vert WA\Vert_{\mc{V}\to\mc{V}}\Vert_{X}&\geq\sup_{k=1,\ldots,n}\Vert\ind_{E}\Vert WAe_k\Vert_{\mc{V}}\Vert_{X}
=\sup_{k=1,\ldots,n}\Vert\ind_{E}Ae_k\Vert_{\mathbf{X}}\\
&\gtrsim_n\sum_{k=1}^n\Vert A_{\mathbf{X},E}Ae_k\Vert_{\mc{V}}\eqsim_{n}\Vert A_{\mathbf{X},E}A\Vert_{\mc{V}\to\mc{V}}.
\end{align*}
The result follows.
\end{proof}

\subsection{Convex-set valued mappings}
Let $\mc{V}$ be an $n$-dimensional Hilbert space over $\F$. We let $\mc{K}(\mc{V})$ denote the non-empty, closed subsets $K$ of $\mc{V}$ satisfying:
\begin{itemize}
    \item \emph{Convexity:} If $u,v\in K$, then $(1-t)u+tv\in K$ for all $0\leq t\leq 1$;
    \item\emph{Symmetry:} If $u\in K$, then $\lambda u\in K$ for all $\lambda\in\F$ with $|\lambda|=1$.
\end{itemize}
Given a set $S\subseteq\mc{V}$, we define $\mc{K}(S)$ as the smallest set in $\mc{K}(\mc{V})$ containing $S$.

Let $(\Omega,\mu)$ be a $\sigma$-finite measure space. We call a mapping $F:\Omega\to\mc{K}(\mc{V})$ \emph{measurable} if for every open set $E\subseteq\mc{V}$ the set
\[
F^{-1}(E):=\{x\in\Omega:E\cap F(x)\neq\emptyset\}
\]
is measurable. Moreover, we let $L^0(\Omega;\mc{K}(V))$ denote the space of measurable mappings $F:\Omega\to\mc{K}(\mc{V})$.

Given a $F\in L^0(\Omega;\mc{K}(\mc{V})$, we define its set of \emph{selections} by
\[
S^0(\Omega;F):=\{f\in L^0(\Omega;\mc{V}):f(x)\in F(x)\text{ a.e.}\}.
\]

Given a $\mc{V}$-directional quasi-Banach function space $\mb{X}$, we let $\mb{X}[\mc{K}]$ denote the set of those $F\in L^0(\Omega;\mc{K}(\mc{V}))$ for which $S^0(\Omega;F)$ is a bounded subset of $\mb{X}$. Moreover, in this case we set
\[
\|F\|_{\mb{X}[\mc{K}]}:=\sup_{f\in S^0(\Omega;F)}\|f\|_{\mb{X}}.
\]
If $X$ is a quasi-Banach function space with the Fatou property and $W:\Omega\to\mc{L}(\mc{V})$ is a matrix weight, then $F\in X_W[\mc{K}]$ if and only if $F(x)$ is a bounded set for a.e. $x\in\Omega$, and $h(x):=\sup_{u\in F(x)}|W(x)u|$ satisfies $h\in X$. Moreover, writing $\|WF\|_{\mc{V}}:=h$, in this case we have
\[
\|F\|_{X_W[\mc{K}]}=\big\|\|WF\|_{\mc{V}}\big\|_X,
\]
see \cite[Proposition~3.10 and 3.11]{Ni24b}. In particular, this is true for matrix weighted Lebesgue spaces, where we write
\[
L^p_W(\Omega;\mc{K}(\mc{V})):=L^p_W(\Omega;\mc{V})[\mc{K}].
\]
Furthermore, when $W$ is the identity mapping, we remove the subscript $W$.

If $F\in L^1(\Omega;\mc{K}(\mc{V}))$, then we define the \emph{Aumann integral} of $F$ by
\[
\int_\Omega\!F\,\mathrm{d}\mu:=\Big\{\int_\Omega\!f\,\mathrm{d}\mu:f\in S^0(\Omega;F)\Big\}.
\]
Note that this is a well-defined set in $\mc{V}$, since any $f\in S^0(\Omega;F)$ belongs to $L^1(\Omega;\mc{V})$. Denoting the supremum of the norms of a bounded set $S\subseteq\mc{V}$ by $\|S\|_{\mc{V}}$, we have the following result:
\begin{proposition}\label{prop:aumannintegralbound}
Let $F\in L^1(\Omega;\mc{K}(\mc{V}))$. Then the Aumann integral of $F$ is a non-empty, convex, symmetric, and bounded set, with bound
\[
\Big\|\int_\Omega\!F\,\mathrm{d}\mu\Big\|_{\mc{V}}\eqsim_n\int_\Omega\!\|F\|_{\mc{V}}\,\mathrm{d}\mu.
\]
\end{proposition}
\begin{proof}
For any $f\in S^0(\Omega;\mc{F})$ we have
\[
\Big\|\int_\Omega\!f\,\mathrm{d}\mu\Big\|_{\mc{V}}\leq\int_\Omega\!\|f\|_{\mc{V}}\,\mathrm{d}\mu\leq\|F\|_{L^1(\Omega;\mc{K}(\mc{V}))}=\int_\Omega\!\|F\|_{\mc{V}}\,\mathrm{d}\mu.
\]
Taking a supremum over all $f\in S^0(\Omega;F)$, this implies that
\[
\Big\|\int_\Omega\!F\,\mathrm{d}\mu\Big\|_{\mc{V}}\leq\int_\Omega\!\|F\|_{\mc{V}}\,\mathrm{d}\mu.
\]
For the converse inequality, pick $f\in S^0(\Omega;\mc{K}(\mc{V}))$ for which $\|f\|_{\mc{V}}=\|F\|_{\mc{V}}$. Let $(e_k)_{k=1}^n$ bean orthonormal basis of $\mc{V}$ and set
\[
f_k:=\frac{|\langle f,e_k\rangle_{\mc{V}}|}{\langle f,e_k\rangle_{\mc{V}}}f
\]
on the support of $\langle f,e_k\rangle_{\mc{V}}$. As $F$ is a symmetric set a.e., we have  $f_k\in S^0(\Omega;F)$. Thus,
\begin{align*}
\int_\Omega\!\|f\|_{\mc{V}}\,\mathrm{d}\mu
&\leq\sum_{k=1}^n\int_\Omega\!|\langle f,e_k\rangle_{\mc{V}}|\,\mathrm{d}\mu=\sum_{k=1}^n\Big\langle\int_\Omega\! f_k\,\mathrm{d}\mu,e_k\Big\rangle_{\mc{V}}\\
&\leq\sum_{k=1}^n\Big\|\int_\Omega\!f_k\,\mathrm{d}\mu\Big\|_{\mc{V}}\leq n\Big\|\int_\Omega\!F\,\mathrm{d}\mu\Big\|_{\mc{V}}.
\end{align*}
The assertion follows.
\end{proof}

We say that $F\in L^1_{\text{loc}}(\R^d;\mc{K}(\mc{V}))$ if $\ind_QF\in L^1(\R^d;\mc{K}(\mc{V}))$ for all cubes $Q\subseteq\R^d$. In this case the Aumann integral
\[
\langle F\rangle_Q=\Big\{\avint_Q\!f(x)\,\mathrm{d}x:f\in S^0(\R^d;\ind_Q F)\Big\}
\]
is a well-defined bounded set in $\mc{K}(\mc{V})$ for all cubes $Q\subseteq\R^d$.

For proofs and further properties of set-valued mappings, we refer the reader to \cite{AF09, BC23, Ni24b}.

\subsection{Sparse families}
\label{s:sparse_families}

In this subsection we briefly describe the definitions and relevant properties of sparse families that we need in this work. Further details and proofs can be found in, e.g., \cite{LN15} and \cite{NPTV17}.

\begin{definition}
    Let $0<\eta<1$. A family of cubes $\mc{S}$ in $\R^d$ is said to be $\eta$-sparse, if there exists a family $\{E_{Q}:~Q\in\mc{S}\}$ of pairwise disjoint measurable subsets of $\R^d$ with $E_{Q}\subseteq Q$ and $|E_{Q}|\geq\eta|Q|$ for all $Q\in\mc{S}$. 
\end{definition}

\begin{definition}
    Let $0<\varepsilon<1$ and let $\mc{D}$ be a dyadic grid in $\R^d$. A subfamily $\mc{S}$ of $\mc{D}$ is said to be \emph{martingale $\varepsilon$-sparse} if for all $Q\in\mc{S}$ we have
    \begin{equation*}
        \sum_{R\in\mathrm{ch}_{\mc{S}}(Q)}|R|\leq\varepsilon|Q|,
    \end{equation*}
    where $\mathrm{ch}_{\mc{S}}(Q)$ is the family of all maximal cubes in $\mc{S}$ that are strictly contained in $Q$.
\end{definition}
Setting $E_Q=Q\backslash\bigcup_{R\in\mathrm{ch}_{\mc{S}}(Q)}R$, we find that a martingale $\varepsilon$-sparse collection is $1-\varepsilon$-sparse.

We will often use the so-called ``$3^d$-lattice trick'' (see \cite[Theorem~3.1]{LN15}), stating that there are $3^d$ dyadic grids $(\mc{D}^\alpha)_{\alpha=1}^{3^d}$ in $\R^d$, such that for each cube $Q\subseteq\R^d$ there are $\alpha\in\{1,\ldots,3^d\}$ and $R\in\mc{D}^\alpha$ with $Q\subseteq R$ and $|R|\leq 6^{d}|Q|$.

In particular, if $\mc{S}$ is an $\eta$-sparse family for some $0<\eta<1$, then there are $\tfrac{\eta}{6^d}$-sparse families $\mc{S}_1,\ldots,\mc{S}_{3^{d}}$ such that $\mc{S}_{\alpha}\subseteq\mc{D}^\alpha$, $\alpha=1,\ldots,3^d$, and for each cube $Q\in\mc{S}$ there are $\alpha\in\{1,\ldots,3^d\}$ and $R\in\mc{S}_\alpha$ with $Q\subseteq R$ and $|R|\leq 6^{d}|Q|$. Moreover, if $\mc{S}$ is finite, then all families $\mc{S}_{\alpha}$, $\alpha=1,\ldots,3^d$ can be taken to be finite.

\section{Averaging operators}
\label{sec:averaging_operators}

As before, throughout this section we let $(\Omega,\mu)$ be a $\sigma$-finite measure space. Moreover, we fix measurable subsets $E,E'$ of $\Omega$ with $0<\mu(E),\mu(E')<\infty$. Then, we define the averaging operator $T_{E,E'}$ by
\begin{equation*}
    T_{E,E'}(\vec{f}):=\Big(\bigotimes_{j=1}^{m}\langle f_j\rangle_{E}\Big)\ind_{E'}
\end{equation*}
acting on $m$-tuples of functions $\vec{f}=(f_1,\ldots,f_m)$ with $f_j\ind_E\in L^1(\Omega)$ for all $j=1,\ldots,m$. If $E=E'$, then we set $T_{E}:=T_{E,E}$.

We wish to obtain a characterization of the boundedness of $T_{E,E'}$ from 
\[
\vec{\mb{X}}:=\mathbf{X}_1\times\ldots\times\mathbf{X}_m,
\]
where each $\mathbf{X}_j$ is a $\mc{H}_j$-directional Banach space, into an $\mc{H}$-directional quasi-Banach space $\mathbf{Y}$ in terms of reducing operators. 

\begin{proposition}
\label{prop:reducingmatrixavop}
Let $\mathbf{X}_j$ be a $\mc{H}_j$-directional Banach function space over $\Omega$ with the component-wise saturation property for each $j=1,\ldots,m$, and let $\mathbf{Y}$ be a $\mc{H}$-directional quasi-Banach function space over $\Omega$ with the Fatou property. Then, the following are equivalent:
\begin{enumerate}[(i)]    
    \item\label{it:reducingmatrixprop1} $f_j\in\mb{X}_j$ is integrable over $E$ for all $j=1,\ldots,m$, and $T_{E,E'}:\vec{\mb{X}}\to\mathbf{Y}$;
    \item\label{it:reducingmatrixprop2} $\ind_E v\in \mathbf{X}_j'$ for all $v\in\mc{H}_j$ for all $j=1,\ldots,m$, and $\ind_{E'}\mathbf{u}\in\mathbf{Y}$ for all $\mathbf{u}\in\mc{H}$.
\end{enumerate}
In this case, we have
\begin{equation}
    \label{eq:bilinear_matrix_Ap_reducing_operator}
    \|T_{E,E'}\|_{\mathbf{X}_1\times\ldots\times\mathbf{X}_m\to\mathbf{Y}}\eqsim_{\vec{n},K_{\mathbf{Y}}}\mu(E)^{-m}\Big\|A_{\mathbf{Y},E'}\Big(\bigotimes_{j=1}^{m} A_{\mathbf{X}_j',E}\Big)\Big\|_{\mc{H}\to\mc{H}}.
\end{equation}
\end{proposition}

\medskip

We note that since $A_{\mathbf{Y},E'}$ and $A_{\mathbf{X}_j',E}$, $j=1,\ldots,m$ are positive definite and thus Hermitian operators, we have
\begin{align*}
    &\Big\|A_{\mathbf{Y},E'}\Big(\bigotimes_{j=1}^{m} A_{\mathbf{X}_j',E}\Big)\Big\|_{\mc{H}\to\mc{H}}
    =\Big\|\Big(\bigotimes_{j=1}^{m} A_{\mathbf{X}_j',E}\Big)A_{\mathbf{Y},E'}\Big\|_{\mc{H}\to\mc{H}}.
\end{align*}

\medskip

Before proceeding to the proof of Proposition~\ref{prop:reducingmatrixavop}, we need to review some elementary tensor algebra.

Fix $i\in\{1,\ldots,m\}$ and define
\begin{equation*}
    \vec{\mc{H}}_{(i)}:=(\mc{H}_i,\ldots,\mc{H}_m),\quad\vec{\mc{H}}^{(i)}:=(\mc{H}_1,\ldots,\mc{H}_i)
\end{equation*}
as well as
\begin{equation*}
    \mc{H}_{(i)}:=\bigotimes_{j=i}^{m}\mc{H}_j,\quad\mc{H}^{(i)}:=\bigotimes_{j=1}^{i}\mc{H}_j.
\end{equation*}
We also set $\mc{H}^{(0)}=\mc{H}_{(m+1)}:=\mathbf{F}$. Given now $\vec{u}\in\vec{\mc{H}}$, we define
\begin{equation*}
    \Big(\bigotimes_{j=1}^{m}u_j\Big)\Big(\bigotimes_{j=i}^{m}v_j\Big):=\Big(\prod_{j=i}^{m}\langle u_j,v_j\rangle_{\mc{H}_j}\Big)\bigotimes_{j=1}^{i-1}u_j\in\mc{H}^{(i-1)},\quad\forall\vec{v}\in\vec{\mc{H}}_{(i)}
\end{equation*}
as well as
\begin{equation*}
    \Big(\bigotimes_{j=1}^{m}u_j\Big)\Big(\bigotimes_{j=1}^{i}v_j\Big):=\Big(\prod_{j=1}^{i}\langle u_j,v_j\rangle_{\mc{H}_j}\Big)\bigotimes_{j=i+1}^{m}u_j\in\mc{H}_{(i+1)},\quad\forall\vec{v}\in\vec{\mc{H}}^{(i)}.
\end{equation*}
These definitions arise from the canonical isomorphisms
\begin{equation*}
    \mc{H}_j\ni u_j\mapsto\langle u_j,\cdot\rangle\in\mc{H}_j^{*},\quad j=1,\ldots,m,
\end{equation*}
where $\mc{H}_j^{*}$ denotes the space of sesquilinear forms on $\mc{H}_j$, and noticing then that these yield canonical isomorphisms
\begin{equation*}
    \mc{H}\cong\Big(\bigotimes_{j=1}^{i-1}\mc{H}_j\Big)\otimes\Big(\bigotimes_{j=i}^{m}\mc{H}_j^{*}\Big)
\end{equation*}
and
\begin{equation*}
    \mc{H}\cong\Big(\bigotimes_{j=1}^{i}\mc{H}_j^{*}\Big)\otimes\Big(\bigotimes_{j=i+1}^{m}\mc{H}_j\Big).
\end{equation*}
By linearity, we can then consider the products $\mathbf{u}\mathbf{v}$ and $\mathbf{u}\mathbf{w}$, for all $\mathbf{u}\in\mc{H}$, $\mathbf{v}\in\mc{H}_{(i)}$ and $\mathbf{w}\in\mc{H}^{(i)}$. One can verify through a direct computation that
\begin{equation}
    \label{eq:partial_tensor_inner_upper}
    \Big\langle \mathbf{u}\Big(\bigotimes_{j=i}^{m}v_j\Big),\bigotimes_{j=1}^{i-1}v_j\Big\rangle_{\mc{H}^{(i-1)}}=\Big\langle \mathbf{u},\bigotimes_{j=1}^{m}v_j\Big\rangle_{\mc{H}}
\end{equation}
and
\begin{equation}
    \label{eq:partial_tensor_inner_lower}
    \Big\langle \mathbf{u}\Big(\bigotimes_{j=1}^{i}v_j\Big),\bigotimes_{j=i+1}^{m}v_j\Big\rangle_{\mc{H}_{(i+1)}}=\Big\langle \mathbf{u},\bigotimes_{j=1}^{m}v_j\Big\rangle_{\mc{H}},
\end{equation}
for all $\mathbf{u}\in\mc{H}$ and $\vec{v}\in\vec{\mc{H}}$.

Finally, we note the identity
\begin{equation}
    \label{eq:iterate_identity}
    (\ldots(((\mathbf{u}v_1)v_2)v_3)\ldots)v_{m}=\Big\langle\mathbf{u},\bigotimes_{j=1}^{m}v_j\Big\rangle_{\mc{H}},
\end{equation}
for all $\mathbf{u}\in\mc{H}$ and $\vec{v}\in\vec{\mc{H}}$.

\medskip

\begin{proof}[Proof of Proposition~\ref{prop:reducingmatrixavop}]
Assume first that all $f_j\in\mb{X}_j$ are integrable over $E$ for $j=1,\ldots,m$ and $T_E:\vec{\mb{X}}\to\mathbf{Y}$. To see that $\ind_E u\in\mb{X}$ for all $u\in\mc{H}$, note that, by linearity, it suffices to show this for $\mathbf{u}=\displaystyle\bigotimes_{j=1}^m u_j$ with $u_j\in\mc{H}_j$. As $\mathbf{X}_j$ possesses the component-wise saturation property, by \cite[Proposition 3.4]{Ni24} we have that the $\F$-Banach function space $(\mathbf{X}_j)_{u_j}$ is saturated, therefore by \cite[Proposition~2.5]{LN23b} there exists a weak order unit $\rho_j\in(\mathbf{X}_j)_{u_j}$ with $\rho_j>0$ $\mu$-a.e.~on $\Omega$, for all $j\in\{1,\ldots,m\}$. Observe that $\rho_{j}u_{j}\in \mathbf{X}_j$, for all $j=1,\ldots,m$, and thus
\begin{equation*}
    \prod_{j=1}^m\langle\rho_j\rangle_E\ind_{E'} u=T_{E,E'}(\rho_1u_1,\ldots,\rho_mu_m)\in\mb{Y}.
\end{equation*}
Since $\displaystyle\prod_{j=1}^m\langle\rho_j\rangle_E>0$, we conclude that $\ind_{E'}\mathbf{u}\in\mb{Y}$, as asserted.

Next, let $\vec{v}\in\vec{\mc{H}}$ be arbitrary. Then, for all $f_j\in\mathbf{X}_j$ with $\Vert f_j\Vert_{\mathbf{X}_j}=1$, $j=1,\ldots,m$, letting $u_j:=\langle f_j\rangle_{E}$, $j=1,\ldots,m$ we estimate
\begin{align*}
    &\mu(E)^{-m}\prod_{j=1}^{m}\Big|\int_{E}\langle f_j(x),v_j\rangle_{\mc{H}_j}\mathrm{d}\mu(x)\Big|
    =\prod_{j=1}^{m}|\langle\langle f_j\rangle_{E},v_j\rangle_{\mc{H}_j}|
    =\Big|\Big\langle\bigotimes_{j=1}^{m}\langle f_j\rangle_E,\bigotimes_{j=1}^{m}v_j\Big\rangle_{\mc{H}}\Big|\\
    &=\frac{\Big|\Big\langle\displaystyle\bigotimes_{j=1}^{m}u_j,\displaystyle\bigotimes_{j=1}^{m}v_j\Big\rangle_{\mc{H}}\Big|}{\Big\|\ind_{E'}\Big(\displaystyle\bigotimes_{j=1}^{m}u_j\Big)\Big\|_{\mathbf{Y}}}\|T_{E,E'}(f_1,\ldots,f_m)\|_{\mathbf{Y}}\leq\|T_{E,E'}\|_{\mathbf{X}_1\times\ldots\times\mathbf{X}_m\to\mathbf{Y}}\frac{\Big|\Big\langle\displaystyle\bigotimes_{j=1}^{m}u_j,\displaystyle\bigotimes_{j=1}^{m}v_j\Big\rangle_{\mc{H}}\Big|}{\Big\|\ind_{E'}\Big(\displaystyle\bigotimes_{j=1}^{m}u_j\Big)\Big\|_{\mathbf{Y}}}\\
    &\leq\|T_{E,E'}\|_{\mathbf{X}_1\times\ldots\times\mathbf{X}_m\to\mathbf{Y}}\sup_{\mathbf{u}\in\mc{H}\backslash\{0\}}\frac{\Big|\Big\langle \mathbf{u},\displaystyle\bigotimes_{j=1}^{m}v_j\Big\rangle_{\mc{H}}\Big|}{\|\ind_{E'}\mathbf{u}\|_{\mathbf{Y}}}\\
    &\leq\|T_{E,E'}\|_{\mathbf{X}_1\times\ldots\times\mathbf{X}_m\to\mathbf{Y}}\Big\Vert\bigotimes_{j=1}^{m}v_j\Big\Vert_{\mc{H}}\sup_{\mathbf{u}\in\mc{H}\backslash\{0\}}\frac{\Vert\mathbf{u}\Vert_{\mc{H}}}{\Vert\ind_{E'}\mathbf{u}\Vert_{\mathbf{Y}}},
\end{align*}
where we observe that
\begin{equation*}
    \sup_{\mathbf{u}\in\mc{H}\backslash\{0\}}\frac{\Vert\mathbf{u}\Vert_{\mc{H}}}{\|\ind_{E'}\mathbf{u}\|_{\mathbf{Y}}}<\infty
\end{equation*}
exactly as in the proof of part (a) in \cite[Proposition 4.1]{Ni24}, or alternatively, as a consequence of \eqref{eq:reducing_matrix}. Therefore, for each $j=1,\ldots,m$, by fixing $f_i$ for $i=1,\ldots,m$ with $i\neq j$ and letting $f_j$ vary, from \eqref{eq:Koethe_dual_directional} we deduce $\ind_E v_j\in\mathbf{X}_j'$.

Next, we estimate
\begin{align*}
    &R:=\mu(E)^{m}\|T_{E,E'}\|_{\vec{\mb{X}}\to\mathbf{Y}}
    =\sup_{\substack{f_j\in\overline{B}_{\mathbf{X}_j}\\j=1,\ldots,m}}\mu(E)^{m}\Vert T_{E,E'}(\vec{f})\Vert_{\mathbf{Y}}\\
    &=\sup_{\substack{f_j\in\overline{B}_{\mathbf{X}_j}\\j=1,\ldots,m}}\mu(E)^{m}\Big\Vert\ind_{E'}\Big(\bigotimes_{j=1}^{m}\langle f\rangle_{E}\Big)\Big\Vert_{\mathbf{Y}}
    \overset{\eqref{eq:reducing_matrix}}{\eqsim_{K_{\mathbf{Y}},\vec{n}}}\sup_{\substack{f_j\in\overline{B}_{\mathbf{X}_j}\\j=1,\ldots,m}}\mu(E)^{m}\Big\Vert A_{\mathbf{Y},E'}\Big(\bigotimes_{j=1}^{m}\langle f_j\rangle_{E}\Big)\Big\Vert_{\mc{H}}\\
    &=\sup_{\substack{f_j\in\overline{B}_{\mathbf{X}_j}\\j=1,\ldots,m}}\sup_{\mathbf{u}\in\overline{B}_\mc{H}}\mu(E)^{m}\Big|\Big\langle A_{\mathbf{Y},E'}\Big(\bigotimes_{j=1}^{m}\langle f_j\rangle_{E}\Big),\mathbf{u}\Big\rangle_{\mc{H}}\Big|\\
    &=\sup_{\mathbf{u}\in\overline{B}_\mc{H}}\sup_{\substack{f_j\in\overline{B}_{\mathbf{X}_j}\\j=1,\ldots,m}}\mu(E)^{m}\Big|\Big\langle\bigotimes_{j=1}^{m}\langle f_j\rangle_{E},A_{\mathbf{Y},E'}\mathbf{u}\Big\rangle_{\mc{H}}\Big|.
\end{align*}
Let us fix for the moment $\mathbf{u}\in\mc{H}$ as well as $f_j\in\overline{B}_{\mathbf{X}_j}$ for $j=2,\ldots,m$ and estimate
\begin{equation*}
    R_1:=\mu(E)\sup_{\substack{f_1\in\overline{B}_{\mathbf{X}_1}}}\Big|\Big\langle\bigotimes_{j=1}^{m}\langle f_j\rangle_{E},A_{\mathbf{Y},E'}\mathbf{u}\Big\rangle_{\mc{H}}\Big|.
\end{equation*}
Thanks to \eqref{eq:partial_tensor_inner_upper} we can write
\begin{align*}
    R_1
    =\mu(E)\sup_{\substack{f_1\in\overline{B}_{\mathbf{X}_1}}}\Big|\Big\langle\langle f_1\rangle_{E},(A_{\mathbf{Y},E'}\mathbf{u})\Big(\bigotimes_{j=2}^{m}\langle f_j\rangle_{E}\Big)\Big\rangle_{\mc{H}_1}\Big|.
\end{align*}
Setting
\begin{equation*}
    u_1:=(A_{\mathbf{Y},E'}\mathbf{u})\Big(\bigotimes_{j=2}^{m}\langle f_j\rangle_{E}\Big)\in\mc{H}_1,
\end{equation*}
we can continue the computation using \eqref{eq:Koethe_dual_directional} to find that
\begin{align*}
    R_1&=\mu(E)\sup_{\substack{f_1\in\overline{B}_{\mathbf{X}_1}}}\Big|\Big\langle\langle f_1\rangle_{E},u_1\Big\rangle_{\mc{H}_1}\Big|
    =\sup_{\substack{f_1\in\overline{B}_{\mathbf{X}_1}}}\Big|\int_{E}\langle f_1(x),u_1\rangle_{\mc{H}_1}\,\mathrm{d}\mu(x)\Big|\\
    &=\Vert \ind_{E}u_1\Vert_{\mathbf{X}_1'}\eqsim_{n_1}\Vert A_{\mathbf{X}_1',E}u_1\Vert_{\mc{H}_1}.
\end{align*}
Thus, we obtain
\begin{align*}
    R&\eqsim_{\vec{n}}\mu(E)^{m-1}\sup_{\mathbf{u}\in\overline{B}_{\mc{H}}}\sup_{\substack{f_j\in\overline{B}_{\mathbf{X}_j}\\j=2,\ldots,m}}\Big\Vert A_{\mathbf{X}_1',E}\Big[(A_{\mathbf{Y},E'}\mathbf{u})\Big(\bigotimes_{j=2}^{m}\langle f_j\rangle_{E}\Big)\Big]\Big\Vert_{\mc{H}_1}\\
    &=\mu(E)^{m-1}\sup_{\mathbf{u}\in\overline{B}_{\mc{H}}}\sup_{\substack{f_j\in\overline{B}_{\mathbf{X}_j}\\j=2,\ldots,m}}\sup_{v_1\in\overline{B}_{\mc{H}_1}}\Big|\Big\langle A_{\mathbf{X}_1',E}\Big[(A_{\mathbf{Y},E'}\mathbf{u})\Big(\bigotimes_{j=2}^{m}\langle f_j\rangle_{E}\Big)\Big],v_1\Big\rangle_{\mc{H}_1}\Big|\\
    &=\mu(E)^{m-1}\sup_{v_1\in\overline{B}_{\mc{H}_1}}\sup_{\mathbf{u}\in\overline{B}_{\mc{H}}}\sup_{\substack{f_j\in\overline{B}_{\mathbf{X}_j}\\j=2,\ldots,m}}\Big|\Big\langle (A_{\mathbf{Y},E'}\mathbf{u})\Big(\bigotimes_{j=2}^{m}\langle f_j\rangle_{E}\Big),A_{\mathbf{X}_1',E}v_1\Big\rangle_{\mc{H}_1}\Big|\\
    &\overset{\eqref{eq:partial_tensor_inner_upper},\eqref{eq:partial_tensor_inner_lower}}{=}\mu(E)^{m-1}\sup_{v_1\in\overline{B}_{\mc{H}_1}}\sup_{\mathbf{u}\in\overline{B}_{\mc{H}}}\sup_{\substack{f_j\in\overline{B}_{\mathbf{X}_j}\\j=2,\ldots,m}}\Big|\Big\langle \bigotimes_{j=2}^{m}\langle f_j\rangle_{E},(A_{\mathbf{Y},E'}\mathbf{u})(A_{\mathbf{X}_1',E}v_1)\Big\rangle_{\mc{H}_{(2)}}\Big|.
\end{align*}
Then, similarly to previously we deduce
\begin{align*}
    R\eqsim_{\vec{n}}\mu(E)^{m-2}\sup_{v_1\in\overline{B}_{\mc{H}_1}}\sup_{\mathbf{u}\in\overline{B}_{\mc{H}}}\sup_{\substack{f_j\in\overline{B}_{\mathbf{X}_j}\\j=3,\ldots,m}}\Big\Vert A_{\mathbf{X}_2',E}\Big[\Big[(A_{\mathbf{Y},E'}\mathbf{u})(A_{\mathbf{X}_1',E}v_1)\Big]\Big(\bigotimes_{j=3}^{m}\langle f_j\rangle_{E}\Big)\Big]\Big\Vert_{\mc{H}_2}.
\end{align*}
Continuing inductively, we eventually arrive at
\begin{align*}
   R&\eqsim_{\vec{n}}\sup_{\substack{v_j\in\overline{B}_{\mc{H}_j}\\j=1,\ldots,m}}\sup_{\mathbf{u}\in\overline{B}_{\mc{H}}}\Big|[\ldots[[(A_{\mathbf{Y},E'}\mathbf{u})A_{\mathbf{X}_1',E}v_1](A_{\mathbf{X}_2',E}v_2)]\ldots](A_{\mathbf{X}_m',E}v_m)]\Big|\\
   &=\sup_{\substack{v_j\in\overline{B}_{\mc{H}_j}\\j=1,\ldots,m}}\sup_{\mathbf{u}\in\overline{B}_{\mc{H}}}\Big|\Big\langle A_{\mathbf{Y},E'}\mathbf{u},\bigotimes_{j=1}^{m}(A_{\mathbf{X}_j',E}v_j)\Big\rangle_{\mc{H}}\Big|\\
   &=\sup_{\substack{v_j\in\overline{B}_{\mc{H}_j}\\j=1,\ldots,m}}\sup_{\mathbf{u}\in\overline{B}_{\mc{H}}}\Big|\Big\langle A_{\mathbf{Y},E'}\mathbf{u},\Big(\bigotimes_{i=1}^{m}A_{\mathbf{X}_j',E}\Big)\Big(\bigotimes_{j=1}^{m}v_j\Big)\Big\rangle_{\mc{H}}\Big|\\
   &=\sup_{\substack{v_j\in\overline{B}_{\mc{H}_j}\\j=1,\ldots,m}}\sup_{\mathbf{u}\in\overline{B}_{\mc{H}}}\Big|\Big\langle \mathbf{u},A_{\mathbf{Y},E'}\Big(\bigotimes_{j=1}^{m}A_{\mathbf{X}_j',E}\Big)\Big(\bigotimes_{j=1}^{m}v_j\Big)\Big\rangle_{\mc{H}}\Big|\\
   &=\sup_{\substack{v_j\in\overline{B}_{\mc{H}_j}\\j=1,\ldots,m}}\Big\Vert A_{\mathbf{Y},E'}\Big(\bigotimes_{j=1}^{m}A_{\mathbf{X}_j',E}\Big)\Big(\bigotimes_{j=1}^{m}v_j\Big)\Big\Vert_{\mc{H}}\overset{\eqref{eq:norm_of_matrix_columns}}{\eqsim_{\vec{n}}}\Big\Vert A_{\mathbf{Y},E'}\Big(\bigotimes_{j=1}^{m}A_{\mathbf{X}_j',E}\Big)\Big\Vert_{\mc{H}\to\mc{H}}.
\end{align*}

\smallskip

Conversely, assume that $\ind_E v\in\mathbf{X}_j'$, for all $v\in\mc{H}_j$, for all $j=1,\ldots,m$ and $\ind_{E'}\mathbf{u}\in\mathbf{Y}$, for all $\mathbf{u}\in\mc{H}$. First of all, observe that for all $j=1,\ldots,m$, since $\ind_E v\in\mathbf{X}_j'$, for all $v\in\mc{H}_j$, exactly as in the proof of part (b) of \cite[Proposition 4.1]{Ni24} we have that every $f\in\mathbf{X}_j$ is integrable over $E$.

Let now $f_j\in\mathbf{X}_j$, $j=1,\ldots,m$ be arbitrary. Set
\begin{equation*}
    C:=\Big\Vert A_{\mathbf{Y},E'}\Big(\bigotimes_{j=1}^{m}A_{\mathbf{X}_j',E}\Big)\Big\Vert_{\mc{H}\to\mc{H}}.
\end{equation*}
We estimate
\begin{align*}
    &\mu(E)^{m}\Vert T_{E,E'}(f_1,\ldots,f_m)\Vert_{\mathbf{Y}}=\mu(E)^{m}\Big\Vert\ind_{E'}\Big(\bigotimes_{j=1}^{m}\langle f_j\rangle_{E}\Big)\Big\Vert_{\mathbf{Y}}\\
    &\overset{\eqref{eq:reducing_matrix}}{\lesssim_{K_{\mathbf{Y}},\vec{n}}}\mu(E)^{m}C\Big\Vert\bigotimes_{j=1}^{m}(A_{\mathbf{X}_j',E}^{-1}\langle f_j\rangle_{E})\Big\Vert_{\mc{H}}
    =C\prod_{j=1}^{m}\Big\Vert\int_{E}A_{\mathbf{X}_j',E}^{-1}f_j(x)\mathrm{d}\mu(x)\Big\Vert_{\mc{H}_j}.
\end{align*}
For each $j=1,\ldots,m$, choosing an orthonormal basis $e_1,\ldots,e_{n_j}$ of $\mc{H}_j$ we get
\begin{align*}
    &\Big\Vert\int_{E}A_{\mathbf{X}_j,E}'f_j(x)\mathrm{d}\mu(x)\Big\Vert_{\mc{H}}^2=
    \sum_{k=1}^{n_j}\Big|\Big\langle \int_{E}A_{\mathbf{X}_j',E}^{-1}f_j(x)\mathrm{d}\mu(x),e_{k}\Big\rangle_{\mc{H}_j}\Big|^2\\
    &=\sum_{k=1}^{n_j}\Big|\int_{E}\langle f_j(x),A_{\mathbf{X}_j',E}^{-1}e_k\rangle_{\mc{H}_j}\mathrm{d}\mu(x)\Big|^2
    \leq\Vert f\Vert_{\mathbf{X}_j}^2\sum_{k=1}^{n_j}\Vert \ind_{E}A_{\mathbf{X}_j',E}^{-1}e_k\Vert_{\mathbf{X}_j'}^2\\
    &\eqsim_{n_j}
    \Vert f\Vert_{\mathbf{X}_j}^2\sum_{k=1}^{n_j}\Vert A_{\mathbf{X}_j',E}A_{\mathbf{X}_j',E}^{-1}e_k\Vert_{\mc{H}_j}^2=n_j\Vert f\Vert_{\mathbf{X}_j}^2.
\end{align*}
Therefore, we deduce
\begin{equation*}
    \mu(E)^{m}\Vert T_{E,E'}(\vec{f})\Vert_{\mathbf{Y}}\lesssim_{\vec{n}}C\prod_{j=1}^{m}\Vert f_j\Vert_{\mathbf{X}_j}.
\end{equation*}
Thus, $T_{E}:\mathbf{X}_1\times\ldots\times\mathbf{X}_m\to\mathbf{Y}$.
\end{proof}

Using Lemma~\ref{lem:reducing_operator_on_operator} in the special case of weighted Lebesgue spaces, we regain the familiar expression featuring integral averages. To state it, we denote by $A_{W,E,p}$ the reducing operator in $L^{p}_{W}\Big(\Omega,\frac{\mu}{\mu(E)}\Big)$ of $E$ for $p\in(0,\infty]$ (where we note that the quasi-Banach function space $L^{p}(\Omega,\frac{\mu}{\mu(E)})$ has the Fatou property). Observe that
\begin{equation*}
    K_{L^{p}\Big(\Omega,\frac{\mu}{\mu(E)}\Big)}=2^{(\frac{1}{p}-1)_+}=:K_p,\quad\forall p\in(0,\infty].
\end{equation*}

\begin{proposition}
    \label{prop:roudenko_type_characterization}
    Let $\vec{t}\in(0,\infty]^{m}$, $q\in(0,\infty]$ and let $\vec{W}$ be a $m$-tuple of matrix weights. Then, the following are equivalent:
    \begin{enumerate}[(i)]
        \item\label{it:roudenko_type_characterization1} $W_j^{-1}$ is $t_j$-integrable over $E$, for all $j=1,\ldots,m$ and $\mathbf{W}$ is $q$-integrable over $E'$

        \item\label{it:roudenko_type_characterization2}  $C:=\Big(\displaystyle\avint_{E'}\displaystyle\prod_{j=1}^{m}\Big(\avint_{E}\|W_j(y)W_j(x)^{-1}\|_{\mc{H}_j\to\mc{H}_j}^{t_j}\mathrm{d}\mu(x)\Big)^{\frac{q}{t_j}}\,\mathrm{d}\mu(y)\Big)^{\frac{1}{q}}<\infty$
    \end{enumerate}
    (where the usual interpretations are in force if $t_j=\infty$ for some $j$ or $q=\infty$).
    Moreover, in this case we have
    \begin{equation}
    \label{eq:Roudenko}
        C\eqsim_{\vec{n},K_{t_1},\ldots,K_{t_m},K_q}\Big\|A_{\mathbf{W},E',q}\Big(\bigotimes_{j=1}^{m} A_{W_j^{-1},E,t_j}\Big)\Big\|_{\mc{H}\to\mc{H}}.
    \end{equation}
\end{proposition}

\medskip

We note that if $t_j<\infty$, then the condition that $W_j^{-1}$ is $t_j$-integrable over $E$ means
\begin{equation*}
    \int_{E}\Vert W_j(x)^{-1}\Vert^{t_j}_{\mc{H}_j\to\mc{H}_j}\mathrm{d}\mu(x)<\infty,
\end{equation*}
or equivalently (by the spectral theorem for positive definite operators on a Hilbert space)
\begin{equation*}
    \int_{E}\Vert W_j(x)^{-t_j}\Vert_{\mc{H}_j\to\mc{H}_j}\mathrm{d}\mu(x)<\infty,
\end{equation*}
that is, $W_j^{-t_j}$ is integrable over $E$. Similarly, if $q<\infty$, then the condition that $\mathbf{W}$ is $q$-integrable over $E'$ is equivalent to the condition that $\mathbf{W}^q$ is integrable over $E'$.

\medskip

\begin{proof}[Proof of Proposition~\ref{prop:roudenko_type_characterization}]
    First of all, if \ref{it:roudenko_type_characterization1} is true, then the reducing matrices in \eqref{eq:Roudenko} are well-defined. Thus, using Lemma \ref{lem:reducing_operator_on_operator}, we have
    \begin{align*}
        C&=\Big(\avint_{E'}\prod_{j=1}^{m}\Big(\avint_{E}\|W_j(y)W_j(x)^{-1}\|_{\mc{H}_j\to\mc{H}_j}^{t_j}\,\mathrm{d}x\Big)^{\frac{q}{t_j}}\,\mathrm{d}y\Big)^{\frac{1}{q}}\\
        &=\Big(\avint_{E'}\prod_{j=1}^{m}\Big(\avint_{E}\|W_j(x)^{-1}W_j(y)\|_{\mc{H}_j\to\mc{H}_j}^{t_j}\,\mathrm{d}x\Big)^{\frac{q}{t_j}}\,\mathrm{d}y\Big)^{\frac{1}{q}}\\
        &\eqsim_{\vec{n},K_{t_1},\ldots,K_{t_m}}\Big(\avint_{E'}\prod_{j=1}^{m}\Vert A_{W_j^{-1},E,t_j}W_j(y)\Vert_{\mc{H}_j\to\mc{H}_j}^{q}\,\mathrm{d}y\Big)^{\frac{1}{q}}\\
        &=\Big(\avint_{E'}\Big\Vert\Big(\bigotimes_{j=1}^{m}A_{W_j^{-1},E,t_j}\Big)\mathbf{W}(y)\Big\Vert_{\mc{H}\to\mc{H}}^{q}\mathrm{d}\mu(y)\Big)^{\frac{1}{q}}\\
        &=
        \Big(\avint_{E'}\Big\Vert\mathbf{W}(y)\Big(\bigotimes_{j=1}^{m}A_{W_j^{-1},E,t_j}\Big)\Big\Vert_{\mc{H}\to\mc{H}}^{q}\,\mathrm{d}y\Big)^{\frac{1}{q}}\\
        &\eqsim_{n,K_q}\Big\|A_{\mathbf{W},E',q}\Big(\bigotimes_{j=1}^{m} A_{W_j^{-1},E,t_j}\Big)\Big\|_{\mc{H}\to\mc{H}}.
    \end{align*}

    Conversely, assume that $C<\infty$. Then, there is $y_0\in E$ such that $\mathbf{W}_j(y_0)$ is invertible, for all $j=1,\ldots,m$ and
    \begin{equation*}
        \prod_{j=1}^{m}\Big(\avint_{E}\|W_j(y_0)W_j(x)^{-1}\|_{\mc{H}_i\to\mc{H}_j}^{t_j}\mathrm{d}\mu(x)\Big)^{\frac{q}{t_j}}<\infty.
    \end{equation*}
    Then
    \begin{equation*}
       \avint_{E}\|W_j(y_0)W_j(x)^{-1}\|_{\mc{H}_j\to\mc{H}_j}^{t_j}\mathrm{d}\mu(x)<\infty,\quad\forall j=1,\ldots,m.
    \end{equation*}
    Since $W_{j}(y_0)$ is a fixed invertible linear operator on $\mc{H}_j$, we deduce that $W_j^{-1}$ is $t_j$-integrable over $E$ for all $j=1,\ldots,m$. Then, similarly to previously, we have
    \begin{align*}
       \infty > C \eqsim_{\vec{n},K_{t_1},\ldots,K_{t_m}} \Big(\avint_{E'}\Big\Vert\Big(\bigotimes_{j=1}^{m}A_{W_j^{-1},E,t_j}\Big)\mathbf{W}(y)\Big\Vert_{\mc{H}\to\mc{H}}^{q}\,\mathrm{d}y\Big)^{\frac{1}{q}}.
    \end{align*}
    Since $\displaystyle\bigotimes_{j=1}^{m}A_{W_j^{-1},E,t_j}\in\mathcal{L}(\mc{H})$ is a fixed invertible linear operator on $\mc{H}$, we deduce that $\mathbf{W}$ is $q$-integrable over $E'$, concluding the proof.
\end{proof}

Finally, we prove a result that shows that the linear matrix Muckenhoupt classes embed into our bilinear matrix Muckenhoupt weight classes.
\begin{proposition}
\label{prop:from_bilinear_to_linear}
Let $X$ be a Banach function space over $\Omega$ with the Fatou property and let $W:\Omega\to\mc{L}(\mc{H})$ be a matrix weight. Then the following are equivalent:
\begin{enumerate}[(i)]
    \item\label{it:bilinearweightembedding1} $T_E: X_W\to X_W$;
    \item\label{it:bilinearweightembedding2} $T_E:(X')_{W^{-1}}\to (X')_{W^{-1}}$;
    \item\label{it:bilinearweightembedding3} $T_E:X_W\times (X')_{W^{-1}}\to L^1_{W\otimes W^{-1}}(\Omega;\mc{H}\otimes\mc{H})$.
\end{enumerate}
Moreover, in this case we have
\[
\|T_E\|_{X_W\to X_W}=\|T_E\|_{(X')_{W^{-1}}\to (X')_{W^{-1}}}\leq \|T_E\|_{X_W\times (X')_{W^{-1}}\to L^1_{W\otimes W^{-1}}(\Omega;\mc{H}\otimes\mc{H})},
\]
\[
\|T_E\|_{X_W\times (X')_{W^{-1}}\to L^1_{W\otimes W^{-1}}(\Omega;\mc{H}\otimes\mc{H})}\leq\|T_E\|_{X_W\to X_W}\|T_E\|_{(X')_{W^{-1}}\to (X')_{W^{-1}}}.
\]
\end{proposition}
\begin{proof}
The equivalence \ref{it:bilinearweightembedding1}$\Leftrightarrow$\ref{it:bilinearweightembedding2} with equal norm follows from \eqref{eq:Koethe_dual_directional} and the fact that
\[
\int_{\Omega}\langle T_Ef,g\rangle_{\mc{H}}\,\mathrm{d}\mu=\int_{\Omega}\langle f, T_Eg\rangle_{\mc{H}}\,\mathrm{d}\mu.
\]
For \ref{it:bilinearweightembedding3}$\Rightarrow$\ref{it:bilinearweightembedding1}, note that for any $g\in (X')_{W^{-1}}$ we have
\begin{align*}
\Big|\int_{\Omega}\!\langle T_Ef , g\rangle_{\mc{H}}\,\mathrm{d}\mu\Big|
&=\Big|\int_E\!\langle \langle f\rangle_E, \langle g\rangle_E\rangle_{\mc{H}}\,\mathrm{d}\mu\Big|
=\Big|\int_E\!\langle W(x)\langle f\rangle_E ,  W(x)^{-1}\langle g\rangle_E\rangle_{\mc{H}}\,\mathrm{d}\mu\Big|\\
&\leq\int_E\!\|W(x)\langle f\rangle_E\|_{\mc{H}}\|W(x)^{-1}\langle g\rangle_E\|_{\mc{H}}\,\mathrm{d}\mu=\|T_E(f,g)\|_{L^1_{W\otimes W^{-1}}(\Omega;\F^n\otimes\F^n)}\\
&\leq \|T_E\|_{X_W\times (X')_{W^{-1}}\to L^1_{W\otimes W^{-1}}(\Omega;\mc{H}\otimes\mc{H})}\|f\|_{X_W}\|g\|_{(X')_{W^{-1}}}
\end{align*}
Taking a supremum over all $g\in (X')_{W^{-1}}=(X_W)'$ of norm $1$ proves that $T_E:X_W\to X_W$, with
\[
\|T_E\|_{X_W\to X_W}\leq \|T_E\|_{X_W\times (X')_{W^{-1}}\to L^1_{W\otimes W^{-1}}(\Omega;\mc{H}\otimes\mc{H})}.
\]
For \ref{it:bilinearweightembedding1}\&\ref{it:bilinearweightembedding2}$\Rightarrow$\ref{it:bilinearweightembedding3}, note that
\begin{align*}
\|T_E(f,g)\|_{L^1_{W\otimes W^{-1}}(\Omega;\mc{H}\otimes\mc{H})}
&=\int_{\Omega}\|W(x)T_Ef(x)\|_{\mc{H}}\|W(x)^{-1}T_Eg(x)\|_{\mc{H}}\,\mathrm{d}\mu\\
&\leq\|T_Ef\|_{X_W}\|T_Eg\|_{(X')_{W^{-1}}},
\end{align*}
which shows that
\[
\|T_E\|_{X_W\times (X')_{W^{-1}}\to L^1_{W\otimes W^{-1}}(\Omega;\mc{H}\otimes\mc{H})}\leq\|T_E\|_{X_W\to X_W}\|T_E\|_{(X')_{W^{-1}}\to (X')_{W^{-1}}},
\]
as desired.
\end{proof}

\begin{remark}
When $\mc{H}=\F$, we have $L^1_{W\otimes W^{-1}}(\Omega;\F\otimes\F)\cong L^1(\Omega)$ under the canonical identification $\F\otimes\F=\F$. Now, in the implication \ref{it:bilinearweightembedding1}\&\ref{it:bilinearweightembedding2}$\Rightarrow$\ref{it:bilinearweightembedding3} we can use the fact that 
\[
\int_{\Omega}T_E(f,g)(x)\,\mathrm{d}\mu=\int_{\Omega}T_Ef(x)g(x)\,\mathrm{d}\mu
\]
to produce a bound using only $T_E:X_W\to X_W$. Thus, in this case we have
\[
\|T_E\|_{X_W\times (X')_{W^{-1}}\to L^1(\Omega)}=\|T_E\|_{X_W\to X_W}=\|T_E\|_{(X')_{W^{-1}}\to (X')_{W^{-1}}}.
\]
It is unclear whether this same equality holds for general $\mc{H}$ or not.
\end{remark}

\section{The multilinear matrix Muckenhoupt condition}
\label{sec:multilinear_muckenhoupt}
Let  $\vec{W}$ be an $m$-tuple of matrix weights and let $\vec{p}=(p_1,\ldots,p_m)\in[1,\infty]^m$. We set
\[
L^{\vec{p}}_{\vec{W}}(\R^d;\vec{\mc{H}}):=L^{p_1}_{W_1}(\R^d;\mc{H}_1)\times\cdots\times L^{p_m}_{W_m}(\R^d;\mc{H}_m),
\]
and, recalling our conventions, we have defined
\[
\mc{H}:=\bigotimes_{j=1}^m\mc{H}_j,\quad \mb{W}:=\bigotimes_{j=1}^m W_j:\R^d\to \mc{L}(\mc{H}),\quad\frac{1}{p}:=\sum_{j=1}^m\frac{1}{p_j}.
\]
We say that $\vec{W}$ satisfies the Muckenhoupt $A_{\vec{p}}$ condition and write $\vec{W}\in A_{\vec{p}}$, if $W_j^{-1}$ is locally $p_j'$-integrable for all $j=1,\ldots,m$ and
\[
[\vec{W}]_{\vec{p}}:=\sup_Q\|T_Q\|_{L^{\vec{p}}_{\vec{W}}(\R^d;\vec{\mc{H}})\to L^p_{\mb{W}}(\R^d;\mc{H})}<\infty,
\]
where the supremum is taken over all cubes $Q\subseteq\R^d$.

\begin{remark}
    The condition that $W_j^{-1}$ is locally $p_j'$-integrable ensures that any function $f\in L^{p_j}_{W_j}(\R^d;\mc{H}_j)$ is locally integrable.
\end{remark}

\begin{definition}
    Let $\vec{W}$ be a $m$-tuple of matrix weights. Let also $\vec{p}\in(0,\infty]^{m}$, $\vec{r}\in(0,\infty)^{m}$ and $s\in(\R\backslash\{0\})\cup\{\infty\}$ with $\vec{p}\geq\vec{r}$ (i.e.~$p_j\geq r_j$ for all $j=1,\ldots,m$) and $\frac{1}{p}\geq\frac{1}{s}$ (where by convention $\frac{1}{\infty}:=0$). Define
    \begin{equation*}
        \frac{1}{q}:=\frac{1}{p}-\frac{1}{s}\in[0,\infty),\quad \frac{1}{t_j}:=\frac{1}{r_j}-\frac{1}{p_j}\in[0,\infty)\quad \text{ for $j=1,\ldots,m$}.
    \end{equation*}
    We define
    \begin{equation*}
        [\vec{W}]_{\vec{p},(\vec{r},s),\mathrm{op}}:=\sup_Q\Big(\avint_Q\prod_{j=1}^m\Big(\avint_Q\!\|W_j(x)W_j(x_j)^{-1}\|_{\mc{H}_j\to\mc{H}_j}^{t_j}\,\mathrm{d}x_j\Big)^{\frac{q}{t_j}}\,\mathrm{d}x\Big)^{\frac{1}{q}},
    \end{equation*}
    where the supremum is taken over all cubes $Q\subseteq\R^d$, and an integral average is interpreted as an essential supremum whenever the associated exponent is infinite. In particular, if $\vec{p}\in[1,\infty]^m$, $\vec{r}=(1,\ldots,1)$ and $s=\infty$ (so that $q=p$ and $t_j=p_j'$ for $j=1,\ldots,m$), then we define
    \begin{equation*}
        [\vec{W}]_{\vec{p},\mathrm{op}}:=[\vec{W}]_{\vec{p},(\vec{r},s),\mathrm{op}}.
    \end{equation*}
\end{definition}

\begin{proposition}
\label{prop:equivalence_averages_operator}
Let $\vec{p}\in[1,\infty]^m$ and let $\vec{W}$ be an $m$-tuple of matrix weights. Then, we have $\vec{W}\in A_{\vec{p}}$ if and only if $[\vec{W}]_{\vec{p},\mathrm{op}}<\infty$. Moreover, in this case we have
\begin{equation*}
    [\vec{W}]_{\vec{p},\mathrm{op}}\lesssim_{\vec{n},m}[\vec{W}]_{\vec{p}}\leq[\vec{W}]_{\vec{p},\mathrm{op}}.
\end{equation*}
\end{proposition}

\begin{proof}
First suppose that $[\vec{W}]_{\vec{p},\text{op}}<\infty$. Then, by Proposition~\ref{prop:roudenko_type_characterization}, we have that $W_j^{-1}$ is locally $p_j'$-integrable for all $j=1,\ldots,m$. Moreover, we estimate
\begin{align*}
\|T_Q(\vec{f})\|_{L^p_{\mb{W}}(\R^d;\mc{H})}&=\Big\|\prod_{j=1}^m |W_j\langle W_j^{-1}W_jf_j\rangle_Q|\ind_Q\Big\|_{L^p(\R^d)}\\
&\leq \Big(\int_Q\prod_{j=1}^m\Big(\avint_Q\!\|W_j(x)W_j(x_j)^{-1}\|_{\mc{H}_j\to\mc{H}_j}\Vert W_j(x_j)f_j(x_j)\Vert_{\mc{H}_j}\,\mathrm{d}x_j\Big)^p\,\mathrm{d}x\Big)^{\frac{1}{p}}\\
&\leq \Big(\int_Q\prod_{j=1}^m\Big(\avint_Q\!\|W_j(x)W_j(x_j)^{-1}\|_{\mc{H}_j\to\mc{H}_j}^{p_j'}\,\mathrm{d}x_j\Big)^{\frac{p}{p_j'}}\langle\Vert W_jf_j\Vert_{\mc{H}_j}\rangle^p_{p_j,Q}\,\mathrm{d}x\Big)^{\frac{1}{p}}\\
&\leq [\vec{W}]_{\vec{p},\text{op}}|Q|^{\frac{1}{p}}\prod_{j=1}^m\langle \Vert W_jf_j\Vert_{\mc{H}_j}\rangle_{p_j,Q}
\leq[\vec{W}]_{\vec{p},\text{op}}\prod_{j=1}^m\|f_j\|_{L^{p_j}_{W_j}(\R^d;\mc{H}_j)}.
\end{align*}
Taking a supremum over all cubes $Q\subseteq\R^d$ proves that $\vec{W}\in A_{\vec{p}}$ with $[\vec{W}]_{\vec{p}}\leq[\vec{W}]_{\vec{p},\text{op}}$.

The converse follows immediately from Proposition~\ref{prop:reducingmatrixavop} coupled with Proposition~\ref{prop:roudenko_type_characterization}, \eqref{eq:reducing_matrix} and the observation that for any measure space $(\Omega,\mu)$ we have $K_{L^{p}(\Omega,\mu)}\leq 2^{m-1}$.
\end{proof}

The main result of this section is the following:
\begin{theorem}
    \label{thm:needed_reverse_Holder}
    Let $\vec{p}\in[1,\infty]^{m}$, and let $\vec{W}\in A_{\vec{p}}$. Then we have
    \[
        \sup_{v\in\mc{H}\setminus\{0\}}[\Vert\mathbf{W}v\Vert_{\mc{H}}]_{p}\leq[\vec{W}]_{\vec{p},\mathrm{op}},
    \]
    and
    \[
        \sup_{v\in\mc{H}_j\setminus\{0\}}[\Vert W_j^{-1}v\Vert_{\mc{H}_j}]_{p_j'}\lesssim_{m,\vec{n}}[\vec{W}]_{\vec{p},\mathrm{op}},\quad\forall j=1,\ldots,m.
    \]
    In particular, if $p<\infty$, we have
    \[
        [\mathbf{W}]_{\mathrm{FW}_{p}}:=\sup_{v\in\mc{H}\setminus\{0\}}[\Vert\mathbf{W}v\Vert_{\mc{H}}^{p}]_{\mathrm{FW}}^{\frac{1}{p}}\lesssim_{m,p}[\vec{W}]_{\vec{p},\mathrm{op}}
    \]
    and, if $p_j>1$,
    \[
        [W_{j}^{-1}]_{\mathrm{FW}_{p_j'}}:=\sup_{v\in\mc{H}_j\setminus\{0\}}[\Vert W_j^{-1}v\Vert_{\mc{H}_j}^{p'_j}]_{\mathrm{FW}}^{\frac{1}{p_j'}}\lesssim_{d,m,\vec{n},p_j}[\vec{W}]_{\vec{p},\mathrm{op}}
    \]
    for all $j=1,\ldots,m$.
\end{theorem}
We defer the proof to the end of this section.

The following result is a multilinear analogue of the result that the norms of an $A_p$ matrix weight applied to vectors, uniformly belong to the scalar $A_p$ class.
\begin{proposition}\label{prop:scalarweightsfrommatrix}
Let $\vec{p}\in[1,\infty]^m$ and let $\vec{W}$ be an $m$-tuple of matrix weights. If $\vec{W}\in A_{\vec{p}}$, then
\[
\vec{W}\vec{u}:=(\Vert W_1u_1\Vert_{\mc{H}_1},\ldots,\Vert W_mu_m\Vert_{\mc{H}_m})\in A_{\vec{p}}
\]
for all $u_1,\ldots,u_m$ with $u_j\in\mc{H}_j\backslash\{0\}$. Moreover, we have
\[
\sup_{\substack{u_1,\ldots,u_m\\ u_j\in\mc{H}_j\backslash\{0\}}}[\vec{W}\vec{u}]_{\vec{p}}\leq[\vec{W}]_{\vec{p}}.
\]
\end{proposition}

\begin{proof}
Observe that
\begin{equation*}
    \Vert u_j\Vert_{\mc{H}_j}^2\Vert W_ju_j\Vert_{\mc{H}_j}^{-1}\leq\Vert W_j^{-1}u_j\Vert_{\mc{H}_j},
\end{equation*}
so $\Vert W_ju_j\Vert_{\mc{H}_j}^{-1}$ is locally $p_j'$-integrable. Let $h_j\in L^{p_j}_{\Vert W_ju_j\Vert_{\mc{H}_j}}(\R^d)$. Then $f_j:=h_ju_j\in L^{p_j}_{W_j}(\R^d;\mc{H}_j)$ with 
\[
\|f_j\|_{L^{p_j}_{W_j}(\R^d;\mc{H}_j)}=\|h_j\|_{L^{p_j}_{\Vert W_ju_j\Vert_{\mc{H}_j}}(\R^d)},
\]
and
\[
T_Q(\vec{f})=\Big(\prod_{j=1}^m\langle h_j\rangle_Q\ind_Q\Big)\bigotimes_{j=1}^m u_j=T_Q(\vec{h})\bigotimes_{j=1}^m u_j.
\]
Hence, since $\Big\Vert\mb{W}\Big(\bigotimes_{j=1}^{m}u_j\Big)\Big\Vert_{\mc{H}}=\prod_{j=1}^m\Vert W_ju_j\Vert_{\mc{H}_j}=:W_{\vec{u}}$,
\begin{align*}
\|T_Q(\vec{h})\|_{L^p_{W_{\vec{u}}}(\R^d)}&=\|T_Q(\vec{f})\|_{L^p_{\mb{W}}(\R^d;\mc{H})}\leq[\vec{W}]_{\vec{p}}\prod_{j=1}^m\|h_j\|_{L^{p_j}_{\Vert W_ju_j\Vert_{\mc{H}_j}}(\R^d)},
\end{align*}
as desired.
\end{proof}

\begin{remark}
    \label{rem:first_scalar_ap}
    In view of the properties of the scalar multilinear Muckenhoupt condition \cite{LOPTT09}, Proposition~\ref{prop:scalarweightsfrommatrix} already implies that
    \begin{equation*}
        \Big[\Big\Vert\mathbf{W}\bigotimes_{j=1}^{m}u_j\Big\Vert_{\mc{H}}^{\frac{1}{m}}\Big]_{mp}^{m}\leq[\vec{W}]_{\vec{p}}
    \end{equation*}
    as well as that
    \begin{equation*}
        [\Vert W^{j}u_j\Vert_{\mc{H}_j}^{-\frac{1}{m}}]_{mp_j'}^{m}\leq[\vec{W}]_{\vec{p}},\quad j=1,\ldots,m
    \end{equation*}
    for all nonzero $u_j\in\mc{H}_j$, $j=1,\ldots,m$. This already implies that these weights satisfy some reverse H\"{o}lder inequalities. However, these do not appear to be sufficient to estimate the mutlilinear convex body sparse operators in Section~\ref{sec:mczo} below, and we need the full result of Theorem~\ref{thm:needed_reverse_Holder}.
\end{remark}

\begin{proposition}
\label{prop:linear_muck_tensor_product}
    Let $\vec{p}\in(0,\infty]^{m}$, $\vec{r}\in(0,\infty)^{m}$ and $s\in(\R\backslash\{0\})\cup\{\infty\}$ with $\vec{p}\geq\vec{r}$ and $\frac{1}{p}\geq\frac{1}{s}$. Let $\vec{W}$ be a $m$-tuple of matrix weights with $[\vec{W}]_{\vec{p},(\vec{r},s),\mathrm{op}}<\infty$. Then
    \begin{equation*}
        [\mathbf{W}]_{p,(r,s),\mathrm{op}}\leq[\vec{W}]_{\vec{p},(\vec{r},s),\mathrm{op}}.
    \end{equation*}
\end{proposition}

\begin{proof}
Like above, set $\tfrac{1}{q}:=\tfrac{1}{p}-\tfrac{1}{s}$, $\tfrac{1}{t_j}:=\tfrac{1}{r_j}-\tfrac{1}{p_j}$ for $j=1,\ldots,m$, and fix a cube $Q\subseteq\R^d$. Observe that
\begin{equation*}
    \frac{1}{t}=\sum_{j=1}^{m}\frac{1}{t_j}=\frac{1}{r}-\frac{1}{p}.
\end{equation*}
Thus, by H\"{o}lder's inequality for the exponents $t_j$, $j=1,\ldots,m$, we get
\begin{align*}
    &\Big(\avint_{Q}\prod_{j=1}^{m}\Big(\avint_{Q}\Vert W_j(y)W_j(x)^{-1}\Vert_{\mc{H}_j\to\mc{H}_j}^{t_j}\mathrm{d}x\Big)^{\frac{q}{t_j}}\mathrm{d}y\Big)^{\frac{1}{q}}\\
    &\geq\Big(\avint_{Q}\Big(\avint_{Q}\prod_{j=1}^{m}\Vert W_j(y)W_j(x)^{-1}\Vert_{\mc{H}_j\to\mc{H}_j}^{t}\mathrm{d}x\Big)^{\frac{q}{t}}\mathrm{d}y\Big)^{\frac{1}{q}}\\
    &=\Big(\avint_{Q}\Big(\avint_{Q}\Vert\mathbf{W}(y)\mathbf{W}(x)^{-1}\Vert_{\mc{H}\to\mc{H}}^{t}\mathrm{d}x\Big)^{\frac{q}{t}}\mathrm{d}y\Big)^{\frac{1}{q}},
\end{align*}
concluding the proof.
\end{proof}

We now prove a matrix weighted version of \cite[Lemma~3.2]{LMO18}. We use the notation similar to the one of \cite[Proposition~3.1.6]{Ni20}.

\begin{proposition}
\label{prop:linear_muck_each_weight}
    Let $\vec{p}\in(0,\infty]^{m}$, $\vec{r}\in(0,\infty)^{m}$ and $s\in(\R\backslash\{0\})\cup\{\infty\}$ with $\vec{p}\geq\vec{r}$ and $\frac{1}{p}\geq\frac{1}{s}$. Let $\vec{W}$ be a $m$-tuple of matrix weights with $[\vec{W}]_{\vec{p},(\vec{r},s),\mathrm{op}}<\infty$. Fix $j\in\{1,\ldots,m\}$. Set
    \begin{equation*}
        \frac{1}{\sigma_j}:=\frac{1}{r_j}-\Big(\frac{1}{r}-\frac{1}{s}\Big)\in\R.
    \end{equation*}
    Then
    \begin{equation*}
        [W_j]_{p_j,(r_j,\sigma_j),\mathrm{op}}\lesssim_{\vec{r},s,\vec{n}}[\vec{W}]_{\vec{p},(\vec{r},s),\mathrm{op}}
    \end{equation*}
    and
    \begin{equation*}
        [W_j^{-1}]_{\hat{p}_j,(r_j,\sigma_j),\mathrm{op}}\lesssim_{\vec{r},s,\vec{n}}[\vec{W}]_{\vec{p},(\vec{r},s),\mathrm{op}},
    \end{equation*}
    where $\frac{1}{\hat{p}_j}:=\frac{1}{r_j}+\frac{1}{\sigma_j}-\frac{1}{p_j}$.
\end{proposition}

For the proof we use the following lemma:
\begin{lemma}
    \label{lem:symmetry}
    Let $p\in(0,\infty]$, $r\in(0,\infty)$ and $s\in(\R\backslash\{0\})\cup\{\infty\}$ with $p\geq r$ and $\frac{1}{p}\geq\frac{1}{s}$. Assume that $[W]_{p,(r,s),\mathrm{op}}<\infty$. Set
    \begin{equation*}
        \frac{1}{\hat{p}}:=\frac{1}{r}+\frac{1}{s}-\frac{1}{p}.
    \end{equation*}
    Then, we have $[W^{-1}]_{\hat{p},(r,s),\mathrm{op}}<\infty$ and
    \begin{equation*}
        [W^{-1}]_{\hat{p},(r,s),\mathrm{op}}\eqsim_{r,s,n}\sup_{Q}\Vert A_{W,Q,\frac{1}{\frac{1}{p}-\frac{1}{s}}}A_{W^{-1},Q,\frac{1}{\frac{1}{r}-\frac{1}{p}}}\Vert_{\mc{H}\to\mc{H}}\eqsim_{r,s,n}[W]_{p,(r,s),\mathrm{op}}.
    \end{equation*}
\end{lemma}

\begin{proof}
     Observing that $\frac{1}{s}\leq\frac{1}{\hat{p}}\leq\frac{1}{r}$ and
    \begin{equation*}
        \frac{1}{\hat{p}}-\frac{1}{s}=\frac{1}{r}-\frac{1}{p}\quad\text{and}\quad\frac{1}{r}-\frac{1}{\hat{p}}=\frac{1}{p}-\frac{1}{s},
    \end{equation*}
    this follows from Proposition~\ref{prop:roudenko_type_characterization} (where we note that $K_{t}\leq 2^{(\frac{1}{r}-1)_+}\lesssim_r 1$ and, similarly, $K_q\lesssim_{r,s} 1$), and the self-adjointness of reducing operators.
\end{proof}

\begin{proof}[Proof of Proposition~\ref{prop:linear_muck_each_weight}]
Set $\tfrac{1}{q}:=\tfrac{1}{p}-\tfrac{1}{s}$, $\tfrac{1}{t_j}:=\tfrac{1}{r_j}-\tfrac{1}{p_j}$, and $\tfrac{1}{\kappa_j}:=\frac{1}{p_j}-\tfrac{1}{\sigma_j}$ for $j=1,\ldots,m$. Then $\kappa_j\in(0,\infty]$. Indeed, we have
\begin{equation*}
    \frac{1}{r}-\frac{1}{p}=\sum_{i=1}^{m}\Big(\frac{1}{r_i}-\frac{1}{p_i}\Big)\geq\frac{1}{r_j}-\frac{1}{p_j},
\end{equation*}
which implies that
\begin{equation*}
    \frac{1}{p_j}-\frac{1}{\sigma_j}=\frac{1}{r}-\frac{1}{s}-\Big(\frac{1}{r_j}-\frac{1}{p_j}\Big)\geq\frac{1}{r}-\frac{1}{s}-\Big(\frac{1}{r}-\frac{1}{p}\Big)=\frac{1}{p}-\frac{1}{s}\geq0.
\end{equation*}
Now fix a cube $Q\subseteq\R^d$ and $j\in\{1,\ldots,m\}$. By permuting the indices, we may assume without loss of generality that $j=1$. Then we first note that
    \[
        \Vert W_1(x)^{-1}W_1(y)\Vert_{\mc{H}_1\to\mc{H}_1}=\Vert (W_1(x)^{-1}W_1(y))\otimes\mathrm{id}_{\mc{H}_{2}}\ldots\otimes\mathrm{id}_{\mc{H}_m}\Vert_{\mc{H}\to\mc{H}}.
    \]
    Since $[\vec{W}]_{\vec{p},(\vec{r},s)}<\infty$, we have that $W_i^{-1}$ is $t_i$-integrable over $Q$. Denoting $A_i:=A_{W_i^{-1},Q,t_i}$ for $i=1,\ldots,m$ as well as $A^{\otimes}:=\bigotimes_{i=1}^{m}A_i$, we write
    \begin{align*}
        (W_1&(x)^{-1}W_1(y))\otimes\mathrm{id}_{\mc{H}_{2}}\ldots\otimes\mathrm{id}_{\mc{H}_m}\\
        &=(W_1(x)^{-1}A_1^{-1}A_1W_1(y))\otimes\Big(\bigotimes_{i=2}^{m}(W_i(y)^{-1}A_i^{-1}A_iW_i(y))\Big)\\
        &=(W_1(x)^{-1}A_1^{-1})\otimes\Big(\bigotimes_{i=2}^{m}(W_i(y)^{-1}A_i^{-1})\Big)\Big](A^{\otimes}\mathbf{W}(y)).
    \end{align*}
    Using first the submultiplicativity of the operator norm and then the fact that the operator norm of the tensor product of operators equals the product of the operator norms of the individual operators, we finally get
    \begin{align*}
        &\Vert W_1(x)^{-1}W_1(y)\Vert_{\mc{H}_1\to\mc{H}_1}\\
        &\leq\Vert W_1(x)^{-1}A_1^{-1}\Vert_{\mc{H}_1\to\mc{H}_1}\Big(\prod_{i=2}^{m}\Vert W_i(y)^{-1}A_i^{-1}\Vert_{\mc{H}_i\to\mc{H}_i}\Big)\Vert A^{\otimes}\mathbf{W}(y)\Vert_{\mc{H}\to\mc{H}}.
    \end{align*}
    Thus, we deduce
    \begin{align*}
        &\Big(\avint_{Q}\Big(\avint_{Q}\Vert W_1(x)^{-1}W_1(y)\Vert_{\mc{H}_1\to\mc{H}_1}^{t_1}\mathrm{d}x\Big)^{\frac{\kappa_1}{t_1}}\mathrm{d}y\Big)^{\frac{1}{\kappa_1}}\\
        &\leq\Big(\avint_{Q}\Vert W_1(x)^{-1}A_1^{-1}\Vert_{\mc{H}_1\to\mc{H}_1}^{t_1}\mathrm{d}x\Big)^{\frac{1}{t_1}}\\
        &\times \Big(\avint_{Q}\Big(\prod_{i=2}^{m}\Vert W_i(y)^{-1}A_i^{-1}\Vert_{\mc{H}_i\to\mc{H}_i}^{\kappa_1}\Big)\Vert A^{\otimes}\mathbf{W}(y)\Vert_{\mc{H}\to\mc{H}}^{\kappa_1}\mathrm{d}y\Big)^{\frac{1}{\kappa_1}}\\
        &\lesssim_{r_1,n_1}\Big(\avint_{Q}\Big(\prod_{i=2}^{m}\Vert W_i(y)^{-1}A_i^{-1}\Vert_{\mc{H}_i\to\mc{H}_i}^{\kappa_1}\Big)\Vert A^{\otimes}\mathbf{W}(y)\Vert_{\mc{H}\to\mc{H}}^{\kappa_1}\mathrm{d}y\Big)^{\frac{1}{\kappa_1}}.
    \end{align*}
    Now, we observe that
    \begin{align*}
        \sum_{i=2}^{m}\frac{1}{t_i}+\frac{1}{q}=\frac{1}{t}-\frac{1}{t_1}+\frac{1}{q}=\frac{1}{r}-\frac{1}{s}-\frac{1}{r_1}+\frac{1}{p_1}=\frac{1}{\kappa_1}.
    \end{align*}
    Thus, we apply H\"{o}lder's inequality for the exponents $t_2,\ldots,t_m,q$ to conclude that
    \begin{align*}
        &\Big(\avint_{Q}\Big(\avint_{Q}\Vert W_1(x)^{-1}W_1(y)\Vert_{\mc{H}_1\to\mc{H}_1}^{t_1}\mathrm{d}x\Big)^{\frac{\kappa_1}{t_1}}\mathrm{d}y\Big)^{\frac{1}{\kappa_1}}\\
        &\lesssim_{r_1,n_1}\Big(\prod_{i=2}^{m}\Big(\avint_{Q}\Vert W_i(y)^{-1}A_i^{-1}\Vert_{\mc{H}_i\to\mc{H}_i}^{t_i}\mathrm{d}y\Big)^{\frac{1}{t_i}}\Big)\Big(\avint_{Q}\Vert A^{\otimes}\mathbf{W}(y)\Vert_{\mc{H}\to\mc{H}}^{q}\mathrm{d}y\Big)^{\frac{1}{q}}\\
        &\lesssim_{\vec{r},s,\vec{n}}[\vec{W}]_{\vec{p},(\vec{r},s),\mathrm{op}},
    \end{align*}
    concluding the proof of the first estimate. The second estimate follows immediately from the first one coupled with Lemma~\ref{lem:symmetry}, proving the result.
    \end{proof}

We now want to show that the conclusions of Proposition~\ref{prop:linear_muck_tensor_product} and Proposition~\ref{prop:linear_muck_each_weight} together are actually equivalent to the condition of the weight $m$-tuple belonging to the multilinear matrix Muckenhoupt class.

\begin{theorem}
    \label{thm:multilinear_muckenhoupt_through_linear_muckenhoupt}
    Let $\vec{p}\in(0,\infty]^{m}$, $\vec{r}\in(0,\infty)^{m}$ and $s\in(\R\backslash\{0\})\cup\{\infty\}$ with $\vec{p}\geq\vec{r}$ and $\frac{1}{p}\geq\frac{1}{s}$. Let $\vec{W}$ be a $m$-tuple of matrix weights with $[\vec{W}]_{\vec{p},(\vec{r},s),\mathrm{op}}<\infty$. Set
    \begin{equation*}
        \frac{1}{\sigma_j}:=\frac{1}{r_j}-\Big(\frac{1}{r}-\frac{1}{s}\Big),\quad j=1,\ldots,m.
    \end{equation*}
   Then, the following are equivalent.
    \begin{enumerate}[(i)]
        \item\label{it:multilinear_muckenhoupt_through_linear_muckenhoupt1} $[\vec{W}]_{\vec{p},(\vec{r},s),\mathrm{op}}<\infty$;

        \item\label{it:multilinear_muckenhoupt_through_linear_muckenhoupt2} $[W_j]_{p_j,(r_j,\sigma_j),\mathrm{op}}<\infty$ for each $j=1,\ldots,m$ and $[\mathbf{W}]_{p,(r,s)}<\infty$.
    \end{enumerate}
    Moreover, in this case we have
    \begin{align*}
        \max([W_1]_{p_1,(r_1,\sigma_1),\mathrm{op}},\ldots,[W_m]_{p_m,(r_m,\sigma_m),\mathrm{op}},[\mathbf{W}]_{p,(r,s),\mathrm{op}})\lesssim_{\vec{r},s,\vec{n}}[\vec{W}]_{\vec{p},(\vec{r},s),\mathrm{op}}
    \end{align*}
    and
    \begin{align*}
        [\vec{W}]_{\vec{p},(\vec{r},s),\mathrm{op}}\lesssim_{\vec{n},\vec{r},s,m}[\mathbf{W}]_{p,(r,s),\mathrm{op}}\prod_{j=1}^{m}[W]_{p_j,(r_j,\sigma_j),\mathrm{op}}.
    \end{align*}
\end{theorem}

\begin{proof}
     
    The implication \ref{it:multilinear_muckenhoupt_through_linear_muckenhoupt1}$\Rightarrow$\ref{it:multilinear_muckenhoupt_through_linear_muckenhoupt2} together with the first inequality follow from Proposition~\ref{prop:linear_muck_tensor_product} and Proposition~\ref{prop:linear_muck_each_weight}. Thus, it remains to show that \ref{it:multilinear_muckenhoupt_through_linear_muckenhoupt2}$\Rightarrow$\ref{it:multilinear_muckenhoupt_through_linear_muckenhoupt1} together with the second inequality.

    Assume that $[W_j]_{p_j,(r_j,\sigma_j),\mathrm{op}}<\infty$ for each $j=1,\ldots,m$ and $[\mathbf{W}]_{p,(r,s)}<\infty$. We set $\tfrac{1}{\ell}:=m\big(\tfrac{1}{r}-\tfrac{1}{s}\big)$, $\tfrac{1}{q}:=\tfrac{1}{p}-\tfrac{1}{s}$, $\tfrac{1}{t_j}:=\tfrac{1}{r_j}-\tfrac{1}{p_j}$, and $\tfrac{1}{\kappa_j}:=\frac{1}{p_j}-\tfrac{1}{\sigma_j}$ for $j=1,\ldots,m$.
Fix a cube $Q\subseteq\R^d$. Then, we have
   \begin{align*}
   		&\Big(\avint_{Q}\prod_{j=1}^{m}\Big(\avint_{Q}\Vert W_j(y)W_j(x)^{-1}\Vert_{\mc{H}_j\to\mc{H}_j}^{t_j}\mathrm{d}x\Big)^{\frac{q}{t_j}}\mathrm{d}y\Big)^{\frac{\ell}{q}}\\
   		&=\avint_{Q}\Big(\avint_{Q}\prod_{j=1}^{m}\Big(\avint_{Q}\Vert W_j(y)W_j(x)^{-1}\Vert_{\mc{H}_j\to\mc{H}_j}^{t_j}\mathrm{d}x\Big)^{\frac{q}{t_j}}\mathrm{d}y\Big)^{\frac{\ell}{q}}\mathrm{d}z\\
   		&\leq\avint_{Q}\Big(\avint_{Q}\prod_{j=1}^{m}\Big(\avint_{Q}\Vert W_j(y)W_j(z)^{-1}\Vert_{\mc{H}_j\to\mc{H}_j}^{t_j}\Vert W_j(z)W_j(x)^{-1}\Vert_{\mc{H}_j\to\mc{H}_j}^{t_j}\mathrm{d}x\Big)^{\frac{q}{t_j}}\mathrm{d}y\Big)^{\frac{\ell}{q}}\mathrm{d}z\\
   		&=\avint_{Q}\Big(\avint_{Q}\Vert\mathbf{W}(y)\mathbf{W}(z)^{-1}\Vert_{\mc{H}\to\mc{H}}^{q}\mathrm{d}y\Big)^{\frac{\ell}{q}}\prod_{j=1}^{m}\Big(\avint_{Q}\Vert W_j(z)W_j(x)^{-1}\Vert_{\mc{H}_j\to\mc{H}_j}^{t_j}\mathrm{d}x\Big)^{\frac{\ell}{t_j}}\mathrm{d}z.
   \end{align*}
   Notice that
   \begin{equation*}
   		\sum_{j=1}^{m}\frac{1}{\kappa_j}+\frac{1}{t}=\frac{1}{p}+\frac{1}{t}-\frac{1}{r}+m\Big(\frac{1}{r}-\frac{1}{s}\Big)=\frac{1}{\ell}.
   \end{equation*}
   Thus, by applying H\"older's inequality with the exponents $\tfrac{t}{\ell},\tfrac{\kappa_1}{\ell},\ldots,\tfrac{\kappa_m}{\ell}$, we deduce
   \begin{align*}
   		&\Big(\avint_{Q}\prod_{j=1}^{m}\Big(\avint_{Q}\Vert W_j(y)W_j(x)^{-1}\Vert_{\mc{H}_j\to\mc{H}_j}^{t_j}\mathrm{d}x\Big)^{\frac{q}{t_j}}\mathrm{d}y\Big)^{\frac{\ell}{q}}\\
   		&\leq\Big(\avint_{Q}\Big(\avint_{Q}\Vert\mathbf{W}(y)\mathbf{W}(z)^{-1}\Vert_{\mc{H}\to\mc{H}}^{q}\mathrm{d}y\Big)^{\frac{t}{q }}\mathrm{d}z\Big)^{\frac{\ell}{t}}\\
   		&\times\prod_{j=1}^{m}\Big(\avint_{Q}\Big(\avint_{Q}\Vert W_j(z)W_j(x)^{-1}\Vert_{\mc{H}_j\to\mc{H}_j}^{t_j}\mathrm{d}x\Big)^{\frac{\kappa_j}{t_j}}\mathrm{d}z\Big)^{\frac{\ell}{\kappa_j}}\\
   		&\leq[\mathbf{W}^{-1}]_{\hat{p},(r,s),\mathrm{op}}^{\ell}\prod_{j=1}^{m}[W_j]_{p_j,(r_j,\sigma_j),\mathrm{op}}^{\ell},
   \end{align*}
   where
   \begin{equation*}
   		\frac{1}{\hat{p}}:=\frac{1}{r}+\frac{1}{s}-\frac{1}{p}.
   \end{equation*}
   Thus, we have shown that
   \begin{equation*}
   		[\vec{W}]_{\vec{p},(\vec{r},s),\mathrm{op}}\leq[\mathbf{W}^{-1}]_{\hat{p},(r,s),\mathrm{op}}\prod_{j=1}^{m}[W_j]_{p_j,(r_j,\sigma_j),\mathrm{op}}.
   \end{equation*}
   An appeal to Lemma~\ref{lem:symmetry} yields $[\mathbf{W}^{-1}]_{\hat{p},(r,s),\mathrm{op}}\eqsim_{n,r,s}[\mathbf{W}]_{p,(r,s),\mathrm{op}}$, concluding the proof.
   \end{proof}

\begin{proposition}
    \label{prop:muck_matrix_to_scalar}
    Let $p\in(0,\infty]$, $r\in(0,\infty)$ and $s\in(\R\backslash\{0\})\cup\{\infty\}$ with $p\geq r$, $\frac{1}{p}\geq\frac{1}{s}$ and $r\neq s$. Assume that $[W]_{p,(r,s),\mathrm{op}}<\infty$. Then, for all $v\in\mc{H}\setminus\{0\}$ we have
    \begin{equation*}
        [\Vert W v\Vert_{\mc{H}}^{\rho}]^{\frac{1}{\rho}}_{\kappa}\leq [W]_{p,(r,s),\mathrm{op}},
    \end{equation*}
    where
    \begin{equation*}
        \frac{1}{\kappa}:=\frac{\frac{1}{p}-\frac{1}{s}}{\frac{1}{r}-\frac{1}{s}}\in[0,1]\quad\text{and}\quad\frac{1}{\rho}:=\frac{1}{r}-\frac{1}{s}.
    \end{equation*}
\end{proposition}

\begin{proof}
    Set
    \begin{equation*}
        \frac{1}{t}:=\frac{1}{r}-\frac{1}{p},\quad\frac{1}{q}:=\frac{1}{p}-\frac{1}{s}
    \end{equation*}
    and observe that
    \begin{equation*}
        \rho\kappa'=t,\quad\rho\kappa=q.
    \end{equation*}
    To begin with, we prove the following estimate for measurable functions $f:\R^d\to\mc{H}$:
    \begin{equation}
        \label{eq:vector_valued}
        \Big(\avint_{Q}\langle\Vert W(x)f\Vert_{\mc{H}}^{\rho}\rangle_{Q}^{\kappa}\mathrm{d}x\Big)^{\frac{1}{\kappa}}\leq [W]_{p,(r,s),\mathrm{op}}\Big(\avint_{Q}\Vert W(x)f(x)\Vert_{\mc{H}}^{\rho\kappa}\mathrm{d}x\Big)^{\frac{1}{\kappa}}.
    \end{equation}
    We estimate
    \begin{align*}
        &\Big(\avint_{Q}\langle\Vert W(x)f\Vert_{\mc{H}}^{\rho}\rangle_{Q}^{\kappa}\mathrm{d}x\Big)^{\frac{1}{\kappa}}
        =\Big(\avint_{Q}\Big(\avint_{Q}\Vert W(x)f(y)\Vert_{\mc{H}}^{\rho}\mathrm{d}y\Big)^{q}\mathrm{d}x\Big)^{\frac{1}{\kappa}}\\
        &\leq\Big(\avint_{Q}\Big(\avint_{Q}\Vert W(x)W(y)^{-1}\Vert_{\mc{H}\to\mc{H}}^{\rho}\cdot\Vert W(y)f(y)\Vert_{\mc{H}}^{\rho}\mathrm{d}y\Big)^{\kappa}\mathrm{d}x\Big)^{\frac{1}{\kappa}}\\
        &\leq\Big(\avint_{Q}\Big(\avint_{Q}\Vert W(x)W(y)^{-1}\Vert_{\mc{H}\to\mc{H}}^{\rho\kappa'}\mathrm{d}y\Big)^{\frac{\kappa}{\kappa'}}\Big(\avint_{Q}\Vert W(y)f(y)\Vert_{\mc{H}}^{\rho\kappa}\mathrm{d}y\Big)\mathrm{d}x\Big)^{\frac{1}{\kappa}}\\
        &\leq [W]_{p,(r,s),\mathrm{op}}^{\rho}\Big(\avint_{Q}\Vert W(x)f(x)\Vert_{\mc{H}}{\rho\kappa}\mathrm{d}x\Big)^{\frac{1}{\kappa}}.
    \end{align*}
    Next, considering the scalar weight $w:=\Vert W v\Vert_{\mc{H}}^{\rho}$, which by Proposition~\ref{prop:roudenko_type_characterization} is locally $\kappa$-integrable, for any measurable function $f:\R^d\to[0,\infty)$ we test \eqref{eq:vector_valued} on the vector valued function $f^{\frac{1}{\rho}}v$, and we get
    \begin{align*}
        \Big(\avint_{Q}\langle\Vert W(x)f^{\frac{1}{\rho}}v\Vert_{\mc{H}}^{\rho}\rangle_{Q}^{\kappa}\mathrm{d}x\Big)^{\frac{1}{\kappa}}\leq [W]_{p,(r,s),\mathrm{op}}^{\rho}\Big(\avint_{Q}\Vert W(x)f(x)^{\frac{1}{\rho}}v\Vert_{\mc{H}}^{\rho\kappa}\mathrm{d}x\Big)^{\frac{1}{\kappa}},
    \end{align*}
    that is
    \begin{equation*}
        \Big(\avint_{Q}(\langle f\rangle_{Q}w(x))^{\kappa}\mathrm{d}x\Big)^{\frac{1}{\kappa}}\leq [W]_{p,(r,s),\mathrm{op}}^{\rho}\Big(\avint_{Q}(f(x)w(x))^{\kappa}\mathrm{d}x\Big)^{\frac{1}{\kappa}}.
    \end{equation*}
    In other words, we have shown that
    \begin{equation}
        \label{eq:scalar_valued}
        \Vert T_{Q}\Vert_{L^{\kappa}_{w}\Big(Q,\frac{\mathrm{d}x}{|Q|}\Big)\to L^{\kappa}_{w}\Big(Q,\frac{\mathrm{d}x}{|Q|}\Big)}\leq [W]_{p,(r,s),\mathrm{op}}^{\rho}<\infty.
    \end{equation}
    Since $\kappa\in[1,\infty]$ and the scalar weight $w$ is $q$-integrable over $Q$, it is a classical fact that
    \begin{equation*}
        \langle w^{\kappa}\rangle_{Q}^{\frac{1}{\kappa}}\langle w^{-\kappa'}\rangle_{Q}^{\frac{1}{\kappa'}}\leq\Vert T_{Q}\Vert_{L^{\kappa}_{w}\Big(Q,\frac{\mathrm{d}x}{|Q|}\Big)\to L^{\kappa}_{w}\Big(Q,\frac{\mathrm{d}x}{|Q|}\Big)},
    \end{equation*}
    yielding the required result.
\end{proof}

\begin{remark}
    Referring to the proof of Proposition~\ref{prop:muck_matrix_to_scalar}, we note that
    \begin{equation*}
        w^{-1}\leq\frac{\Vert W^{-1}v\Vert_{\mc{H}}^{\rho}}{\Vert u\Vert_{\mc{H}}^{2\rho}}\quad\text{a.e.~on }\R^d.
    \end{equation*}
    Since $W^{-1}$ is locally $t$-integrable by Proposition~\ref{prop:roudenko_type_characterization}, we have that $w^{-1}$ is locally $\kappa'$-integrable.
\end{remark}

\begin{corollary}
    \label{cor:needed_reverse_Holder}
    Let $\vec{p}\in(0,\infty]^{m}$, $\vec{r}\in(0,\infty)^{m}$ and $s\in(\R\backslash\{0\})\cup\{\infty\}$ with $\vec{p}\geq\vec{r}$, $\frac{1}{p}\geq\frac{1}{s}$. and $r\neq s$. Let $\vec{W}$ be a $m$-tuple of matrix weights with $[\vec{W}]_{\vec{p},(\vec{r},s),\mathrm{op}}<\infty$. Set
    \begin{equation*}
        \frac{1}{q}:=\frac{1}{p}-\frac{1}{s},\quad\frac{1}{\kappa}:=\frac{\frac{1}{p}-\frac{1}{s}}{\frac{1}{r}-\frac{1}{s}},\quad\frac{1}{\rho}:=\frac{1}{r}-\frac{1}{s},
    \end{equation*}
    and
    \begin{equation*}
        \frac{1}{t_j}:=\frac{1}{r_j}-\frac{1}{p_j},\quad \lambda_j:=t_j\Big(\frac{1}{r}-\frac{1}{s}\Big),\quad j=1,\ldots,m.
    \end{equation*}
    Then we have
    \begin{equation}
        \label{eq:ap_tensor}
        \sup_{v\in\mc{H}\setminus\{0\}}[\Vert\mathbf{W}v\Vert_{\mc{H}}^{\rho}]^{\frac{1}{\rho}}_{\kappa}\leq[\vec{W}]_{\vec{p},(\vec{r},s),\mathrm{op}},
    \end{equation}
    and
    \begin{equation}
        \label{eq:ap_each}
        \sup_{v\in\mc{H}_j\setminus\{0\}}[\Vert W_j^{-1}v\Vert_{\mc{H}_j}^{\rho}]^{\frac{1}{\rho}}_{\lambda_j}\lesssim_{\vec{r},s,\vec{n}}[\vec{W}]_{\vec{p},(\vec{r},s),\mathrm{op}},\quad\forall j=1,\ldots,m.
    \end{equation}
In particular, we have
    \begin{equation}
        \label{eq:rh_tensor}
        [\mathbf{W}]_{\mathrm{FW}_q}:=\sup_{v\in\mc{H}\setminus\{0\}}[\Vert\mathbf{W}v\Vert_{\mc{H}}^{q}]_{\mathrm{FW}}^{\frac{1}{q}}\lesssim_{r,s,p}[\vec{W}]_{\vec{p},(\vec{r},s),\mathrm{op}}\quad\text{if }p\neq s
    \end{equation}
    and
    \begin{equation}
        \label{eq:rh_each}
        [W_{j}^{-1}]_{\mathrm{FW}_{t_j}}:=\sup_{v\in\mc{H}_j\setminus\{0\}}[\Vert W_j^{-1}v\Vert_{\mc{H}_j}^{t_j}]_{\mathrm{FW}}^{\frac{1}{t_j}}\lesssim_{\vec{r},s,\vec{n},p_j}[\vec{W}]_{\vec{p},(\vec{r},s),\mathrm{op}}\quad\text{ if $p_j\neq r_j$}
    \end{equation}
    for all $j=1,\ldots,m$.
\end{corollary}

\begin{proof}
    Estimate \eqref{eq:ap_tensor} follows immediately from Proposition~\ref{prop:linear_muck_tensor_product} combined with Proposition~\ref{prop:muck_matrix_to_scalar}. Estimate \eqref{eq:ap_each} follows immediately from Proposition~\ref{prop:linear_muck_each_weight} combined with Proposition~\ref{prop:muck_matrix_to_scalar} (since in the notation of Proposition~\ref{prop:linear_muck_each_weight} we have $\sigma_j\neq r_j$ for each $j=1,\ldots,m$).
    
    Finally, since $\rho\kappa=q$ and $\lambda_j\rho=t_j$ for all $j=1,\ldots,m$, \eqref{eq:rh_tensor} and \eqref{eq:rh_each} follow from \eqref{eq:ap_tensor}, respectively \eqref{eq:ap_each}, through the  fact that if $\alpha\in[1,\infty)$ and $w$ is a scalar weight, then
    \begin{equation*}
        [w^\alpha]^{\frac{1}{\alpha}}_{\text{FW}}\lesssim_\alpha[w]_\alpha,
    \end{equation*}
    see, e.g., \cite{HP13}. The result follows.
\end{proof}

\begin{remark}\label{rem:sharppdendence}
If one cares about dependence on $\vec{p}$, one can use the fact that if $\alpha\leq\beta$ and $w$ is a scalar weight, then
    \begin{equation*}
        \Big[w^{-\frac{1}{\frac{1}{\alpha}-\frac{1}{\beta}}}\Big]^\frac{1}{\beta}_{\mathrm{FW}}\leq e^{\frac{1}{\alpha}}[w^\alpha]_{\frac{\beta}{\alpha}}^{\frac{\frac{1}{\alpha}\frac{1}{\beta}}{\frac{1}{\beta}-\frac{1}{\alpha}}},
    \end{equation*}
    see \cite[Proposition~3.3.3]{Ni20}. If one applies this in the final estimate above respectively with $\alpha=\rho$, $\tfrac{1}{\beta}=\tfrac{1}{r}-\tfrac{1}{p}$, $w=\|Wv\|_{\mc{H}}^{-\rho}$, and $\alpha=\rho$, $\beta=t_j$, $w=\|W_j^{-1}v\|_{\mc{H}_j}$, we obtain  
\[
        \sup_{v\in\mc{H}\setminus\{0\}}[\Vert\mathbf{W}v\Vert_{\mc{H}}^{q}]_{\mathrm{FW}}^{\frac{1}{r}-\frac{1}{p}}\lesssim_{r,s}[\vec{W}]_{\vec{p},(\vec{r},s),\mathrm{op}}^{\frac{\frac{1}{r}-\frac{1}{p}}{\frac{1}{p}-\frac{1}{s}}}\quad\text{if }p\neq s
    \]
    and
    \[
        [W_j^{-1}]_{\mathrm{FW}_{\frac{1}{\frac{1}{r_j}-\frac{1}{p_j}}}}^{\frac{\frac{1}{p_j}}{\frac{1}{r_j}-\frac{1}{p_j}}}=\sup_{v\in\mc{H}_j\setminus\{0\}}[\Vert W_j^{-1}v\Vert_{\mc{H}_j}^{t_j}]_{\mathrm{FW}}^{\frac{1}{p_j}}\lesssim_{\vec{r},s,\vec{n}}[\vec{W}]_{\vec{p},(\vec{r},s),\mathrm{op}}^{\frac{t_j}{p_j}}\quad\text{ if $p_j\neq r_j$}
    \]
    for all $j=1,\ldots,m$. Setting $\vec{r}=(1,\ldots,1)$, $s=\infty$, this yields
    \[
        \sup_{v\in\mc{H}\setminus\{0\}}[\Vert\mathbf{W}v\Vert_{\mc{H}}^{p}]_{\mathrm{FW}}^{m-\frac{1}{p}}\lesssim_m[\vec{W}]_{\vec{p},\mathrm{op}}^{pm-1}\quad\text{if }p<\infty
    \]
    and
    \[
        \sup_{v\in\mc{H}_j\setminus\{0\}}[\Vert W_j^{-1}v\Vert_{\mc{H}_j}^{p'_j}]_{\mathrm{FW}}^{\frac{1}{p_j}}\lesssim_{\vec{r},s,\vec{n}}[\vec{W}]_{\vec{p},\mathrm{op}}^{\frac{p'_j}{p_j}}\quad\text{ if $p_j>1$}
    \]
    for all $j=1,\ldots,m$. As these last terms are exactly the ones that appears in Theorem~\ref{thm:B} (and Theorem~\ref{thm:strong_type_maximal_general} below), the dependence on $\vec{p}$ can be determined more precisely this way.
\end{remark}

\begin{proof}[Proof of Theorem~\ref{thm:needed_reverse_Holder}]
The result follows from Corollary~\ref{cor:needed_reverse_Holder} applied with $\vec{r}=(1,\ldots,1)$ and $s=\infty$.
\end{proof}

\section{The tensor product maximal operator}
\label{sec:tensor_maximal}

Given $\vec{F}=(F_1,\ldots,F_m)$ with $F_j\in L^1_{\text{loc}}(\R^d;\mc{K}(\mc{H}_j))$ and a collection of cubes $\mc{P}$ we define
\[
M^{\mc{K}}_{\mc{P}}\vec{F}(x):=\mc{K}\Big(\bigcup_{Q\in\mc{P}}\bigotimes_{j=1}^m\langle F_j\rangle_Q\ind_Q(x)\Big)
\]
and
\[
M_{\mc{P}}\vec{F}(x):=\overline{\bigcup_{Q\in\mc{P}}\mc{K}\Big(\bigotimes_{j=1}^m\langle F_j\rangle_Q\Big)\ind_Q(x)}
\]
Moreover, we drop the index $\mc{P}$ when the collection consists of all cubes.

\subsection{Weak type bounds}

We define $L^p_{\mb{W}}(\R^d;\mc{H})_{\text{weak}}$ by
\[
\|F\|_{L^p_{\mb{W}}(\R^d;\mc{H})_{\text{weak}}}:=\sup_{u\in\mc{H}}\|\ind_{\{x\in\R^d:u\in F(x)\}}u\|_{L^p_{\mb{W}}(\R^d;\mc{H})}.
\]
As $\ind_{\{x\in\R^d:u\in F(x)\}}u\in S^0(\R^d;F)$, any $F\in L^p_{\mb{W}}(\R^d;\mc{H})$ belongs to $L^p_{\mb{W}}(\R^d;\mc{H})_{\text{weak}}$ with
\[
\|F\|_{L^p_{\mb{W}}(\R^d;\mc{H})_{\text{weak}}}\leq\|F\|_{L^p_{\mb{W}}(\R^d;\mc{H})}.
\]
The following result shows that the weak-type boundedness of $M$ characterizes the multilinear matrix Muckenhoupt condition.
\begin{proposition}\label{prop:weaktypechar}
Let $\vec{p}\in[1,\infty]^m$ and let $\vec{W}$ be matrix weights. Then the following are equivalent:
\begin{enumerate}[(i)]
    \item\label{it:weaktypechar1} $\vec{W}\in A_{\vec{p}}$;
    \item\label{it:weaktypechar2} $M:L^{\vec{p}}_{\vec{W}}(\R^d;\mc{K}(\vec{\mc{H}}))\to L^p_{\mb{W}}(\R^d;\mc{K}(\mc{H}))_{\emph{weak}}$.
\end{enumerate}
Moreover, in this case we have
\[
[\vec{W}]_{\vec{p}}\eqsim_{d,n,m}\|M\|_{L^{\vec{p}}_{\vec{W}}(\R^d;\mc{K}(\vec{\mc{H}}))\to L^p_{\mb{W}}(\R^d;\mc{K}(\mc{H}))_{\emph{weak}}}.
\]
\end{proposition}
For the proof, we require the following lemma:
\begin{lemma}\label{lem:strongmuckenhouptconvex}
Let $\vec{p}\in[1,\infty]^m$, $\vec{W}\in A_{\vec{p}}$, $\vec{F}\in L^{\vec{p}}_{\vec{W}}(\R^d;\mc{K}(\vec{\mc{H}}))$, and let $\mc{P}$ be a pairwise disjoint collection of cubes in a dyadic grid $\mc{D}$. Then
\[
\Big\|\sum_{Q\in\mc{P}}\mc{K}\Big(\bigotimes_{j=1}^m\langle F_j\rangle_Q\Big)\ind_Q\Big\|_{L^p_{\mb{W}}(\R^d;\mc{K}(\mc{H}))}\lesssim_{n,m}[\vec{W}]_{\vec{p}}\prod_{j=1}^m\|F_j\|_{L^{p_j}_{W_j}(\R^d;\mc{K}(\mc{H}_j))}.
\]
\end{lemma}
\begin{proof}
We set $T_{\mc{P}}\vec{F}:=\sum_{Q\in\mc{P}_k}\mc{K}\big(\bigotimes_{j=1}^m\langle F_j\rangle_Q\big)\ind_Q$ and let $g\in S^0(\R^d;T_{\mc{P}}\vec{F})$. Then $g=\sum_{Q\in\mc{P}}g_Q$, where $g_Q\in S^0(\R^d;\mc{K}(\bigotimes_{j=1}^m\langle F_j\rangle_Q\ind_Q))$. By the John ellipsoid theorem, there is an $A_Q\in\mc{L}(\mc{H})$ such that
\begin{equation}\label{eq:johnellipsoidpairwisedisjoint}
A_Q\overline{B}\subseteq \mc{K}\Big(\bigotimes_{j=1}^m\langle F_j\rangle_Q\ind_Q\Big)\subseteq n^{\frac{1}{2}}A_Q\overline{B},
\end{equation}
where $\overline{B}$ is the closed unit ball in $\mc{H}=\bigotimes_{j=1}^m\mc{H}_j$. Choosing an orthonormal basis $(e_k)_{k=1}^{n_j}$ for each $\mc{H}_j$, for $\vec{k}=(k_1,\ldots,k_m)$, $k_j\in\{1,\ldots,n_j\}$, we write $e_{\vec{k}}:=\bigotimes_{j=1}^m e_{k_j}$, and set $v_{\vec{k}}^Q:=A_Q e_{\vec{k}}$. By the first inclusion in \eqref{eq:johnellipsoidpairwisedisjoint} and Carath\'eodory's theorem on convex hulls, we have
\[
v_{\vec{k}}^Q=\sum_{i=1}^{2n+1}\theta_{i,\vec{k}}^Q\bigotimes_{j=1}^m\langle f_{\vec{k},i,j}^Q\rangle_Q,
\]
with $\theta_{i,\vec{k}}^Q\geq 0$, $\sum_{i=1}^{2n+1}\theta_{i,\vec{k}}^Q=1$, and $f_{\vec{k},i,j}^Q\in S^0(\R^d;F_j\ind_Q)$. Now, since $g_Q(x)\in n^{\frac{1}{2}} A_Q\overline{B}$ by the second inclusion in \eqref{eq:johnellipsoidpairwisedisjoint}, it is of the form
\[
g_Q(x)=\sum_{\vec{k}}h_{\vec{k}}^Q(x)v_{\vec{k}}^Q=\sum_{\vec{k}}h_{\vec{k}}^Q(x)\sum_{i=1}^{2n+1}\theta_{i,\vec{k}}^Q\bigotimes_{j=1}^m\langle f_{\vec{k},i,j}^Q\rangle_Q\ind_Q(x),
\]
for $h^Q_{\vec{k}}$ satisfying $\Big(\sum_{\vec{k}}|h^Q_{\vec{k}}(x)|^2\Big)^{\frac{1}{2}}\leq n^{\frac{1}{2}}$. Now, setting $h_{\vec{k}}:=\sum_{Q\in\mc{P}}h^Q_{\vec{k}}\ind_Q$, it follows from the quasi-triangle inequality that
\begin{align*}
\|g\|_{L^p_{\mb{W}}(\R^d;\vec{\mc{H}})}&=\Big(\int_{\R^d}\Big|\sum_{\vec{k}}h_{\vec{k}}(x)\sum_{Q\in\mc{P}}\sum_{i=1}^{2n+1}\theta_{i,\vec{k}}^Q\bigotimes_{j=1}^m W_j(x)\langle f_{\vec{k},i,j}^Q\rangle_Q\ind_Q\Big|^p\,\mathrm{d}x\Big)^{\frac{1}{p}}\\
&\leq 2^{(n-1)(\frac{1}{p}-1)_+} n^{\frac{1}{2}}\sum_{\vec{k}}\Big(\int_{\R^d}\Big|\sum_{Q\in\mc{P}}\sum_{i=1}^{2n+1}\theta_{i,\vec{k}}^Q\bigotimes_{j=1}^m W_j(x)\langle f_{\vec{k},i,j}^Q\rangle_Q\ind_Q\Big|^p\,\mathrm{d}x\Big)^{\frac{1}{p}}\\
&\lesssim_{n,m}\sum_{\vec{k}}\Big(\sum_{Q\in\mc{P}}\int_Q\Big|\sum_{i=1}^{2n+1}\theta_{i,\vec{k}}^Q\bigotimes_{j=1}^m W_j(x)\langle f_{\vec{k},i,j}^Q\rangle_Q\Big|^p\,\mathrm{d}x\Big)^{\frac{1}{p}}\\
&\lesssim_{n,m}\sum_{\vec{k}}\Big(\sum_{Q\in\mc{P}}\sum_{i=1}^{2n+1}(\theta_{i,\vec{k}}^Q)^p\|T_Q\vec{f}_{\vec{k},i}^Q\|_{L^p_{\mb{W}}(\R^d)}^p\Big)^{\frac{1}{p}}\\
&\leq[\vec{W}]_{\vec{p}}\sum_{\vec{k}}\Big(\sum_{Q\in\mc{P}}\sum_{i=1}^{2n+1}(\theta_{i,\vec{k}}^Q)^p\prod_{j=1}^m\Big(\int_Q\!|W_jF_j|^{p_j}\,\mathrm{d}x\Big)^{\frac{p}{p_j}}\Big)^{\frac{1}{p}}\\
&\lesssim_n (2n+1)^{\frac{1}{p}}[\vec{W}]_{\vec{p}}\Big(\sum_{Q\in\mc{P}}\prod_{j=1}^m\Big(\int_Q\!|W_jF_j|^{p_j}\,\mathrm{d}x\Big)^{\frac{p}{p_j}}\Big)^{\frac{1}{p}}\\
&\lesssim_{n,m}[\vec{W}]_{\vec{p}}\prod_{j=1}^m\Big(\sum_{Q\in\mc{P}}\int_Q\!|W_jF_j|^{p_j}\,\mathrm{d}x\Big)^{\frac{1}{p_j}}\leq[\vec{W}]_{\vec{p}}\prod_{j=1}^m\|F_j\|_{L^{p_j}_{W_j}(\R^d;\mc{K}(\mc{H}_j))}.
\end{align*}
Taking a supremum over all $g\in S^0(\R^d;T_{\mc{P}}\vec{F})$, the assertion follows.
\end{proof}

\begin{proof}[Proof of Proposition~\ref{prop:weaktypechar}]
For \ref{it:weaktypechar2}$\Rightarrow$\ref{it:weaktypechar1}, note that $T_Q\vec{f}$ is a selection of $M(\mc{K}(\vec{f}))$. Hence,
\begin{align*}
\|T_Q\vec{f}\|_{L^p_{\mb{W}}(\R^d;\mc{H})}&=\|T_Q\vec{f}\|_{L^p_{\mb{W}}(\R^d;\mc{H})_{\text{weak}}}\leq\|M(\mc{K}(\vec{f}))\|_{L^p_{\mb{W}}(\R^d;\mc{K}(\mc{H}))_{\text{weak}}}\\
&\leq \|M\|_{L^{\vec{p}}_{\vec{W}}(\R^d;\mc{K}(\vec{\mc{H}}))\to L^p_{\mb{W}}(\R^d;\mc{K}(\mc{H}))_{\emph{weak}}}\prod_{j=1}^m\|\mc{K}(f_j)\|_{L^{p_j}_{W_j}(\R^d;\mc{K}(\mc{H}_j))}\\
&=\|M\|_{L^{\vec{p}}_{\vec{W}}(\R^d;\mc{K}(\vec{\mc{H}}))\to L^p_{\mb{W}}(\R^d;\mc{K}(\mc{H}))_{\emph{weak}}}\prod_{j=1}^m\|f_j\|_{L^{p_j}_{W_j}(\R^d;\mc{H}_j)}.
\end{align*}
This proves that $\vec{W}\in A_{\vec{p}}$ with
\[
[\vec{W}]_{\vec{p}}\leq \|M\|_{L^{\vec{p}}_{\vec{W}}(\R^d;\vec{\mc{H}})\to L^p_{\mb{W}}(\R^d;\mc{H})_{\text{weak}}}.
\]

For \ref{it:weaktypechar1}$\Rightarrow$\ref{it:weaktypechar2}, by a $3^d$-lattice reduction and the monotone convergence property it suffices to bound $M_{\mc{F}}$ uniformly for finite collections of cubes $\mc{F}$ contained in a dyadic grid $\mc{D}$.

Let $u\in\mc{H}$. If $u\in M_{\mc{F}}\vec{F}(x)$, there is a $Q\in\mc{F}$ with $x\in Q$ for which $u\in \mc{K}\Big(\bigotimes_{j=1}^m\langle F_j\rangle_Q\Big)$. Denote the maximal cubes $Q\in \mc{F}$ for which $u\in \mc{K}\Big(\bigotimes_{j=1}^m\langle F_j\rangle_Q\Big)$ by $\mc{P}$. Then
\[
\ind_{\{x\in\R^d:u\in M_{\mc{F}}\vec{F}(x)\}}u\in \sum_{Q\in\mc{P}}\bigotimes_{j=1}^m\langle F_j\rangle_Q\ind_Q
\]
a.e., so that by Lemma~\ref{lem:strongmuckenhouptconvex} we have
\begin{align*}
\|\ind_{\{x\in\R^d:u\in M_{\mc{F}}\vec{F}(x)\}}u\|_{L^p_{\mb{W}}(\R^d;\mc{H})}&\leq\Big\|\sum_{Q\in\mc{P}}\mc{K}\Big(\bigotimes_{j=1}^m\langle F_j\rangle_Q\Big)\ind_Q\Big\|_{L^p_{\mb{W}}(\R^d;\mc{K}(\mc{H}))}\\
&\lesssim_{n,m}[\vec{W}]_{\vec{p}}\prod_{j=1}^m\|F_j\|_{L^{p_j}_{W_j}(\R^d;\mc{K}(\mc{H}_j))}.
\end{align*}
The result follows.
\end{proof}

\subsection{Strong type upper bounds}

Next, we prove a strong-type bound for $M^{\mc{K}}$.

\begin{theorem}\label{thm:strongboundsmaxop}
Let $\vec{p}\in(1,\infty]^m$ and $\vec{W}\in A_{\vec{p}}$. Then
\[
M^{\mc{K}}:L^{\vec{p}}_{\vec{W}}(\R^d;\mc{K}(\vec{\mc{H}}))\to L^p_{\mb{W}}(\R^d;\mc{K}(\mc{H}))
\]
with
\begin{align*}
\|M^{\mc{K}}\|_{L^{\vec{p}}_{\vec{W}}(\R^d;\mc{K}(\vec{\mc{H}}))\to L^p_{\mb{W}}(\R^d;\mc{K}(\mc{H}))}.
&\lesssim_{d,m,\vec{n},\vec{p}}[\vec{W}]_{\vec{p}}\prod_{j=1}^m[W_j^{-1}]_{\mathrm{FW}_{p_j'}}^{\frac{p_j'}{p_j}}\\
    &\lesssim_{d,m,\vec{n},\vec{p}}[\vec{W}]_{\vec{p}}^{1+\sum_{j=1}^m\frac{p_j'}{p_j}}.
\end{align*}
\end{theorem}

We begin with a couple of standard reductions. Using Proposition~\ref{prop:aumannintegralbound}, for any $\vec{F}\in L^1_{\mathrm{loc}}(\R^d;\mc{K}(\vec{\mc{H}}))$ we have
\begin{align*}
    &\Vert\mathbf{W}(x)\mc{M}^{\mc{K}}(\vec{F})(x)\Vert_{\mc{H}}=\Big\Vert \mathbf{W}(x)\mc{K}\Big(\bigcup_{Q}\bigotimes_{j=1}^m\langle F_j\rangle_Q\ind_Q(x)\Big)\Big\Vert_{\mc{H}}
    =\Big\Vert \bigcup_{Q}\bigotimes_{j=1}^m\langle W_j(x)F_j\rangle_Q\ind_Q(x)\Big\Vert_{\mc{H}}\\
    &=\sup_{Q}\Big\Vert\bigotimes_{j=1}^{m}\langle W_j(x)F_j\rangle_Q\Big\Vert_{\mc{H}}\ind_Q(x)
    =\sup_{Q}\prod_{j=1}^{m}\Big\Vert\langle W_j(x)F_j\rangle_Q\Big\Vert_{\mc{H}_j}\ind_Q(x)\\
    &\eqsim_n\sup_{Q}\prod_{j=1}^{m}\langle\Vert W_j(x)F_j\Vert_{\mc{H}_j}\rangle_Q\ind_Q(x):=\mc{M}_{\vec{W}}(\vec{F})(x).
\end{align*}
Thus, we only have to find $L^{\vec{p}}_{\vec{W}}(\R^d;\mc{K}(\mc{H}))\to L^{p}(\R^d)$ bounds for $\mc{M}_{\vec{W}}$.

We will actually prove a more general result. Let $\vec{p}\in(0,\infty]^{m}$ and $\vec{r}\in(0,\infty)^{m}$ with $\vec{p}\geq\vec{r}$, $p\neq\infty$ and $p_j\neq r_j$ for all $j=1,\ldots,m$. Let $\vec{W}$ be a $m$-tuple of matrix weights with $[\vec{W}]_{\vec{p},(\vec{r},\infty),\mathrm{op}}<\infty$. For $\vec{F}\in L^{\vec{r}}_{\mathrm{loc}}(\R^d;\mc{K}(\vec{\mc{H}}))$ we define
\begin{equation*}
    \mc{M}_{\vec{W},\vec{r}}(\vec{F})(x):=\sup_{Q}\Big(\prod_{j=1}^{m}\langle\Vert W_j(x)F_j\Vert_{\mc{H}_j}^{r_{j}}\rangle_{Q}^{\frac{1}{r_j}}\Big)\ind_{Q}(x),
\end{equation*}
where the supremum is taken over all cubes $Q\subseteq\R^d$. We will deduce $L^{\vec{p}}_{\vec{W}}(\R^d;\mc{K}(\vec{\mc{H}}))\to L^{p}(\R^d)$ bounds for $\mc{M}_{\vec{W},\vec{r}}$ in terms of $[\vec{W}]_{\vec{p},(\vec{r},\infty),\mathrm{op}}$.

\begin{theorem}
    \label{thm:strong_type_maximal_general} 
    Let $\vec{p}\in(0,\infty]^{m}$ and $\vec{r}\in(0,\infty)^{m}$ with $\vec{p}>\vec{r}$. Let $\vec{W}$ be a $m$-tuple of matrix weights with $[\vec{W}]_{\vec{p},(\vec{r},\infty),\mathrm{op}}<\infty$. We have
    \begin{align*}
        \Vert\mc{M}_{\vec{W},\vec{r}}\Vert_{L^{\vec{p}}_{\vec{W}}(\R^d;\mc{K}(\vec{\mc{H}}))\to L^{p}(\R^d)}&
        \lesssim_{d,m,\vec{n},\vec{p},\vec{r}}[\vec{W}]_{\vec{p},(\vec{r},\infty),\mathrm{op}}\prod_{j=1}^{m}[W_j^{-1}]_{\mathrm{FW}_{\frac{1}{\frac{1}{r_j}-\frac{1}{p_j}}}}^{\frac{\frac{1}{p_j}}{\frac{1}{r_j}-\frac{1}{p_j}}}\\
        &\lesssim_{d,m,\vec{n},\vec{r}}[\vec{W}]_{\vec{p},(\vec{r},\infty),\mathrm{op}}^{1+\sum_{j=1}^{m}\frac{\frac{1}{p_j}}{\frac{1}{r_j}-\frac{1}{p_j}}}.
    \end{align*}
\end{theorem}
The second inequality follows from the first through Corollary~\ref{cor:needed_reverse_Holder} (and Remark~\ref{rem:sharppdendence}), so we need only prove the first one. Moreover, we note that by setting $\vec{r}=(1,\ldots,1)$, Theorem~\ref{thm:strongboundsmaxop} follows from this one.

The proof of Theorem~\ref{thm:strong_type_maximal_general} follows the same strategy as in \cite[Theorem 3.2]{Go03}. We first need to bound two auxiliary operators. The first one is the multilinear counterpart of ``Goldberg's auxiliary maximal operator'' from \cite{Go03} and is of interest in its own. Namely, for $\vec{F}\in L^{\vec{r}}_{\mathrm{loc}}(\R^d;\mc{K}(\vec{\mc{H}}))$ we set
    \begin{align*}
        \widetilde{\mc{M}}_{\vec{W},\vec{r}}(\vec{F})(x):=\sup_{Q}\Big(\prod_{j=1}^{m}\langle\Vert(A_{W_j^{-1},Q,t_j})^{-1}F_j\Vert_{\mc{H}_j}^{r_j}\rangle_{Q}^{\frac{1}{r_j}}\Big)\ind_{Q}(x),
    \end{align*}
    where the supremum is taken over all cubes $Q\subseteq\R^d$. If $\vec{r}=(1,\ldots,1)$ then we write $\widetilde{\mc{M}}_{\vec{W}}:=\widetilde{\mc{M}}_{\vec{W},\vec{r}}$.

\begin{proposition}
    \label{prop:auxiliary_maximal}
     Let $\vec{p}\in(0,\infty]^{m}$ and $\vec{r}\in(0,\infty)^{m}$ with $\vec{p}>\vec{r}$. Let $\vec{W}$ be a $m$-tuple of matrix weights with $[W_j^{-1}]_{\mathrm{FW}_{\frac{1}{\frac{1}{r_j}-\frac{1}{p_j}}}}<\infty$ for all $j=1,\ldots,m$. We have
    \begin{align*}
        \Vert\widetilde{\mc{M}}_{\vec{W},\vec{r}}\Vert_{L^{\vec{p}}_{\vec{W}}(\R^d;\mc{K}(\vec{\mc{H}}))\to L^{p}(\R^d)}&\lesssim_{d,\vec{n},\vec{p},\vec{r}}\prod_{j=1}^{m}[W_j^{-1}]_{\mathrm{FW}_{\frac{1}{\frac{1}{r_j}-\frac{1}{p_j}}}}^{\frac{\frac{1}{p_j}}{\frac{1}{r_j}-\frac{1}{p_j}}}.
    \end{align*}
\end{proposition}

\begin{proof}
    We adapt the argument of \cite[p.~252]{Isralowitz_Kwon_Pott2017}. Set
    \begin{equation*}
        \frac{1}{\vec{t}}:=\frac{1}{\vec{r}}-\frac{1}{\vec{p}},\quad\vec{\kappa}:=\frac{\vec{p}}{\vec{r}}.
    \end{equation*}
    We note that $\vec{r}\leq\vec{t}$. For each $j=1,\ldots,m$ let $\tau_j\in(1,\infty)$ with
    \begin{equation*}
        \tau_j'=2^{d+1}\sup_{\substack{A\in\mc{L}(\mc{H}_j)\\A\text{ invertible}}}[\Vert W_j A\Vert_{\mc{H}_j\to\mc{H}_j}^{t_j}]_{\mathrm{FW}},
    \end{equation*}
    where as a consequence of Lemma~\ref{lem:reducing_operator_on_operator} we observe that
    \begin{equation*}
        \sup_{\substack{A\in\mc{L}(\mc{H}_j)\\A\text{ invertible}}}[\Vert W_j A\Vert_{\mc{H}_j\to\mc{H}_j}^{t_j}]^{\frac{1}{t_j}}_{\mathrm{FW}}\lesssim_{n}[W_j^{-1}]_{\mathrm{FW}_{t_j}}<\infty.
    \end{equation*}
    Then, using in each factor H\"{o}lder's inequality for the exponents $\frac{\tau_j t_j}{r_j}$, $\Big(\frac{\tau_j t_j}{r_j}\Big)'$ (where we note that $t_j\geq r_j$), the sharp Reverse H\"{o}lder inequality \cite[Theorem 2.3]{HPR12}, and Lemma~\ref{lem:reducing_operator_on_operator}, we estimate
    \begin{align*}
        &\widetilde{\mc{M}}_{\vec{W},\vec{r}}(\vec{F})(x)\\
        &\leq\sup_{Q}\Big(\prod_{j=1}^{m}\langle\Vert(A_{W_j^{-1},Q,t_j})^{-1}W_j^{-1}\Vert_{\mc{H}_j\to\mc{H}_j}^{\tau_jt_j}\rangle_{Q}^{\frac{1}{\tau_jt_j}}\langle\Vert W_jF_j\Vert_{\mc{H}_j}^{r_j\Big(\frac{\tau_jt_j}{r_j}\Big)'}\rangle_{Q}^{\frac{1}{r_j\Big(\frac{\tau_jt_j}{r_j}\Big)'}}\Big)\ind_{Q}(x)\\
        &\lesssim_{\vec{r}}\sup_{Q}\Big(\prod_{j=1}^{m}\langle\Vert W_j^{-1}(A_{W_j^{-1},Q,t_j})^{-1}\Vert_{\mc{H}_j\to\mc{H}_j}^{t_j}\rangle_{Q}^{\frac{1}{t_j}}\langle\Vert W_jF_j\Vert_{\mc{H}_j}^{r_j\Big(\frac{\tau_jt_j}{r_j}\Big)'}\rangle_{Q}^{\frac{1}{r_j\Big(\frac{\tau_jt_j}{r_j}\Big)'}}\Big)\ind_{Q}(x)\\
        &\eqsim_{\vec{n},\vec{r}}\sup_{Q}\Big(\prod_{j=1}^{m}\langle\Vert W_jF_j\Vert^{r_j\Big(\frac{\tau_jt_j}{r_j}\Big)'}_{\mc{H}_j}\rangle_{Q}^{\frac{1}{r_j\Big(\frac{\tau_jt_j}{r_j}\Big)'}}\Big)\ind_{Q}(x).
    \end{align*}
    Set
    \begin{equation*}
        w_j:=\frac{p_j}{r_j\Big(\frac{\tau_jt_j}{r_j}\Big)'},\quad j=1,\ldots,m.
    \end{equation*}
    We observe that
    \begin{align*}
        \frac{\tau_jt_j}{r_j}>\frac{t_j}{r_j}=\kappa_{j}',
    \end{align*}
    therefore $w_j=\frac{\kappa_j}{\Big(\frac{\tau_jt_j}{r_j}\Big)'}>1$. Thus, applying again H\"{o}lder's inequality and then the classical strong type estimate for the Hardy--Littlewood maximal function $M$ on $\R^d$ we obtain
    \begin{align*}
        &\Big(\int_{\R^d}\widetilde{M}_{\vec{W},\vec{r}}(\vec{F})(x)^{p}\mathrm{d}x\Big)^{\frac{1}{p}}
        \lesssim_{\vec{n},\vec{r}}\prod_{j=1}^{m}\Big(\int_{\R^d}M(\Vert W_jF_j\Vert_{\mc{H}_j}^{r_j\Big(\frac{\tau_jt_j}{r_j}\Big)'})(x)^{\frac{p_j}{r_j\Big(\frac{\tau_jr_j}{r_j}\Big)'}}\mathrm{d}x\Big)^{\frac{1}{p_j}}\\
        &\lesssim_{d,\vec{p},\vec{r}}\prod_{j=1}^{m}(w_j')^{\frac{1}{p_j}}\Big(\int_{\R^d}\Vert W_j(x)F_j(x)\Vert_{\mc{H}_j}^{p_j}\mathrm{d}x\Big)^{\frac{1}{p_j}}.
    \end{align*}
    Now notice that
    \begin{align*}
        w_j'&=\frac{p_j}{p_j-r_j\Big(\frac{\tau_jt_j}{r_j}\Big)'}=\frac{p_j}{p_j-\frac{r_j\tau_jt_j}{\tau_jt_j-r_j}}=\frac{p_j(\tau_jt_j-r_j)}{p_j(\tau_jt_j-r_j)-r_j\tau_jt_j}\\
        &=\frac{p_j(\tau_jt_j-r_j)}{\tau_jt_j(p_j-r_j)-p_jr_j}
        =\frac{\frac{\tau_jt_j}{r_j}-1}{\tau_jt_j\Big(\frac{1}{r_j}-\frac{1}{p_j}\Big)-1}
        =\frac{\frac{\tau_jt_j}{r_j}-1}{\tau_j-1}\lesssim_{p_j,r_j}\tau_j',
    \end{align*}
    therefore
    \begin{align*}
        (w_j')^{\frac{1}{p_j}}\lesssim_{p_j,r_j,s}(\tau_{j}')^{\frac{1}{p_j}}\lesssim_{d,n,p_j,r_j}[W_j^{-1}]_{\mathrm{FW}_{t_j}}^{\frac{t_j}{p_j}}.
    \end{align*}
    It now remains to observe that $\frac{t_j}{p_j}=\frac{\frac{1}{p_j}}{\frac{1}{r_j}-\frac{1}{p_j}}$.
\end{proof}

\begin{lemma}
    Suppose $p<\infty$. Let $\mc{F}$ be a finite family of cubes contained in a dyadic grid $\mc{D}$. For all $Q\in\mc{F}$, we consider the auxiliary operator
    \begin{equation*}
        N_{Q,\mc{F}}(x):=\sup_{R\in\mc{F}(Q)}\Big(\prod_{j=1}^{m}\Vert W_j(x)A_{W_j^{-1},R,t_j}\Vert_{\mc{H}_j\to\mc{H}_j}\Big)\ind_{R}(x).
    \end{equation*}
    Then
    \begin{equation*}
        \int_{Q}N_{Q,\mc{F}}(x)^{p}\mathrm{d}x\lesssim_{\vec{n},\vec{r},\vec{p}}[\vec{W}]_{\vec{p},(\vec{r},\infty),\mathrm{op}}^p|Q|
    \end{equation*}
    for all $Q\in\mc{F}$.
\end{lemma}

\begin{proof}
    We adapt the proof of \cite[Lemma 3.3]{Go03}. Since $[\vec{W}]_{\vec{p},(\vec{r},\infty),\mathrm{op}}<\infty$, we have that $\mathbf{W}$ is $p$-integrable over $Q$. Since $\mc{F}$ is finite, a crude estimate shows that there is a best constant $0<B<\infty$ such that
    \begin{align*}
        \int_{Q}N_{Q,\mc{F}}(x)^{p}\mathrm{d}x\leq B|Q|,
    \end{align*}
    for all $Q\in\mc{F}$. We show that $B\lesssim_{\vec{n},\vec{r},\vec{p}}[\vec{W}]_{\vec{p},(\vec{r},\infty),\mathrm{op}}^p$.

    Let $0<A<\infty$ be a large constant to be specified later. Fix $Q\in\mc{F}$. Denote by $\mc{S}$ the set of maximal cubes $R$ in $\mc{F}(Q)$ satisfying
    \begin{equation*}
        \prod_{j=1}^{m}\Vert(A_{W_j^{-1},Q,t_j})^{-1}A_{W_j^{-1},R,t_j}\Vert_{\mc{H}_j\to\mc{H}_j}>A.
    \end{equation*}
    It is clear that
    \begin{equation*}
        N_{Q,\mc{F}}(x)\leq A\prod_{j=1}^{m}\Vert W_j(x)A_{W_j^{-1},Q,t_j}\Vert_{\mc{H}_j\to\mc{H}_j},\quad\forall x\in Q\setminus(\bigcup\mc{S}),
    \end{equation*}
    therefore, using Lemma~\ref{lem:reducing_operator_on_operator},
    \begin{align*}
        \int_{Q\setminus(\bigcup\mc{S})} N_{Q,\mc{F}}(x)^{p}\mathrm{d}x\leq A^{p}\int_{Q}\prod_{j=1}^{m}\Vert W_j(x)A_{W_j^{-1},Q,t_j}\Vert_{\mc{H}_j\to\mc{H}_j}^{p}\mathrm{d}x\leq c^p A^{p}[\vec{W}]_{\vec{p},(\vec{r},\infty),\mathrm{op}}^{p}|Q|,
    \end{align*}
    where the constant $c>0$ depends only on $\vec{n},\vec{r}$.

    We now claim that
    \begin{equation*}
        \sum_{R\in\mc{S}}|R|<\frac{1}{2}|Q|
    \end{equation*}
    if $A$ is sufficiently large. Indeed, for each $R\in\mc{S}$ we estimate, using Lemma~\ref{lem:reducing_operator_on_operator},
    \begin{align*}
        &A^{t}\leq\prod_{j=1}^{m}\Vert(A_{W_j^{-1},Q,t_j})^{-1}A_{W_j^{-1},R,t_j}\Vert_{\mc{H}_j\to\mc{H}_j}^{t}\\
        &\eqsim_{\vec{n},\vec{p},\vec{r},m}\frac{1}{|R|}\prod_{j=1}^{m}\Big(\int_{R}\Vert W_j(x)^{-1}(A_{W_j^{-1},Q,t_j})^{-1}\Vert_{\mc{H}_j\to\mc{H}_j}^{t_j}\mathrm{d}x\Big)^{\frac{t}{t_j}},
    \end{align*}
    so by applying H\"older's inequality (for sums) we obtain
    \begin{align*}
        A^{t}\sum_{R\in\mc{S}}|R|&\lesssim_{\vec{n},\vec{p},\vec{r},m}\sum_{R\in\mc{S}}\prod_{j=1}^{m}\Big(\int_{R}\Vert W_j(x)^{-1}(A_{W_j^{-1},Q,t_j})^{-1}\Vert_{\mc{H}_j\to\mc{H}_j}^{t_j}\mathrm{d}x\Big)^{\frac{t}{t_j}}\\
        &\leq\prod_{j=1}^{m}\Big(\sum_{R\in\mc{S}}\int_{R}\Vert W_j(x)^{-1}(A_{W_j^{-1},Q,t_j})^{-1}\Vert_{\mc{H}_j\to\mc{H}_j}^{t_j}\mathrm{d}x\Big)^{\frac{t}{t_j}}\\
        &\leq\prod_{j=1}^{m}\Big(\int_{Q}\Vert W_j(x)^{-1}(A_{W_j^{-1},Q,t_j})^{-1}\Vert_{\mc{H}_j\to\mc{H}_j}^{t_j}\mathrm{d}x\Big)^{\frac{t}{t_j}}\eqsim_{\vec{n},\vec{p},\vec{r},m}1.
    \end{align*}
    This shows that one has only to choose large enough $A$ depending only on $\vec{n},\vec{p},\vec{r}$.

    Observe that for each $R\in\mc{S}$ and each $x\in R$ we have either $N_{Q,\mc{F}}(x)=N_{R,\mc{F}}(x)$ or $N_{Q,\mc{F}}(x)\leq A\prod_{j=1}^{m}\Vert W_j(x)A_{W_j^{-1},Q,t_j}\Vert_{\mc{H}_j\to\mc{H}_j}$. Thus, since the cubes in $\mc{S}$ are pairwise disjoint we conclude
    \begin{align*}
        \int_{\bigcup\mc{S}}N_{Q,\mc{F}}(x)^{p}\mathrm{d}x&=\sum_{R\in\mc{S}}\int_{R}N_{Q,\mc{F}}(x)^{p}\mathrm{d}x\leq cA^{p}[\vec{W}]_{\vec{p},(\vec{r},\infty),\mathrm{op}}^{p}|Q|+B\sum_{R\in\mc{S}}|R|\\
        &<\Big(cA^{p}[\vec{W}]_{\vec{p},(\vec{r},\infty),\mathrm{op}}^{p}+\frac{B}{2}\Big)|Q|.
    \end{align*}
    Thus, we deduce that $B\leq 2cA^{p}[\vec{W}]_{\vec{p},(\vec{r},\infty),\mathrm{op}}^{p}+\dfrac{B}{2}$, showing that $B\leq 4cA^{p}[\vec{W}]_{\vec{p},(\vec{r},\infty),\mathrm{op}}^{p}$.
\end{proof}

We can now prove Theorem~\ref{thm:strong_type_maximal_general}.

\begin{proof}[Proof of Theorem~\ref{thm:strong_type_maximal_general}]
If $p=\infty$, note that for every cube $Q$ and a.e. $x\in Q$ we have
\begin{align*}
\prod_{j=1}^{m}\langle\Vert W_j(x)F_j\Vert_{\mc{H}_j}^{r_{j}}\rangle_{Q}^{\frac{1}{r_j}}
&\leq\prod_{j=1}^{m}\langle\Vert W_j(x)W_j^{-1}\Vert_{\mc{H}_j}^{r_{j}}\rangle_{Q}^{\frac{1}{r_j}}\prod_{j=1}^m\|F_j\|_{L^\infty_{W_j}(\R^d;\mc{K}(\mc{H}_j)}\\
&\leq[\vec{W}]_{\vec{\infty},(\vec{r},\infty),\mathrm{op}}\prod_{j=1}^m\|F_j\|_{L^\infty_{W_j}(\R^d;\mc{K}(\mc{H}_j)},
\end{align*}
proving the result.

Now suppose $p<\infty$. Through the $3^d$-lattice trick and monotone convergence it suffices to prove such a bound uniformly for the operators $\mc{M}_{\mc{F},\vec{W},\vec{r}}$ obtained by restricting the supremum in the definition of $\mc{M}_{\vec{W},\vec{r}}$ to cubes in finite collections $\mc{F}$ contained in a dyadic grid $\mc{D}$. In particular, $\mc{M}_{\mc{F},\vec{W},\vec{r}}(\vec{F})(x)<\infty$ for every $x\in\R^d$.

    For each $j\in\Z$, let $\mc{S}_j$ be the family of all maximal cubes $R\in\mc{F}$ with the property
    \begin{equation*}
        2^{j}<\prod_{j=1}^{m}\langle\Vert (A_{W_j^{-1},R,t_j})^{-1}F_j\Vert_{\mc{H}_j}^{r_j}\rangle_{R}^{\frac{1}{r_j}}\leq 2^{j+1},
    \end{equation*}
    and set
    \begin{equation*}
        E_j:=\bigcup_{R\in\mc{S}_j}\{x\in R:~M_{\mc{F},\vec{W},\vec{r}}(\vec{F})(x)\leq2^{j+2}N_{R,\mc{F}}(x)\}.
    \end{equation*}
    Let $x\in\R^d$ with $\mc{M}_{\mc{F},\vec{W},\vec{r}}(\vec{F})(x)\neq0$. We claim that there exists $j\in\Z$ with $x\in E_j$. Indeed, pick $R_x\in\mc{F}$ with $x\in R_x$ and
    \begin{equation*}
        \mc{M}_{\mc{F},\vec{W},\vec{r}}(\vec{F})(x)\leq 2\prod_{j=1}^{m}\langle\Vert W_j(x)F_j\Vert_{\mc{H}_j}^{r_j}\rangle_{R_x}^{\frac{1}{r_j}}.
    \end{equation*}
    Observe that
    \begin{align*}
        \prod_{j=1}^{m}\langle\Vert W_j(x)F_j\Vert_{\mc{H}_j}^{r_j}\rangle_{R_x}^{\frac{1}{r_j}}
        &\leq\prod_{j=1}^{m}\Vert W_j(x)A_{W_j^{-1},R_x,t_j}\Vert_{\mc{H}_j\to\mc{H}_j}\cdot\langle\Vert (A_{W_j^{-1},R_x,t_j})^{-1}F_j\Vert_{\mc{H}_j}^{r_j}\rangle_{R_x}^{\frac{1}{r_j}},
    \end{align*}
    thus $\prod_{j=1}^{m}\langle\Vert (A_{W_j^{-1},R_x,t_j})^{-1}F_j\Vert_{\mc{H}_j}^{r_j}\rangle_{R_x}^{\frac{1}{r_j}}>0$. Therefore, there exists $j\in\Z$ with
    \begin{equation*}
        2^{j}<\prod_{j=1}^{m}\langle\Vert (A_{W_j^{-1},R_x,t_j})^{-1}F_j\Vert_{\mc{H}_j}^{r_j}\rangle_{R_x}^{\frac{1}{r_j}}\leq 2^{j+1}.
    \end{equation*}
    Then, there is $R\in\mc{S}_j$ such that $R_x\subseteq R$. Then $x\in R$ and
    \begin{align*}
        \prod_{j=1}^{m}\langle\Vert W_j(x)F_j\Vert_{\mc{H}_j}^{r_j}\rangle_{R_x}^{\frac{1}{r_j}}
        &\leq\prod_{j=1}^{m}\Vert W_j(x)A_{W_j^{-1},R_x,t_j}\Vert_{\mc{H}_j\to\mc{H}_j}\cdot\langle\Vert (A_{W_j^{-1},R_x,t_j})^{-1}F_j\Vert_{\mc{H}_j}^{r_j}\rangle_{R_x}^{\frac{1}{r_j}}\\
        &\leq 2^{j+1}N_{R,\mc{F}}(x),
    \end{align*}
    thus $x\in E_j$.
     
    Observe also that $\bigcup\mc{S}_j\subseteq\{\widetilde{\mc{M}}_{\mc{F},\vec{W},\vec{r}}(\vec{F})>2^{j}\}$ for each $j\in\Z$. Thus, we deduce
    \begin{align*}
        &\int_{\R^d}M_{\mc{F},\vec{W},\vec{r}}(\vec{F})(x)^{p}\mathrm{d}x\leq\sum_{j\in\Z}\int_{E_j}M_{\mc{F},\vec{W},\vec{r}}(\vec{F})(x)^{p}\mathrm{d}x\\
        &\leq\sum_{j\in\Z}2^{(j+2)p}\sum_{R\in\mc{S}_j}\int_{R}N_{R,\mc{F}}(x)^{p}\mathrm{d}x\\
        &\lesssim_{m,\vec{n},\vec{p},\vec{r}}[\vec{W}]_{\vec{p},(\vec{r},\infty),\mathrm{op}}^{p}\sum_{j\in\Z}2^{jp}\sum_{R\in\mc{S}_j}|R|=
        [\vec{W}]_{\vec{p},(\vec{r},\infty),\mathrm{op}}^{p}\sum_{j\in\Z}2^{jp}\Big|\bigcup\mc{S}_j\Big|\\
        &\leq[\vec{W}]_{\vec{p},(\vec{r},\infty),\mathrm{op}}^{p}\sum_{j\in\Z}2^{jp}|\{\widetilde{\mc{M}}_{\mc{F},\vec{W}}(\vec{F})>2^{j}\}|
        \eqsim_{p}[\vec{W}]_{\vec{p},(\vec{r},\infty),\mathrm{op}}^{p}\Vert\widetilde{\mc{M}}_{\mc{F},\vec{W},\vec{r}}(\vec{F})\Vert_{L^{p}(\R^d)}^{p}\\
        &\lesssim_{d,\vec{n},\vec{p},\vec{r}}[\vec{W}]_{\vec{p},(\vec{r},\infty),\mathrm{op}}\Big(\prod_{j=1}^{m}[W_j^{-1}]_{\mathrm{FW}_{\frac{1}{\frac{1}{r_j}-\frac{1}{p_j}}}}^{\frac{\frac{1}{p_j}}{\frac{1}{r_j}-\frac{1}{p_j}}}\Big)^p\prod_{j=1}^{m}\Vert F_j\Vert_{L^{p_j}_{W_j}(\R^d;\mc{K}(\mc{H}_j))}^{p},
    \end{align*}
    concluding the proof.
\end{proof}

\subsection{Sparse domination}
While we did not use this in the proof of Theorem~\ref{thm:strongboundsmaxop}, we nonetheless prove convex body sparse domination of $M^{\mc{K}}$, generalizing \cite[Theorem~C]{Ni24b}.
\begin{lemma}
Let $\mc{D}$ be a dyadic grid in $\R^d$ and let $\mc{F}\subseteq\mc{D}$ be finite. For all $\vec{F}\in L^{\vec{1}}_{\mathrm{loc}}(\R^d;\vec{\mc{H}})$ there exists a martingale $\frac{1}{2}$-sparse collection $\mc{S}\subseteq\mc{F}$ for which
\[
M^{\mc{K}}_{\mc{F}}\vec{F}(x)\subseteq n^{\frac{3}{2}}(2n)^mM^{\mc{K}}_{\mc{S}}\vec{F}(x)
\]
for a.e. $x\in\R^d$.
\end{lemma}
\begin{proof}
Let $Q_0\in\mc{F}$. Using the John ellipsoid theorem, for each $j\in\{1,\ldots,m\}$ we can find an orthonormal set $\mc{B}_j(Q_0)$ of $\mc{H}_j$ and $(\lambda_{e_j})_{e_j\in\mc{B}_j(Q_0)}$, $\lambda_{e_j}\geq 0$, for which
\[
\mc{E}_j(Q_0):=\Big\{\sum_{e_j\in\mc{B}_j(Q_0)} \langle u,e_j\rangle_{\mc{H}_j}\lambda_{e_j} e_j:\|u\|_{\mc{H}_j}\leq 1\Big\}
\]
satisfies
\[
\mc{E}_j(Q_0)\subseteq\langle F_j\rangle_{Q_0}\subseteq n_j^{\frac{1}{2}}\mc{E}_j(Q_0).
\]
We set
\[
\mc{B}(Q_0)=\Big\{\bigotimes_{j=1}^m e_j:e_j\in\mc{B}_j(Q_0),\,j=1,\ldots,m\Big\},
\]
and for each $e\in\mc{B}(Q_0)$ we let $\text{ch}_e(Q_0)$ denote the collection of maximal cubes $Q\in\mc{F}$ contained in $Q_0$ satisfying
\[
\prod_{j=1}^m|\langle \langle F_j, e_j\rangle_{\mc{H}_j}\rangle_Q|>(2n)^m \prod_{j=1}^m\langle|\langle F_j, e_j\rangle_{\mc{H}_j}|\rangle_{Q_0},
\]
where $e=\bigotimes_{j=1}^m e_j$. 

Letting $\text{ch}(Q_0)$ denote the maximal cubes in $\bigcup_{e\in\mc{B}(Q_0)}\text{ch}_e(Q_0)$, we let $\mc{S}_0$ denote the maximal cubes in $\mc{F}$, and iteratively define 
\[
\mc{S}_{i+1}=\bigcup_{Q\in\mc{S}_i}\text{ch}(Q),\quad \mc{S}:=\bigcup_{i=0}^\infty\mc{S}_i.
\]
We will show that $\mc{S}$ is sparse. Indeed, for each $Q_0\in\mc{S}$ we have
\begin{align*}
\sum_{Q\in\text{ch}(Q_0)}|Q|&\leq\sum_{e\in\mc{B}(Q_0)}\sum_{Q\in\text{ch}_e(Q_0)}|Q|\\
&\leq\sum_{e\in\mc{B}(Q_0)}\frac{1}{2 n\prod_{j=1}^m\langle|\langle F_j, e_j\rangle_{\mc{H}_j}|\rangle_{Q_0}^{\frac{1}{m}}}\sum_{Q\in\text{ch}_e(Q_0)}\prod_{j=1}^m\Big(\int_Q|\langle F_j(x),e_j\rangle_{\mc{H}_j}|\,\mathrm{d}x\Big)^{\frac{1}{m}}\\
&\leq \sum_{e\in\mc{B}(Q_0)}\frac{1}{2 n\prod_{j=1}^m\langle|\langle F_j, e_j\rangle_{\mc{H}_j}|\rangle_{Q_0}^{\frac{1}{m}}}\prod_{j=1}^m\Big(\sum_{Q\in\text{ch}_e(Q_0)}\int_Q|\langle F_j(x),e_j\rangle_{\mc{H}_j}|\,\mathrm{d}x\Big)^{\frac{1}{m}}\\
&\leq\sum_{e\in\mc{B}(Q_0)}\frac{|Q_0|}{2n}=\frac{|Q_0|}{2},
\end{align*}
as desired.

For each $Q\in\mc{F}$, we let $\pi_{\mc{S}}(Q)$ denote the smallest cube $Q_0\in\mc{S}$ containing $Q$. Suppose $Q\in\mc{F}$ satisfies $\pi_{\mc{S}}(Q)=Q_0$. If $Q\neq Q_0$, then $Q$ fails the stopping condition for all $e\in\mc{B}(Q_0)$, i.e., if $e=\bigotimes_{j=1}^m e_j$, then
\[
\prod_{j=1}^m|\langle \langle F_j, e_j\rangle_{\mc{H}_j}\rangle_Q|\leq (2n)^m \prod_{j=1}^m\langle|\langle F_j, e_j\rangle_{\mc{H}_j}|\rangle_{Q_0}.
\]
As this estimate remains true when $Q=Q_0$, we find that
\begin{equation}\label{eq:sparseinclusion1}
\bigotimes_{j=1}^m\langle F_j\rangle_Q=\sum_{e\in\mc{B}(Q_0)}\bigotimes_{j=1}^m\langle \langle F_j, e_j\rangle_{\mc{H}_j}\rangle_Q e_j\subseteq (2n)^m\sum_{e\in\mc{B}(Q_0)}\Big(\prod_{j=1}^m\langle |\langle F_j, e_j\rangle_{\mc{H}_j}|\rangle_{Q_0}\Big)\mc{K}(e).
\end{equation}
Now, fix $e\in\mc{B}(Q_0)$ where $e=\bigotimes_{j=1}^m e_j$, and pick $f_j\in S^0(\R^d;F_j)$ for which 
\[
\langle|\langle F_j,e_j\rangle_{\mc{H}_j}|\rangle_{Q_0}=\langle \langle f_j,e_j\rangle_{\mc{H}_j}\rangle_{Q_0},\quad j=1,\ldots,m.
\]
Fixing $j\in\{1,\ldots,m\}$, since $\langle f_j\rangle_{Q_0}\in\langle F_j\rangle_{Q_0}\subseteq n_j^{\frac{1}{2}}\mc{E}_j(Q_0)$, we can write
\[
\langle f_j\rangle_{Q_0}=n_j^{\frac{1}{2}}\sum_{e'_j\in\mc{B}_j(Q_0)}\langle u,e'_j\rangle_{\mc{H}_j}\lambda_{e'_j}e'_j
\]
with $\|u\|_{\mc{H}_j}\leq 1$. Thus,
\[
\langle |\langle F_j, e_j\rangle_{\mc{H}_j}|\rangle_{Q_0}e_j=\langle\langle\langle f_j\rangle_Q,e_j\rangle_{\mc{H}_j}\rangle_{Q_0}e_j=n_j^{\frac{1}{2}}\langle u,e_j\rangle_{\mc{H}_j}\lambda_{e_j}e_j\in n_j^{\frac{1}{2}}\mc{E}_j(Q_0)\subseteq n_j^{\frac{1}{2}}\langle F_j\rangle_{Q_0}.
\]
Combining this with \eqref{eq:sparseinclusion1} yields
\[
\bigotimes_{j=1}^m\langle F_j\rangle_Q\subseteq n^{\frac{3}{2}}(2n)^m\mc{K}\Big(\bigotimes_{j=1}^m\langle F_j\rangle_{Q_0}\Big).
\]
In conclusion, we find that
\begin{align*}
M_{\mc{F}}\vec{F}(x)
&=\bigcup_{Q_0\in\mc{S}}\bigcup_{\substack{Q\in\mc{F}\\\pi_{\mc{S}}(Q)=Q_0}}\bigotimes_{j=1}^m\langle F_j\rangle_Q\ind_Q(x)
\subseteq n^{\frac{3}{2}}(2n)^m\bigcup_{Q_0\in\mc{S}}\mc{K}\Big(\bigotimes_{j=1}^m\langle F_j\rangle_{Q_0}\Big)\\
&\subseteq n^{\frac{3}{2}}(2n)^m M^{\mc{K}}_{\mc{S}}\vec{F}(x).
\end{align*}
As $M^{\mc{K}}_{\mc{S}}\vec{F}(x)\in\mc{K}$, we conclude that also 
\[
M^{\mc{K}}_{\mc{F}}\vec{F}(x)=\mc{K}(M_{\mc{F}}\vec{F}(x))\subseteq n^{\frac{3}{2}}(2n)^m M^{\mc{K}}_{\mc{S}}\vec{F}(x),
\]
as desired.
\end{proof}

\section{Matrix weighted bounds for multilinear Calder\'{o}n--Zygmund operators}\label{sec:mczo}

In this section we prove matrix weighted upper bounds for multilinear Calder\'{o}n--Zygmund operators as well as matrix weighted lower bounds for non-degenerate multilinear Calder\'{o}n--Zygmund operators.

\subsection{Multilinear Calder\'{o}n--Zygmund operators}

Here we briefly recall the definition of multilinear Calder\'{o}n--Zygmund operators and properties of them that will be important for the convex body sparse domination.

\smallskip

Let $\omega:[0,\infty)\to[0,\infty)$ be a \emph{modulus of continuity}, meaning in this paper an increasing, continuous function with $\omega(0)=0$ such that there is a constant $D\in[1,\infty)$ with
\begin{equation*}
    \omega(2t)\leq D\,\omega(t),\quad\forall t\in[0,\infty).
\end{equation*}
The best such $D$ is called \emph{doubling constant of $\omega$}. In particular, setting $r:=\log_{2}(D)\in[0,\infty)$, we have
\begin{equation*}
    \omega(ct)\leq c^{r}D\,\omega(t),\quad\forall c\in[1,\infty),\quad\forall t\in[0,\infty).
\end{equation*}
We will also assume that $\omega$ satisfies the so-called \emph{Dini condition}, that is
\begin{equation*}
    \Vert\omega\Vert_{\mathrm{Dini}}:=\int^{1}_{0}\!\omega(t)\,\frac{\mathrm{d}t}{t}<\infty.
\end{equation*}

\smallskip

Let $\Delta:=\{(y_0,y_1,\ldots,y_m)\in(\R^{d})^{m+1}:~y_0=y_1=\ldots=y_m\}$. A \emph{$m$-linear Calder\'{o}n--Zygmund kernel with modulus of continuity $\omega$} is a function $K:(\R^d)^{m+1}\setminus\Delta\to\C$ for which there exists a constant $C_{K}\in(0,\infty)$ such that the following two conditions are fullfilled:
\begin{enumerate}
    \item \emph{Size condition:} For all $(y_0,y_1,\ldots,y_m)\in(\R^d)^{m+1}\setminus\Delta$ we have
    \begin{equation*}
        |K(y_0,y_1,\ldots,y_m)|\leq\frac{C_{K}}{\Big(\sum_{k,\ell=0}^{m}|y_k-y_{\ell}|\Big)^{md}}.
    \end{equation*}

    \item \emph{Smoothness condition:} For all $j=0,1,\ldots,m$, for all
    \begin{equation*}
        (y_0,\ldots,y_{j},\ldots,y_m),(y_0,\ldots,y_{j}',\ldots,y_{m})\in(\R^d)^{m+1}\setminus\Delta
    \end{equation*}
    with
    \begin{equation*}
        |y_j-y_j'|\leq\frac{1}{2}\max_{k=0,\ldots,m}|y_j-y_{k}|
    \end{equation*}
    we have
    \begin{align*}
        &|K(y_0,\ldots,y_j,\ldots,y_m)-K(y_0,\ldots,y_j',\ldots,y_m)|\\
        &\leq C_{K}\omega\Big(\frac{|y_j-y_j'|}{\sum_{k,\ell=0}^{m}|y_k-y_{\ell}|}\Big)\cdot\frac{1}{\Big(\sum_{k,\ell=0}^{m}|y_k-y_{\ell}|\Big)^{md}}.
    \end{align*}
\end{enumerate}

\smallskip

A \emph{multilinear Calder\'{o}n--Zygmund operator with kernel $K$} is an $m$-linear operator
\begin{equation*}
    T:\mc{S}(\R^d)^{m}\to\mc{S}'(\R^d),
\end{equation*}
where $\mc{S}(\R^d)$ denotes the space of Schwarz functions on $\R^d$ and $\mc{S}'(\R^d)$ the space of tempered distributions, such that the following hold:
\begin{enumerate}
    \item The operator $T$ admits a bounded, $m$-linear extension $L^{\vec{q}}(\R^d)\to L^{q}(\R^d)$ for some $m$-tuple $q\in[1,\infty)^{m}$. By abuse of notation we denote this extention also by $T$.

    \item For all $m$-tuples of (suitable) complex valued functions $\vec{f}=(f_1,\ldots,f_m)$ there holds
    \begin{equation*}
        T(\vec{f})(x)=\int_{(\R^d)^{m}}K(x,\vec{y})
    \prod_{j=1}^mf_j(y_j)\,\mathrm{d}\vec{y},
    \end{equation*}
    for all $x\in\R^{d}\setminus\bigcap_{j=1}^{m}\mathrm{supp}(f_j)$.
\end{enumerate}

We will need some properties of multilinear Calder\'{o}n--Zygmund operators that have appeared earlier in the literature. We recall that for $p\in(0,\infty)$ and a measurable function $f:\R^d\to\C$ we define
\begin{equation*}
    \Vert f\Vert_{L^{p,\infty}(\R^d)}:=\sup_{\lambda>0}\lambda|\{|f|>\lambda\}|^{\frac{1}{p}}.
\end{equation*}

\begin{proposition}
    \label{prop:weak_bound_multilinear_CZ}
    The operator $T$ admits a bounded extension $L^{\vec{1}}(\R^d)\to L^{\frac{1}{m},\infty}(\R^d)$ with
    \begin{equation*}
        \Vert T\Vert_{L^{\vec{1}}(\R^d)\to L^{\frac{1}{m},\infty}(\R^d)}\leq C(\Vert T\Vert_{L^{\vec{q}}(\R^d)\to L^{q}(\R^d)}+C_{K}\Vert\omega\Vert_{\mathrm{Dini}}), 
    \end{equation*}
    where the constant $C>0$ depends only on $d,m,\vec{q}$ and the doubling constant of $\omega$.
\end{proposition}

Following \cite{Li18}, we consider the ``grand maximal operator'' $\mc{M}_{T}$ given by
\begin{equation*}
    \mc{M}_{T}(\vec{f})(x):=\sup_{Q}\esssup_{\xi\in Q}|T(\vec{f})(\xi)-T(\vec{f}\ind_{3Q})(\xi)|\ind_{Q}(x),\quad x\in\R^d,
\end{equation*}
where the first supremum ranges over all cubes $Q\subseteq\R^d$. Moreover, following \cite{Li18}, for any cube $Q_0\subseteq\R^d$ we consider the localized grand maximal operator $\mc{M}_{T,Q_0}$ given by
\begin{equation*}
    \mc{M}_{T,Q_0}(\vec{f})(x):=\sup_{Q\subseteq Q_0}\esssup_{\xi\in Q}|T(\vec{f}\ind_{3Q_0})(\xi)-T(\vec{f}\ind_{3Q})(\xi)|\ind_{Q}(x),\quad x\in\R^d,
\end{equation*}
where the first supremum ranges over all cubes $Q\subseteq\R^d$ contained in $Q_0$. Observe that $\mc{M}_{T,Q_0}(\vec{f})\leq\mc{M}_{T}(\vec{f}\ind_{3Q_0})$.

\begin{proposition}
    \label{prop:pointwise_bound_CZ_grand_maximal}
    For any cube $Q_0\subseteq\R^d$ one has
    \begin{equation*}
        |T(\vec{f}\ind_{3Q_0})(x)|\leq c\Vert T\Vert_{L^{\vec{1}(\R^d)\to L^{\frac{1}{m},\infty}(\R^d)}}\prod_{j=1}^{m}|f_j(x)|+\mc{M}_{T,Q_0}(\vec{f})(x)
    \end{equation*}
    for a.e.~$x\in Q_0$, where the constant $c>0$ depends only on $d$.
\end{proposition}

For any $0<\eta<\infty$, we denote by $M_{\eta}$ the Hardy--Littlewood $\eta$-maximal function on $\R^d$, that is, for $f\in L^{\eta}_{\mathrm{loc}}(\R^d)$ we define
\begin{equation*}
    M_{\eta}f(x):=\sup_{Q}\langle|f|^{\eta}\rangle_{Q}^{\frac{1}{\eta}}\ind_{Q}(x),\quad x\in\R^d,
\end{equation*}
where the supremum ranges over all cubes $Q\subseteq\R^d$. Moreover, we denote by $\mc{M}$ the $m$-(sub)linear maximal function on $\R^d$, that is, for $\vec{f}\in L^{\vec{1}}_{\mathrm{loc}}(\R^d)$ we define
\begin{equation*}
    \mc{M}(\vec{f})(x):=\sup_{Q}\prod_{j=1}^{m}\langle|f_j|\rangle_{Q}\ind_{Q}(x),\quad x\in\R^d,
\end{equation*}
where the supremum ranges over all cubes $Q\subseteq\R^d$.

\begin{proposition}
    \label{prop:Cotlat_CZ_multilinear}
    For all $\eta\in\Big(0,\frac{1}{m}\Big)$ we have
    \begin{equation*}
        \mc{M}_{T}(\vec{f})\leq C(C_{K}\Vert\omega\Vert_{\mathrm{Dini}}+\Vert T\Vert_{L^{\vec{1}}(\R^d)\to L^{\frac{1}{m},\infty}(\R^d)})\mc{M}(\vec{f})+M_{\eta}(|T(\vec{f})|)\quad\text{a.e.~on }\R^d,
    \end{equation*}
    where the constant $C>0$ depends only on $m,d,\eta$ and the doubling constant of $\omega$.
\end{proposition}

Propositions~\ref{prop:weak_bound_multilinear_CZ},~\ref{prop:pointwise_bound_CZ_grand_maximal} and~\ref{prop:Cotlat_CZ_multilinear} and their proofs appear already implicitly in \cite{Damian2018, Li18}. Nevertheless, since the proofs there concern only the case $m=2$ and especially in \cite{Li18} a stronger, $L^{r}$-H\"{o}rmander smoothness assumption on the kernel is considered, we give detailed proofs of Propositions~\ref{prop:weak_bound_multilinear_CZ},~\ref{prop:pointwise_bound_CZ_grand_maximal} and~\ref{prop:Cotlat_CZ_multilinear} in the appendix.

\begin{corollary}
    \label{cor:weak_bound_grand_maximal}
    The grand maximal operator $\mc{M}_{T}$ maps $L^{\vec{1}}(\R^d)\to L^{\frac{1}{m},\infty}(\R^d)$ boundedly with
    \begin{equation*}
        \Vert\mc{M}_{T}\Vert_{L^{\vec{1}}(\R^d)\to L^{\frac{1}{m},\infty}(\R^d)}\leq C(C_{K}\Vert\omega\Vert_{\mathrm{Dini}}+\Vert T\Vert_{L^{\vec{1}}(\R^d)\to L^{\frac{1}{m},\infty}(\R^d)}),
    \end{equation*}
    where the constant $C$ depends only on $m,d$ and the doubling constant of $\omega$.
\end{corollary}

\begin{proof}
    This follows immediately by combining Propositions~\ref{prop:weak_bound_multilinear_CZ} and~\ref{prop:Cotlat_CZ_multilinear}, after observing that the $m$-linear maximal operator $\mc{M}$ maps $L^{\vec{1}}(\R^d)\to L^{\frac{1}{m},\infty}(\R^d)$ boundedly and that the Hardy Littlewood $\eta$-maximal function maps $L^{\frac{1}{m},\infty}(\R^d)\to L^{\frac{1}{m},\infty}(\R^d)$ boundedly for any $0<\eta<\frac{1}{m}$.
\end{proof}

\begin{remark}
    Let $\widetilde{T}$ be the $m$-linear extension of $T$ acting on $m$-tuples of vector valued functions as defined in \eqref{eq:vector_valued_extension}. Then, for all $m$-tuples of (suitable) vector valued functions $\vec{f}=(f_1,\ldots,f_m)$ with $f_j:\R^d\to\mc{H}_j$, we have
    \begin{equation*}
        \widetilde{T}(\vec{f})(x)=\int_{(\R^d)^m}\!K(x,\vec{y})\bigotimes_{j=1}^m f_j(y_j)\,\mathrm{d}\vec{y}
    \end{equation*}
    for all $x\in\R^{d}\setminus\bigcap_{j=1}^{m}\mathrm{supp}(f_j)$.
\end{remark}

\subsection{Sparse domination of multilinear Calder\'{o}n--Zygmund operators}

In this subsection we prove a pointwise domination of the vector valued extensions of multilinear Calder\'{o}n--Zygmund operators by a ``multilinear'' convex body sparse operator. This domination is based on a bootstrapping principle whose starting point is the pointwise sparse domination algorithms in the scalar setting, which in turn relies on the following multilinear adaptation of \cite[Lemma 3.2]{NPTV17}.

For a measurable function $f:\R^d\to\mc{V}$ that is integrable over the cube $Q$ we denote $\llangle f\rrangle_{Q}:=\langle\mc{K}(f)\rangle_{Q}$.

\begin{lemma}
    \label{lem:sparse_domination_inductive_step}
    Let $T:L^{\vec{1}}_{\mathrm{c}}(\R^d)\to L^{0}(\R^d)$ be a multilinear operator. Let $\mc{D}$ be a dyadic grid. Assume that for some cube $Q_0\in\mathcal{D}$ there exists $\lambda\in[1,\infty)$ such that the following holds. For each $0<\varepsilon<1$ and for each $m$-tuple $\vec{f}\in L^{\vec{1}}_{\mathrm{c}}(\R^d)$ of functions supported in $\lambda Q_0$, there exists a collection $\mathcal{G}$ of disjoint cubes in $\mc{D}(Q_0)$ with the following properties:
    \begin{enumerate}
        \item $\sum_{Q\in\mathcal{G}}|Q|\leq\varepsilon|Q_0|$

        \item For any family $\bar{\mathcal{G}}$ of pairwise disjoint cubes in $\mc{D}(Q_0)$ covering $\mathcal{G}$ (meaning that for each $P\in\mathcal{G}$ there is $\bar{P}\in\bar{\mathcal{G}}$ with $P\subseteq\bar{P}$, and no cube in $\bar{\mc{G}}$ is a proper subset of a cube in $\mc{G}$) one has
        \begin{equation*}
            \bigg|T(\vec{f})(x)-\sum_{Q\in\bar{\mathcal{G}}}\ind_{Q}(x)T(\vec{f}\ind_{\lambda Q_0})(x)\bigg|\leq C\prod_{j=1}^{m}\langle |f_j|\rangle_{\lambda Q_0}\quad\text{for a.e. }x\in Q_0,
        \end{equation*}
        where the constant $C>0$ depends only on $\lambda$, $d$, $m$, $\varepsilon$ and $T$.
    \end{enumerate}
    Then, for every $0<\delta<1$ and for every $m$-tuple $\vec{f}\in L^{\vec{1}}_{\mathrm{c}}(\R^d;\vec{\mc{H}})$ of functions supported in $\lambda Q_0$, there exists a collection $\mathcal{G}$ of disjoint cubes in $\mc{D}(Q_0)$ with the following properties:
    \begin{enumerate}
        \item $\sum_{Q\in\mathcal{G}}|Q|\leq\delta|Q_0|$

        \item We have
        \begin{equation*}
            T(\vec{f})(x)\in C\mc{K}\Big(\bigotimes_{j=1}^{m}\llangle f_j\rrangle_{\lambda Q_0}\Big)+\sum_{Q\in\mathcal{G}}\ind_{Q}(x)T(\vec{f}\ind_{\lambda Q})(x)\quad\text{for a.e. }x\in Q_0,
        \end{equation*}
        where the constant $C>0$ depends only on $\lambda$, $d$, $m$, $\delta$, $\vec{n}$ and $T$.
    \end{enumerate}
\end{lemma}

\begin{proof}
   Fix $0<\delta<1$ and a $m$-tuple $\vec{f}\in L^{\vec{1}}_{\mathrm{c}}(\R^d;\vec{\mc{H}})$ of functions supported in $\lambda Q_0$. For each $j=1,\ldots,m$, let $\mc{E}_{j}$ be the John ellipsoid of $\llangle f_j\rrangle_{\lambda Q_0}$. Then, for each $j=1,\ldots,m$, we can find an orthonormal basis $e_{j,1},\ldots,e_{j,n_j}$ for $\mc{H}_j$ as well as $\alpha_{j,1},\ldots,\alpha_{j,n_j}\in[0,\infty)$, such that
   \begin{equation*}
       \mc{E}_j:=\Big\{\sum_{k=1}^{n_j}x_{k}\alpha_k e_k:~x_k\in\F,~\sum_{k=1}^{d}|x_k|^2\leq 1\Big\}.
   \end{equation*}
   Exactly as in \cite[Lemma 3.2]{NPTV17} we have $\langle |f_{j,k}|\rangle_{\lambda Q_0}\leq \sqrt{n_j}\,\alpha_{j,k}$ for each $k=1,\ldots,n_j$ and $j=1,\ldots,m$.

   For each $\vec{k}\in\prod_{j=1}^{m}\{1,\ldots,n_j\}$, apply the hypothesis on the $m$-tuple $f^{\vec{k}}:=(f_{1,k_1},\ldots,f_{m,k_m})$ for $\varepsilon=\frac{\delta}{n}$, where we recall that $n:=\prod_{j=1}^{m}n_j$. Thus, we obtain a family $\mc{G}^{\vec{k}}$ of pairwise disjoint cubes in $\mc{D}(Q_0)$, such that $\sum_{Q\in\mathcal{G}}|Q|\leq\varepsilon|Q_0|$ and for any family $\bar{\mathcal{G}}$ of pairwise disjoint cubes in $\mc{D}(Q_0)$ covering $\mathcal{G}$ one has
    \begin{equation*}
            \bigg|T(f^{\vec{k}})(x)-\sum_{Q\in\bar{\mathcal{G}}}\ind_{Q}(x)T(f^{\vec{k}}\ind_{\lambda Q_0})(x)\bigg|\leq C\prod_{j=1}^{m}\langle |f_{j,k_j}|\rangle_{\lambda Q_0}\quad\text{for a.e. }x\in Q_0,
    \end{equation*}
    where the constant $C>0$ depends only on $\lambda$, $d$, $m$, $\varepsilon$ and $T$.

    Let now $\mc{G}$ be the family of all maximal cubes in $\bigcup_{\vec{k}}\mc{G}^{\vec{k}}$. Observe that
    \begin{equation*}
        \sum_{Q\in\mc{G}}|Q|\leq\sum_{\vec{k}}\sum_{Q\in\mc{G}^{\vec{k}}}|Q|\leq\sum_{\vec{k}}\varepsilon|Q_0|=\delta|Q_0|.
    \end{equation*}
    Moreover, observe that for each $\vec{k}$ we have
    \begin{equation*}
        \bigg|T(f^{\vec{k}})(x)-\sum_{Q\in\mathcal{G}}\ind_{Q}(x)T(f^{\vec{k}}\ind_{\lambda Q_0})(x)\bigg|\leq C\prod_{j=1}^{m}\langle |f_{j,k_j}|\rangle_{\lambda Q_0}\leq Cn^{\frac{1}{2}}\prod_{j=1}^{m}\alpha_{j,k_j}
    \end{equation*}
    for a.e.~$x\in Q_0$, because $\mc{G}$ covers $\mc{G}^{\vec{k}}$. From the definition of the vector valued extension of $T$ we deduce
    \begin{equation*}
        T(\vec{f})(x)\in Cn^{\frac{1}{2}}P+\sum_{Q\in\mathcal{G}}\ind_{Q}(x)T(\vec{f}\ind_{\lambda Q})(x)
    \end{equation*}
    for a.e.~$x\in Q_0$, where
    \begin{equation*}
        P:=\bigg\{\sum_{\vec{k}}x_{\vec{k}}\bigg(\prod_{j=1}^{m}\alpha_{j,k_j}\bigg)e_{\vec{k}}:~x_{\vec{k}}\in\F\text{ with }|x_{\vec{k}}|\leq1\bigg\},
    \end{equation*}
    where $e_{\vec{k}}:=\bigotimes_{j=1}^{m}e_{j,k_j}$. Observe that
    \begin{equation*}
        P\subseteq\sqrt{n}\mc{K}\Big(\Big\{\lambda\bigg(\prod_{j=1}^{m}\alpha_{j,k_j}\bigg)e_{\vec{k}}:~\vec{k}\in\prod_{j=1}^{m}\{1,\ldots,n_j\},~\lambda\in\F\text{ with }|\lambda|=1\Big\}\Big).
    \end{equation*}
    Moreover, as
    \begin{equation*}
        \lambda\bigg(\prod_{j=1}^{m}\alpha_{j,k_j}\bigg)e_{\vec{k}}\in\bigotimes_{j=1}^{m}\mc{E}_j
    \end{equation*}
    for every $\vec{k}\in\prod_{j=1}^{m}\{1,\ldots,n_j\}$ and $\lambda\in\F$ with $|\lambda|=1$, it follows that
    \begin{equation*}
        P\subseteq n^{\frac{1}{2}}\mc{K}\Big(\bigotimes_{j=1}^{m}\mc{E}_j\Big)\subseteq n^{\frac{1}{2}}\mc{K}\Big(\bigotimes_{j=1}^{m}\llangle f_j\rrangle_{\lambda Q_0}\Big),
    \end{equation*}
    concluding the proof.
\end{proof}

\begin{lemma}
    \label{lem:multilinear_CZ_inductive_step}
    Let $T$ be a $m$-linear Calder\'{o}n--Zygmund operator with a kernel having a modulus of continuity that satisfies the Dini condition. Let $\mc{D}$ be a dyadic grid in $\R^d$. Then, for any cube $Q_0\in\mc{D}$, the assumptions of Lemma~\ref{lem:sparse_domination_inductive_step} are satisfied with $\lambda=3$.
\end{lemma}

\begin{proof}
    We adapt \cite[3.2.2]{NPTV17}. Let $Q\in\mc{D}$ and $0<\varepsilon<1$ be arbitrary. Let $\vec{f}\in L^{\vec{1}}_{\mathrm{c}}(\R^d)$ be a $m$-tuple of functions supported in $3Q_0$. Then, by Corollary~\ref{cor:weak_bound_grand_maximal} and Chebyshev's inequality, we have that there exists a constant $C_1>0$ depending only on $T$, $d$ and $m$, such that for the set
    \begin{align*}
        E&:=\Big\{x\in Q_0:~\mc{M}_{T,Q_0}(\vec{f})(x)>\frac{C_1}{\varepsilon^m}\prod_{j=1}^{m}\langle|f_j|\rangle_{3Q_0}\Big\}\\
        &\cup\Big(\bigcup_{j=1}^{m}\Big\{x\in Q_0:~|f_j(x)|>\frac{C_1}{\varepsilon}\langle|f_j|\rangle_{3Q_0}\Big\}\Big)
    \end{align*}
    we have $|E|<\frac{\varepsilon}{2^{d+1}}|Q_0|$. Then, for all $x\in Q_0\setminus E$, by Proposition~\ref{prop:pointwise_bound_CZ_grand_maximal} we have
    \begin{equation*}
        T(\vec{f})(x)\leq\frac{C_2}{\varepsilon^{m}}\prod_{j=1}^{n}\langle|f_j|\rangle_{3Q_0},
    \end{equation*}
    where the constant $C_2>0$ depends only on $T$, $d$ and $m$.

    Let now $\mc{G}$ be the family of all maximal cubes $Q\in\mc{D}(Q_0)$ such that
    \begin{equation*}
        \langle\ind_{E}\rangle_{Q}>\frac{1}{2^{d+1}}.
    \end{equation*}
    Let $\bar{\mc{G}}$ be an arbitrary family of pairwise disjoint cubes in $\mc{D}(Q_0)$ covering $\mc{G}$. Since $\Big|E\setminus\Big(\bigcup_{Q\in\mc{G}}Q\Big)\Big|=0$ by the Lebesgue differentiation theorem, we conclude that
    \begin{equation*}
        T(\vec{f})(x)\leq\frac{C_2}{\varepsilon^{m}}\prod_{j=1}^{n}\langle|f_j|\rangle_{3Q_0},
    \end{equation*}
    for a.e.~$x\in Q_0\setminus\Big(\bigcup_{Q\in\mc{G}}Q\Big)$. In particular, we have
    \begin{equation*}
        T(\vec{f})(x)\leq\frac{C_2}{\varepsilon^{m}}\prod_{j=1}^{n}\langle|f_j|\rangle_{3Q_0},
    \end{equation*}
    for a.e.~$x\in Q_0\setminus\Big(\bigcup_{Q\in\bar{\mc{G}}}Q\Big)$.

    Let now $Q\in\bar{\mc{G}}$ be arbitrary. We show that
    \begin{equation*}
        |T(\vec{f})(x)-T(\vec{f}\1_{3Q_0})(x)|\leq\frac{C_1}{\varepsilon^{m}}\prod_{j=1}^{m}\langle|f_j|\rangle_{3Q_0}
    \end{equation*}
    for a.e.~$x\in Q$.

    If $Q=Q_0$, then $\langle\ind_{E}\rangle_{Q}\leq\frac{1}{2}$. If $Q\neq Q_0$, then by the maximality of the cubes in $\mc{G}$ we deduce that the parent $\hat{Q}$ of $Q$ in $\mc{D}$, which is contained in $Q_0$, satisfies $\langle \ind_{E}\rangle_{\hat{Q}}\leq\frac{1}{2^{d+1}}$, therefore $\langle\ind_{E}\rangle_{Q}\leq\frac{1}{2}$. In particular, $Q$ cannot be a subset of $E$, and therefore there exists $x_0\in Q$ with $\mc{M}_{T,Q_0}(\vec{f})(x_0)\leq\frac{C_1}{\varepsilon^m}\prod_{j=1}^{m}\langle|f_j|\rangle_{3Q_0}$. Since $x_0\in Q$ and $Q\subseteq Q_0$, from the definition of the localized grand maximal operator we deduce
    \begin{equation*}
        |T(\vec{f}\ind_{3Q_0})(x)-T(\vec{f}\ind_{3Q})(x)|\leq\frac{C_1}{\varepsilon^m}\prod_{j=1}^{m}\langle|f_j|\rangle_{3Q_0}\quad\text{for a.e. }x\in Q,
    \end{equation*}
    in other words
    \begin{equation*}
        |T(\vec{f})(x)-T(\vec{f}\ind_{3Q})(x)|\leq\frac{C_1}{\varepsilon^m}\prod_{j=1}^{m}\langle|f_j|\rangle_{3Q_0}\quad\text{for a.e. }x\in Q,
    \end{equation*}
    concluding the proof.
\end{proof}

\begin{lemma}
    \label{lem:sparse_operator_well_defined}
    Let $\mc{S}$ be an $\eta$-sparse family of cubes in $\R^d$ for some $0<\eta<1$ and let $\vec{F}\in L^{\vec{1}}_{\mathrm{c}}(\R^d;\mc{K}(\vec{\mc{H}}))$. Then, for a.e.~$x\in\R^d$, the Minkowski sum
    \begin{equation*}
        \sum_{Q\in\mc{S}}\mc{K}\Big(\bigotimes_{j=1}^{m}\langle F_j\rangle_{Q}\Big)\1_{Q}(x)
    \end{equation*}
    is well-defined as a bounded, convex, complex symmetric subset of $\mc{H}$.
\end{lemma}

\begin{proof}
This is an adaptation of \cite[Lemma 2.5]{NPTV17}. It suffices to show that
\begin{equation*}
    \sum_{Q\in\mc{S}}\Big\Vert\mc{K}\Big(\bigotimes_{j=1}^{m}\langle F_j\rangle_{Q}\Big)\Big\Vert_{\mc{H}}\1_{Q}(x)<\infty\quad\text{for a.e. }x\in\R^d,
\end{equation*}
in other words that
\begin{equation*}
    \sum_{Q\in\mc{S}}\prod_{j=1}^{m}\Vert\langle F_j\rangle_{Q}\Vert_{\mc{H}_j}\ind_{Q}(x)<\infty\quad\text{for a.e. }x\in\R^d.
\end{equation*}
In view of Proposition~\ref{prop:aumannintegralbound}, it suffices to show that
\begin{equation*}
    \sum_{Q\in\mc{S}}\prod_{j=1}^{m}\langle\Vert F_j\Vert_{\mc{H}_j}\rangle_{Q}\ind_{Q}(x)<\infty\quad\text{for a.e. }x\in\R^d.
\end{equation*}
By a standard application of the $3^d$-lattice trick, we can without loss of generality assume that $\mc{S}$ is a subset of some dyadic grid $\mc{D}$ in $\R^d$. Set
\begin{equation*}
    \mc{S}_1:=\{Q\in\mc{S}:~\ell(Q)\leq1\},\quad \mc{S}_2:=\mc{S}\setminus\mc{S}_1.
\end{equation*}
Exactly as in \cite[Lemma 2.5]{NPTV17} we have that the set $\{Q\in\mc{S}_1:~x\in Q\}$ is finite, for a.e.~$x\in\R^d$. Thus, we only have to prove that
\begin{equation*}
    \sum_{Q\in\mc{S}_2}\prod_{j=1}^{m}\langle\Vert F_j\Vert_{\mc{H}_j}\rangle_{Q}\ind_{Q}(x)<\infty\quad\text{for a.e. }x\in\R^d.
\end{equation*}
It is clear that there exists a finite subset $\mc{F}$ of $\mc{D}$ such that $\ell(Q)=2$, for all $Q\in\mc{S}$ and
\begin{equation*}
    \bigcup_{j=1}^{m}\mathrm{supp}(F_j)\subseteq\bigcup_{Q\in\mc{F}}Q.
\end{equation*}
Thus, we have
\begin{align*}
    \sum_{Q\in\mc{S}_2}\prod_{j=1}^{m}\langle\Vert F_j\Vert_{\mc{H}_j}\rangle_{Q}&\leq
    \sum_{Q\in\mc{F}}\sum_{\substack{R\in\mc{D}\\Q\subseteq\R}}\prod_{j=1}^{m}\langle\Vert F_j\Vert_{\mc{H}_j}\rangle_{Q}\\
    &\leq
    \sum_{Q\in\mc{F}}\sum_{k=2}^{\infty}2^{-mdk}\prod_{j=1}^{m}\int_{\R^d}\Vert F_j(x)\Vert_{\mc{H}_j}\mathrm{d}x<\infty,
\end{align*}
concluding the proof.
\end{proof}

If $\mc{S}$ is an $\eta$-sparse family of cubes in $\R^d$ for some $0<\eta<1$, then for  we consider the \emph{multilinear convex body sparse operator} $A_{\mc{S}}$ acting on $m$-tuples $\vec{F}\in L^{\vec{1}}_{\mathrm{c}}(\R^d;\mc{K}(\vec{\mc{H}}))$ by
\begin{equation*}
A_{\mc{S}}(\vec{F})(x):=\sum_{Q\in\mc{S}}\mc{K}\Big(\bigotimes_{j=1}^{m}\langle F_j\rangle_{Q}\Big)\1_{Q}(x)\quad\text{for~ a.e. }x\in\R^d.
\end{equation*}
By Lemma \ref{lem:sparse_operator_well_defined} and \cite[Theorem 3.3]{BC23} we have that $A_{\mc{S}}(\vec{F}):\R^d\to\mc{K}(\mc{H})$ is a well-defined measurable map. For $\vec{f}\in L^{\vec{1}}_{\mathrm{c}}(\R^d;\vec{\mc{H}})$ we denote $A_{\mc{S}}(\vec{f}):=A_{\mc{S}}(\mc{K}(\vec{f}))$.

\begin{theorem}
    \label{thm:sparse_domination_multilinear_CZ}
    Let $T$ be a $m$-linear Calder\'{o}n--Zygmund operator with a kernel having a modulus of continuity that satisfies the Dini condition. Then, for any cube $Q_0\subseteq\R^d$ and for any $m$-tuple of functions $\vec{f}\in L^{\vec{1}}_{\mathrm{c}}(\R^d;\vec{\mc{H}})$ supported in $Q_0$, there exists a martingale $\frac{1}{2\cdot 3^{d}}$-sparse family $\mc{S}$ of cubes in $\R^d$ contained in $3Q_0$ such that
    \begin{equation*}
        T(\vec{f})(x)\in CA_{\mc{S}}(\vec{f})(x)\quad\text{for a.e. }x\in Q_0,
    \end{equation*}
    where the constant $C$ depends only on $d$, $m$, $\vec{n}$ and $T$.
\end{theorem}

\begin{proof}
We adapt \cite[Theorem 3.4]{NPTV17}. Let $\vec{f}\in L^{\vec{1}}_{\mathrm{c}}(\R^d;\vec{\mc{H}})$ be a $m$-tuple of functions supported in $Q_0$. Let $\mc{D}$ be a dyadic grid with $Q_0\in\mc{D}$. Set $\mc{G}_0:=\{Q_0\}$. We apply Lemma~\ref{lem:multilinear_CZ_inductive_step} for the cube $Q_0$ and the $m$-tuple $\vec{f}$ with $\delta=\frac{1}{2}$. Thus, we obtain a family $\mc{G}_1$ of pairwise disjoint cubes in $\mc{D}(Q_0)$ such that
\begin{equation*}
    \sum_{Q\in\mc{G}_1}|Q|\leq\frac{1}{2}|Q_0|
\end{equation*}
and
\begin{equation*}
    T(\vec{f})(x)\in C\mc{K}\Big(\bigotimes_{j=1}^{m}\llangle f_j\rrangle_{3 Q_0}\Big)+\sum_{Q\in\mathcal{G}_1}\ind_{Q}(x)T(\vec{f}\ind_{3 Q})(x)\quad\text{for a.e. }x\in Q_0,
\end{equation*}
where the constant $C>0$ depends only on $d$, $m$, $\vec{n}$ and $T$.

Next, for each $Q_1\in\mc{G}_1$, we repeat this step for the cube $Q_1$ and the $m$-tuple $\vec{f}\ind_{3Q_1}$ with $\delta=\frac{1}{2}$, obtaining a family $\mc{G}_{2,Q_1}$ of pairwise disjoint cubes in $\mc{D}(Q_1)$ such that
\begin{equation*}
    \sum_{Q\in\mc{G}_{2,Q_1}}|Q|\leq\frac{1}{2}|Q_1|
\end{equation*}
and
\begin{equation*}
    T(\vec{f})(x)\in C\mc{K}\Big(\bigotimes_{j=1}^{m}\llangle f_j\rrangle_{3 Q_1}\Big)+\sum_{Q\in\mathcal{G}_{2,Q_2}}\ind_{Q}(x)T(\vec{f}\ind_{3 Q})(x)\quad\text{for a.e. }x\in Q_1,
\end{equation*}
with the same constant $C$ as above. Setting $\mc{G}_2:=\bigcup_{Q\in\mc{G}_1}\mc{G}_{2,Q}$ we deduce
\begin{equation*}
    T(\vec{f})(x)\in C\sum_{Q\in\mc{G}_0\cup\mc{G}_1}\mc{K}\Big(\bigotimes_{j=1}^{m}\llangle f_j\rrangle_{3 Q}\Big)\ind_{Q}(x)+\sum_{Q\in\mathcal{G}_2}\ind_{Q}(x)T(\vec{f}\ind_{3 Q})(x)\quad\text{for a.e. }x\in Q_0.
\end{equation*}
Now we repeat this steps for the cubes in $\mc{G}_2$ and continue this process inductively, obtaining successively families $\mc{G}_k$, $k=1,2,3,\ldots$. After the $k$-th step we have
\begin{equation*}
     T(\vec{f})(x)\in C\sum_{Q\in\bigcup_{\ell=0}^{k-1}\mc{G}_\ell}\mc{K}\Big(\bigotimes_{j=1}^{m}\llangle f_j\rrangle_{3 Q}\Big)\ind_{Q}(x)+\sum_{Q\in\mathcal{G}_k}\ind_{Q}(x)T(\vec{f}\ind_{3 Q})(x)\quad\text{for a.e. }x\in Q_0,
\end{equation*}
for all $k=1,2,\ldots$. Set $\mc{G}:=\bigcup_{k=0}^{\infty}\mc{G}_k$. Since
\begin{equation*}
    \Big|\bigcap_{k=1}^{\infty}\bigcup_{Q\in\mc{G}_k}Q\Big|=\lim_{k\to\infty}\Big|\bigcup_{Q\in\mc{G}_k}Q\Big|=0,
\end{equation*}
we deduce that
\begin{equation*}
     T(\vec{f})(x)\in C\sum_{Q\in\mc{G}}\mc{K}\Big(\bigotimes_{j=1}^{m}\llangle f_j\rrangle_{3 Q}\Big)\ind_{Q}(x)\quad\text{for a.e. }x\in Q_0.
\end{equation*}
Set $\mc{S}:=\{3Q:~Q\in\mc{G}\}$. Then, $\mc{S}$ is a martingale $\frac{1}{2\cdot 3^d}$-sparse family with
\begin{equation*}
     T(\vec{f})(x)\in C\sum_{Q\in\mc{S}}\mc{K}\Big(\bigotimes_{j=1}^{m}\llangle f_j\rrangle_{ Q}\Big)\ind_{Q}(x)\quad\text{for a.e. }x\in Q_0,
\end{equation*}
concluding the proof.
\end{proof}

\subsection{Matrix weighted bounds for multilinear convex body sparse operators}

In this subsection we prove matrix weighted bounds for our ``multilinear'' convex body sparse operator.

For each $C\in\mc{K}(\mc{V})$ and $v\in\mc{V}$ we denote
\begin{equation*}
    \langle C,v\rangle_{\mc{V}}:=\{\langle u,v\rangle_{\mc{V}}:~u\in\mc{K}\}\subseteq\F.
\end{equation*}
Then we have $\langle C,v\rangle_{\mc{V}}\in\mc{K}(\F)$.

\begin{theorem}\label{thm:multilinearconvexbodyweightedbounds}
Let $\vec{p}\in(1,\infty]^m$ satisfy $\tfrac{1}{p}:=\sum_{j=1}^m\tfrac{1}{p_j}>0$, and let $\vec{W}\in A_{\vec{p}}$. If $\mc{S}$ is an $\eta$-sparse collection of cubes in $\R^d$ for some $0<\eta<1$, then
\begin{align*}
&\|A_{\mc{S}}\|_{L^{\vec{p}}_{\vec{W}}(\R^d;\mc{K}(\vec{\mc{H}}))\to L^p_{\mb{W}}(\R^d;\mc{K}(\mc{H}))}
\lesssim_{d,m,\vec{n},\vec{p},\eta}[\vec{W}]_{\vec{p}}[\mb{W}]_{\mathrm{FW}_p}^{(p-1)^{+}}\prod_{j=1}^m[W_j^{-1}]_{\mathrm{FW}_{p_j'}}^{\frac{p_j'}{p_j}}\\
&\lesssim_{d,m,\vec{n},p}[\vec{W}]_{\vec{p}}^{\max(1,p)+\sum_{j=1}^m\frac{p_j'}{p_j}}.
\end{align*}
\end{theorem}

\begin{proof}[Proof of Theorem~\ref{thm:multilinearconvexbodyweightedbounds}]
The second bound follows immediately from Theorem~\ref{thm:needed_reverse_Holder} (and Remark~\ref{rem:sharppdendence}), so we only have to prove the first.

First of all, by monotone convergence (see \cite[Proposition~3.7]{Ni24b}), we may assume that $\mc{S}$ is finite. For each cube $Q\subseteq\R^d$ we let $A_{j, Q}:=A_{W_j^{-1},Q,p_j'}$, and $A_{Q}:=A_{\mathbf{W},Q,p}$. We now distinguish two cases.

\medskip

\textbf{Case 1.: $p\geq 1$}. Let $h\in L^{p'}_{\mb{W}^{-1}}(\R^d;\mc{H})$ be arbitrary. We need to show that
\begin{align*}
    &\Big|\int_{\R^d}\langle g(x),h(x)\rangle_{\mc{H}}\mathrm{d}x\Big|\\
    &\lesssim_{d,m,\vec{n},\vec{p},\varepsilon}[\vec{W}]_{\vec{p}}[\mb{W}]_{\mathrm{FW}_p}^{(p-1)^{+}}\prod_{j=1}^m[W_j^{-1}]_{\mathrm{FW}_{p_j'}}^{\frac{p_j'}{p_j}}\Vert\vec{F}\Vert_{L^{\vec{p}}_{\vec{W}}(\R^d;\mc{K}(\vec{\mc{H}}))}\Vert h\Vert_{L^{p'}_{\mb{W}^{-1}}(\R^d;\mc{H})},
\end{align*}
for every $g\in S^{0}(\R^d;A_{\mc{S}}(\vec{F}))$. It is clear that
\begin{equation*}
    \Big|\int_{\R^d}\langle g(x),h(x)\rangle_{\mc{H}}\mathrm{d}x\Big|\leq\int_{\R^d}|\langle A_{\mc{S}}(\vec{F})(x),h(x)\rangle_{\mc{H}}|\mathrm{d}x,
\end{equation*}
for any $g\in S^{0}(\R^d;A_{\mc{S}}(\vec{F}))$. Thus, we only have to bound $\int_{\R^d}|\langle A_{\mc{S}}(\vec{F})(x),h(x)\rangle_{\mc{H}}|\mathrm{d}x$. We estimate
\begin{align*}
    &\int_{\R^d}|\langle A_{\mc{S}}(\vec{F})(x),h(x)\rangle_{\mc{H}}|\mathrm{d}x=
    \int_{\R^d}\Big|\sum_{Q\in\mc{S}}\Big\langle\mc{K}\Big(\bigotimes_{j=1}^{m}\langle F_j\rangle_{Q}\Big),h(x)\Big\rangle_{\mc{H}}\ind_{Q}(x)\Big|\mathrm{d}x\\
    &=\int_{\R^d}\Big|\sum_{Q\in\mc{S}}\Big\langle A_{Q}\Big(\bigotimes_{j=1}^m A_{j,Q}\Big)\mc{K}\Big(\bigotimes_{j=1}^{m}\langle A_{j,Q}^{-1}F_j\rangle_{Q}\Big),A_{Q}^{-1}h(x)\Big\rangle_{\mc{H}}\ind_{Q}(x)\Big|\mathrm{d}x\\
    &\lesssim_{m,\vec{n},\vec{p}}[\vec{W}]_{\vec{p}}\int_{\R^d}\sum_{Q\in\mc{S}}\Big\Vert\mc{K}\Big(\bigotimes_{j=1}^{m}\langle A_{j,Q}^{-1}F_j\rangle_{Q}\Big)\Big\Vert_{\mc{H}}\cdot\Vert A_{Q}^{-1}h(x)\Vert_{\mc{H}}\ind_{Q}(x)\mathrm{d}x\\
    &=[\vec{W}]_{\vec{p}}\sum_{Q\in\mc{S}}\Big\Vert\mc{K}\Big(\bigotimes_{j=1}^{m}\langle A_{j,Q}^{-1}F_j\rangle_{Q}\Big)\Big\Vert_{\mc{H}}\cdot\langle\Vert A_{Q}^{-1}h\Vert_{\mc{H}}\rangle_{Q}\cdot|Q|\\
    &\leq[\vec{W}]_{\vec{p}}\sum_{Q\in\mc{S}}\prod_{j=1}^{m}\langle\Vert A_{j,Q}^{-1}F_j\Vert_{\mc{H}_j}\rangle_{Q}\cdot\langle\Vert A_{Q}^{-1}h\Vert_{\mc{H}}\rangle_{Q}\cdot|Q|\\
    &\leq[\vec{W}]_{\vec{p}}\Big(\sum_{Q\in\mc{S}}\prod_{j=1}^{m}\langle\Vert A_{j,Q}^{-1}F_j\Vert_{\mc{H}_j}\rangle_{Q}^{p}|Q|\Big)^{\frac{1}{p}}\Big(
    \sum_{Q\in\mc{S}}\langle\Vert A_{Q}^{-1}h\Vert_{\mc{H}}\rangle_{Q}^{p'}|Q|\Big)^{\frac{1}{p'}}.
\end{align*}
Using Proposition~\ref{prop:auxiliary_maximal} we obtain
\begin{align*}
    &\Big(\sum_{Q\in\mc{S}}\prod_{j=1}^{m}\langle\Vert A_{j,Q}^{-1}F_j\Vert_{\mc{H}_j}\rangle_{Q}^{p}|Q|\Big)^{\frac{1}{p}}\lesssim_{\eta,p}
    \Big(\sum_{Q\in\mc{S}}\prod_{j=1}^{m}\langle\Vert A_{j,Q}^{-1}F_j\Vert_{\mc{H}_j}\rangle_{Q}^{p}|E_{Q}|\Big)^{\frac{1}{p}}\\
    &\leq\Vert\widetilde{\mc{M}}_{\vec{W}}(\vec{F})\Vert_{L^{p}(\R^d)}
    \lesssim_{d,m,\vec{n},\vec{p}}\prod_{j=1}^m[W_j^{-1}]_{\mathrm{FW}_{p_j'}}^{\frac{p_j'}{p_j}}\Vert\vec{F}\Vert_{L^{\vec{p}}_{\vec{W}}(\R^d;\mc{K}(\vec{\mc{H}}))}.
\end{align*}
As for the last term, let $r\in(1,\infty)$ with
\begin{equation*}
    r'=2^{d+1}\sup_{\substack{A\in\mc{L}(\C^{\vec{n}})\\A\text{ invertible}}}[\Vert\mathbf{W} A\Vert_{\mc{H}\to\mc{H}}^{p}]_{\mathrm{FW}},
\end{equation*}
where as a consequence of Lemma~\ref{lem:reducing_operator_on_operator} we observe that
\begin{equation*}
    \sup_{\substack{A\in\mc{L}(\C^{\vec{n}})\\A\text{ invertible}}}[\Vert\mathbf{W} A\Vert_{\mc{H}\to\mc{H}}^{p}]^{\frac{1}{p}}_{\mathrm{FW}}\lesssim_{n}[\mathbf{W}]_{\mathrm{FW}_{p}}<\infty.
\end{equation*}
Thus, by the sharp reverse H\"older inequality \cite[Theorem 2.3]{HPR12}, we obtain
\begin{align*}
    &\Big(
    \sum_{Q\in\mc{S}}\langle\Vert A_{Q}^{-1}h\Vert_{\mc{H}}\rangle_{Q}^{p'}|Q|\Big)^{\frac{1}{p'}}
    \leq\Big(\sum_{Q\in\mc{S}}\langle\Vert A_{Q}^{-1}\mathbf{W}\Vert_{\mc{H}\to\mc{H}}\cdot\Vert\mathbf{W}h\Vert_{\mc{H}}\rangle_{Q}^{p'}|Q|\Big)^{\frac{1}{p'}}\\
    &\leq\Big(\sum_{Q\in\mc{S}}\langle\Vert A_{Q}^{-1}\mathbf{W}\Vert_{\mc{H}\to\mc{H}}^{rp}\rangle_{Q}^{\frac{p'}{rp}}\cdot\langle\Vert\mathbf{W}h\Vert_{\mc{H}}^{(rp)'}\rangle_{Q}^{\frac{p'}{(rp)'}}|Q|\Big)^{\frac{1}{p'}}\\
    &\lesssim_{d,p}\Big(\sum_{Q\in\mc{S}}\langle\Vert A_{Q}^{-1}\mathbf{W}\Vert_{\mc{H}\to\mc{H}}^{p}\rangle_{Q}^{\frac{p'}{p}}\cdot\langle\Vert\mathbf{W}h\Vert_{\mc{H}}^{(rp)'}\rangle_{Q}^{\frac{p'}{(rp)'}}|Q|\Big)^{\frac{1}{p'}}\\
    &\lesssim_{n,p}\Big(\sum_{Q\in\mc{S}}\langle\Vert\mathbf{W}h\Vert_{\mc{H}}^{(rp)'}\rangle_{Q}^{\frac{p'}{(rp)'}}|Q|\Big)^{\frac{1}{p'}}
    \lesssim_{\varepsilon,p}\Big(\sum_{Q\in\mc{S}}\langle\Vert\mathbf{W}h\Vert_{\mc{H}}^{(rp)'}\rangle_{Q}^{\frac{p'}{(rp)'}}|E_{Q}|\Big)^{\frac{1}{p'}}\\
    &\leq\Big(\int_{\R^d}M(\Vert\mathbf{W}h\Vert_{\mc{H}}^{(rp)'})(x)^{\frac{p'}{(rp)'}}\mathrm{d}x\Big)^{\frac{1}{p'}},
\end{align*}
where $M$ is the usual Hardy--Littlewood maximal function on $\R^d$. Note that $\frac{p'}{(rp)'}>1$, because $r>1$. Thus, we have
\begin{equation*}
    \Big(\int_{\R^d}M(\Vert\mathbf{W}h\Vert_{\mc{H}}^{(rp)'})(x)^{\frac{p'}{(rp)'}}\mathrm{d}x\Big)^{\frac{1}{p'}}
    \lesssim_{d,p}\Big[\Big(\frac{p'}{(rp)'}\Big)'\Big]^{\frac{1}{p'}}\Vert h\Vert_{L^{p'}(\R^d;\mc{H})}.
\end{equation*}
As in the proof of Proposition~\ref{prop:auxiliary_maximal} we have
\begin{equation*}
    \Big[\Big(\frac{p'}{(rp)'}\Big)'\Big]^{\frac{1}{p'}}\lesssim_{n,d,p}[\mb{W}]_{\mathrm{FW}_p}^{\frac{p}{p'}},
\end{equation*}
concluding the proof in this case.

\medskip

\textbf{Case 2.: $p<1$}. Using that
\begin{equation*}
    \Vert x\Vert_{\ell^1(\N)}\leq\Vert x\Vert_{\ell^{p}(\N)},
\end{equation*}
for all sequences $x=(x_{k})_{k\in\N}$ of real numbers, we estimate
\begin{align*}
    &\Vert A_{\mc{S}}(\vec{F})\Vert_{L^{p}_{\mathbf{W}}(\R^d;\mc{H})}^{p}=
    \int_{\R^d}\Big\Vert\mathbf{W}(x)\sum_{Q\in\mc{S}}\mc{K}\Big(\bigotimes_{j=1}^{m}\langle F_j\rangle_{Q}\Big)\ind_Q(x)\Big\Vert_{\mc{H}}^{p}\,\mathrm{d}x\\
    &\leq\int_{\R^d}\sum_{Q\in\mc{S}}\Big\Vert\mathbf{W}(x)\mc{K}\Big(\bigotimes_{j=1}^{m}\langle F_j\rangle_{Q}\Big)\Big\Vert_{\mc{H}}^{p}\ind_Q(x)\,\mathrm{d}x\\
    &\leq\int_{\R^d}\sum_{Q\in\mc{S}}\Big\Vert\mathbf{W}(x)\Big(\bigotimes_{j=1}^{m}A_{j, Q}\Big)\Big\Vert_{\mc{H}\to\mc{H}}^{p}\cdot\Big\Vert\mc{K}\Big(\bigotimes_{j=1}^{m}\langle A_{j,Q}^{-1}F_j\rangle_{Q}\Big)\Big\Vert_{\mc{H}}^{p}\ind_Q(x)\,\mathrm{d}x\\
    &\lesssim_{p,m,\vec{n}}[\vec{W}]_{\vec{p}}\sum_{Q\in\mc{S}}\prod_{j=1}^{m}\langle\Vert A_{j,Q}^{-1}F_j\Vert_{\mc{H}_j}\rangle_{Q}^{p}|Q|.
\end{align*}
Using now the sparseness of $\mc{S}$ and then Proposition~\ref{prop:auxiliary_maximal}, we obtain
\begin{align*}
    &\sum_{Q\in\mc{S}}\prod_{j=1}^{m}\langle\Vert A_{j,Q}^{-1}F_j\Vert_{\mc{H}_j}\rangle_{Q}^{p}|Q|
    \lesssim_{\eta}\sum_{Q\in\mc{S}}\prod_{j=1}^{m}\langle\Vert A_{j,Q}^{-1}F_j\Vert_{\mc{H}_j}\rangle_{Q}^{p}|E_{Q}|\\
    &\leq\int_{\R^d}\widetilde{\mc{M}}_{\vec{W}}(\vec{F})(x)^{p}\mathrm{d}x
    \lesssim_{d,m,\vec{n},\vec{p}}\Big(\prod_{j=1}^{m}[W_{j}^{-1}]_{\mathrm{FW}_{p_j'}}^{\frac{p_j'}{p_j}}\Big)^{p}\Vert\vec{F}\Vert_{L^{\vec{p}}_{\vec{W}}(\R^d;\mc{K}(\vec{\mc{H}}))}^{p},
\end{align*}
concluding the proof.
\end{proof}

\begin{corollary}\label{cor:czoquantitativebound}
    Let $T$ be a $m$-linear Calder\'{o}n--Zygmund operator with a kernel having a modulus of continuity that satisfies the Dini condition. Let $\vec{p}\in(1,\infty]^{m}$ with $p<\infty$, and let $\vec{W}\in A_{\vec{p}}$. Then, we have
    \begin{align*}
        \Vert T\Vert_{L^{\vec{p}}_{\vec{W}}(\R^d;\vec{\mc{H}})\to L^{p}_{\mathbf{W}}(\R^d;\mc{H})}&\lesssim_{T,d,m,\vec{n},\vec{p}}[\vec{W}]_{\vec{p}}[\mb{W}]_{\mathrm{FW}_p}^{(p-1)^{+}}\prod_{j=1}^m[W_j^{-1}]_{\mathrm{FW}_{p_j'}}^{\frac{p_j'}{p_j}}\\
        &\lesssim_{d,m,\vec{n},\vec{p}}[\vec{W}]_{\vec{p}}^{\max(1,p)+\sum_{j=1}^m\frac{p_j'}{p_j}}.
    \end{align*}
\end{corollary}

\begin{proof}
    Let $\vec{f}\in L^{\vec{1}}_{\mathrm{c}}(\R^d;\vec{\mc{H}})$ be arbitrary. By monotone convergence it suffices to prove that for every cube $Q_0\subseteq\R^d$ containing $\bigcup_{j=1}^{m}\mathrm{supp}(f_j)$ we have
    \begin{equation*}
        \Big(\int_{Q_0}\Vert\mathbf{W}(x)T(\vec{f})(x)\Vert_{\mc{H}}^{p}\mathrm{d}x\Big)^{\frac{1}{p}}\lesssim_{T,d,m,\vec{n},\vec{p}}[\vec{W}]_{\vec{p}}[\mb{W}]_{\mathrm{FW}_p}^{(p-1)^{+}}\prod_{j=1}^m[W_j^{-1}]_{\mathrm{FW}_{p_j'}}^{\frac{p_j'}{p_j}}\Vert\vec{f}\Vert_{L^{\vec{p}}_{\vec{W}}(\R^d;\vec{\mc{H})}}.
    \end{equation*}

    Let $Q_0\subseteq\R^d$ be any cube containing $\bigcup_{j=1}^{m}\mathrm{supp}(f_j)$. By Theorem~\ref{thm:sparse_domination_multilinear_CZ} we have that there exists a martingale $\frac{1}{2\cdot 3^d}$-sparse family $\mc{S}$ of cubes in $\R^d$ such that
    \begin{equation*}
        T(\vec{f})(x)\in CA_{\mc{S}}(\vec{f})(x)\quad\text{for a.e. }x\in Q_0,
    \end{equation*}
    where the constant $C$ depends only on $T,d,m$ and $\vec{n}$. Thus, by Theorem~\ref{thm:multilinearconvexbodyweightedbounds} we obtain
    \begin{align*}
        &\Big(\int_{Q_0}\Vert\mathbf{W}(x)T(\vec{f})(x)\Vert_{\mc{H}}^{p}\mathrm{d}x\Big)^{\frac{1}{p}}\lesssim_{T,d,m,\vec{n}}\Big(\int_{Q_0}\Vert\mathbf{W}(x)A_{\mc{S}}(\vec{f})(x)\Vert_{\mc{H}}^{p}\mathrm{d}x\Big)^{\frac{1}{p}}\\
        &\lesssim_{d,m,\vec{n},\vec{p}}[\vec{W}]_{\vec{p}}[\mb{W}]_{\mathrm{FW}_p}^{(p-1)^{+}}\prod_{j=1}^m[W_j^{-1}]_{\mathrm{FW}_{p_j'}}^{\frac{p_j'}{p_j}}\Vert\vec{f}\Vert_{L^{\vec{p}}_{\vec{W}}(\R^d;\vec{\mc{H}})},
    \end{align*}
    concluding the proof.
\end{proof}
\subsection{Non-degeneracy}
Finally, we prove that the boundedness of non-degenerate $m$-linear Calder\'on--Zygmund operators implies the multilinear Muckenhoupt condition.

Such a result for matrix weights in the case $m=1$ is proven in \cite[Theorem~5.2]{Go03}. However, to treat the case $m>1$, we require an idea that is more involved. In \cite{Go03}, the operator $T$ is applied twice in order to obtain an operator satisfying appropriate estimates with respect to the averaging operator, whose boundedness characterizes the Muckenhoupt condition. In the multilinear setting, an operator maps $m$ functions into a single function and, hence, cannot be applied twice. In order to deal with this, we will split the argument into two steps in a way similar to what is done in \cite{Ni24, Le24b}, and use a general factorization technique to reduce $m$ functions back to one.

\begin{definition}
Let $T$ be an $m$-linear Calder\'on--Zygmund operator. We call $T$ \emph{directionally non-degenerate} if there is a $C>0$ such that for all cubes $Q$ there is a cube $Q'$ satisfying the following properties:
\begin{enumerate}[(a)]
    \item\label{it:dirnondegdef1} For all $0\leq h_1,\ldots,h_m\in L^\infty_c(\R^d)$ supported in $Q'$ with $\prod_{j=1}^m h_j=\ind_{Q'}$, we have
    \[
    1\leq C |T(\vec{h})(x)|
    \]
    for a.e. $x\in Q$.
    \item\label{it:dirnondegdef2} For all $0<\alpha<1$ there is a mapping $S:Q'\times Q^m\to\C$ satisfying
    \[
    |S(x,\vec{y})|\leq|Q|^{-m}
    \]
    a.e. for which for all $\vec{f}\in L^{\vec{\infty}}_c(\R^d;\vec{\mc{H}})$ supported in $Q$ we have
    \[
    \bigotimes_{j=1}^m\langle f_j\rangle_Q=(1-\alpha)C\widetilde{T}(\vec{f})(x)+\alpha\int_{Q^m}S(x,\vec{y})\bigotimes_{j=1}^mf_j(y_j)\,\mathrm{d}\vec{y}
    \]
    for a.e. $x\in Q'$.
\end{enumerate}
\end{definition}

Our main result of this section is as follows:
\begin{theorem}\label{thm:nondeglebesgue}
Let $\vec{p}\in(1,\infty]^m$ with $p<\infty$ and let $\vec{W}$ be matrix weights for which $W_j^{-1}$ is locally $p_j'$-integrable for $j=1,\ldots,m$. If $T$ is a non-degenerate $m$-linear Calder\'on--Zygmund operator satisfying
\[
\widetilde{T}:L^{\vec{p}}_{\vec{W}}(\R^d;\vec{\mc{H}})\to L^p_{\mb{W}}(\R^d;\mc{H}),
\]
then $\vec{W}\in A_{\vec{p}}$, with
\[
[\vec{W}]_{\vec{p}}\lesssim \|\widetilde{T}\|_{L^{\vec{p}}_{\vec{W}}(\R^d;\vec{\mc{H}})\to L^p_{\mb{W}}(\R^d;\mc{H})}^2.
\]
\end{theorem}

Before we get to the proof, we first define a scalar multilinear non-degeneracy condition which reduces back to \cite[Definition~3.11]{Le24b} when $m=1$.
\begin{definition}
Let $T$ be an $m$-linear Calder\'on--Zygmund operator. We call $T$ \emph{non-degenerate} if there is a $C>0$ such that for all cubes $Q$ there is a cube $Q'$ satisfying the following properties:
\begin{enumerate}[(a)]
    \item\label{it:dirnondegscalardef1} For all $0\leq h_1,\ldots,h_m\in L^\infty_c(\R^d)$ supported in $Q'$ with $\prod_{j=1}^m h_j=\ind_{Q'}$, we have
    \[
    1\leq C |T(\vec{h})(x)|
    \]
    for a.e. $x\in Q$.
    \item\label{it:dirnondegscalardef2} For all non-negative $\vec{h}\in L^{\vec{\infty}}_c(\R^d)$ supported in $Q$ we have
    \[
    \prod_{j=1}^m\langle h_j\rangle_Q\leq C|T(\vec{h})(x)|
    \]
    for a.e. $x\in Q'$.
\end{enumerate}
\end{definition}
The boundedness of such an operator implies the scalar Muckenhoupt condition. As a matter of fact, we only need a weak-type boundedness, where
\[
\|f\|_{X_{\text{weak}}}:=\sup_{\lambda>0}\|\ind_{\{|f|>\lambda\}}\lambda\|_X.
\]
Note that for any $E\subseteq\R^d$ we have $\|\ind_E\|_X=\|\ind_E\|_{X_{\text{weak}}}$.
\begin{proposition}\label{prop:scalarnondeg}
Let $\vec{X}$ be Banach function spaces over $\R^d$ with the Fatou property and let $T$ be an $m$-linear non-degenerate Calder\'on--Zygmund operator for which
\[
T:\vec{X}\to X_{\emph{weak}}.
\]
Then $T_Q:\vec{X}\to X$ uniformly for all cubes $Q$, with
\[
\sup_Q\|T_Q\|_{\vec{X}\to X}\leq C^2\|T\|^2_{\vec{X}\to X_{\emph{weak}}}.
\]
In particular, $\ind_Q\in X$ and $\ind_Q\in X_j'$ for $j=1,\ldots,m$ for all cubes $Q$.
\end{proposition}
\begin{proof}
Let $Q$ be a cube, let $Q'$ be as in the definition of the non-degeneracy of $T$, and let $0\leq h_1,\ldots,h_m\in L^\infty_c(\R^d)$ supported in $Q'$ with $\prod_{j=1}^m h_j=\ind_{Q'}$. Then, by property \ref{it:dirnondegscalardef1} and the ideal property of $X_{\text{weak}}$, $\ind_{Q'}\in X$ with
\[
\|\ind_Q\|_X=\|\ind_{Q'}\|_{X_{\text{weak}}}\leq C\|T\|_{\vec{X}\to X_{\text{weak}}}\prod_{j=1}^m\|h_j\|_{X_j}.
\]
By the ideal property of the $X_j$, this is true for all $0\leq h_j\in X_j$ with $\ind_{Q'}\leq\prod_{j=1}^m h_j$. Thus, taking an infimum over such $h_j$, we conclude from the definition of $X=\prod_{j=1}^m X_j$ that
\[
\|\ind_Q\|_X=\|\ind_{Q'}\|_{X_{\text{weak}}}\leq C\|T\|_{\vec{X}\to X_{\text{weak}}}\|\ind_{Q'}\|_X.
\]
Now, let $\vec{f}\in L^{\vec{\infty}}_c(\R^d)\cap\vec{X}$. Then the above estimate combined with property \ref{it:dirnondegscalardef2} yields
\begin{align*}
\|T_Q(\vec{f})\|_X&=\prod_{j=1}^m\langle f_j\rangle_Q\|\ind_Q\|_X
\leq C\|T\|_{\vec{X}\to X_{\text{weak}}}\Big\|\prod_{j=1}^m\langle f_j\rangle_Q\ind_{Q'}\Big\|_{X_{\text{weak}}}\\
&\leq C^2\|T\|^2_{\vec{X}\to X_{\text{weak}}}\prod_{j=1}^m\|f_j\|_{X_j}.
\end{align*}
The Fatou property of the $X_j$ now extends this result to all $\vec{f}\in\vec{X}$, as desired. Now, the final assertion follows from Proposition~\ref{prop:reducingmatrixavop} in the case $n=1$.
\end{proof}

To prove Theorem~\ref{thm:nondeg}, we will actually prove a more general result for matrix-weighted Banach function spaces. Given a $\mc{V}$-directional quasi-Banach function space $\mb{X}$ over $\Omega$, following \cite{Ni24b}, we define $\mb{X}_{\text{weak}}$ as those $f\in L^0(\Omega;\mc{V})$ for which
\[
\|f\|_{\mb{X}_{\text{weak}}}:=\sup_{u\in\mc{V}}\|\ind_{\{x\in\Omega:u\in\mc{K}(f)(x)\}}u\|_{\mb{X}}<\infty.
\]
Note that for any $E\subseteq\Omega$, $v\in\mc{V}$, and $f(x):=\ind_E(x)v$, $u\in\mc{H}$ satisfies
\[
u\in\mc{K}(f)(x)
\]
whenever $x\in E$ and $u=\lambda v$ with $|\lambda|\leq 1$, which shows that
\[
\|\ind_E v\|_{\mb{X}_{\text{weak}}}=\sup_{|\lambda|\leq 1}\|\ind_E \lambda v\|_{\mb{X}}=\|\ind_E v\|_{\mb{X}}.
\]
Our general result is as follows:
\begin{theorem}\label{thm:nondeg}
Let $\vec{X}$ be $m$ Banach function spaces over $\R^d$ with the Fatou property and $X:=\prod_{j=1}^m X_j$. Let $\vec{W}$ be matrix weights for which $\ind_Q\|W_j^{-1}\|_{\mc{H}_j\to\mc{H}_j}\in X_j'$ for all cubes $Q$ for $j=1,\ldots,m$. If $T$ is a non-degenerate $m$-linear Calder\'on--Zygmund operator satisfying
\[
\widetilde{T}:\vec{X}_{\vec{W}}\to X_{\mb{W}},
\]
then
\[
\sup_Q\|T_Q\|_{\vec{X}_{\vec{W}}\to X_{\mb{W}}}\lesssim_{\vec{n},K_X} C^2 \|\widetilde{T}\|_{\vec{X}_{\vec{W}}\to (X_{\mb{W}})_{\emph{weak}}}\|\widetilde{T}\|_{\vec{X}_{\vec{W}}\to X_{\mb{W}}}.
\]
\end{theorem}
The reason we need to assume the strong-type bound for for $\widetilde{T}$ is because, in general, the space $(X_{\mb{W}})_{\text{weak}}$ is not a vector space. However, this is needed when applying the directional non-degeneracy property \ref{it:dirnondegdef2}.

For the proof of the theorem, we require several lemmata.
\begin{lemma}\label{lem:dirnondegimpliesnondeg}
Let $T$ be a directionally non-degenerate $m$-linear Calder\'on--Zygmund operator. Then $T$ is non-degenerate with the same constant $C$ and cubes $Q'$.
\end{lemma}
\begin{proof}
Let $Q$ be a cube and let $Q'$ be as in the definition of directional non-degeneracy. Let $u_j\in\mc{H}_j$ be unit vectors for $j=1,\ldots,m$, and note that for any $\vec{h}\in L^{\vec{\infty}}_c(\R^d)$ we have
\[
\widetilde{T}(h_1u_1,\ldots,h_mu_m)=T(\vec{h})u,
\]
where $u:=\bigotimes_{j=1}^m u_j$. Thus, property \ref{it:dirnondegdef2} of directional non-degeneracy implies that there is an $S:Q'\times Q^m\to\C$ with $|S(x,\vec{y})|\leq|Q|^{-m}$ for which for all non-negative $\vec{h}\in L^{\vec{\infty}}_c(\R^d)$ we have
\[
\prod_{j=1}^m\langle h_j\rangle_Q u=\tfrac{1}{2} CT(\vec{h})(x)u+\tfrac{1}{2}\Big(\int_{Q^m}\!S(x,\vec{y})\prod_{j=1}^m h_j(y_j)\,\mathrm{d}\vec{y}\Big)u
\]
for a.e. $x\in Q'$. Hence, since $|S(x,\vec{y})|\leq |Q|^{-m}$,
\[
\prod_{j=1}^m\langle h_j\rangle_Q\ind_{Q'}\leq\tfrac{1}{2}C |T(\vec{h})(x)|+\tfrac{1}{2}\prod_{j=1}^m\langle h_j\rangle_{Q}\ind_{Q'},
\]
so
\[
\prod_{j=1}^m\langle h_j\rangle_Q\ind_{Q'}\leq C|T(\vec{h})(x)|,
\]
proving that $T$ satisfies property \ref{it:dirnondegscalardef2} of non-degeneracy, as asserted.
\end{proof}

The following lemma is a multilinear analogue of \cite[Proposition~5.3]{Go03}:
\begin{lemma}\label{lem:boundedkernel}
Let $Q,Q'$ be cubes in $\R^d$, let $\vec{X}$ be Banach function spaces over $\R^d$, and let $\vec{W}$ be matrix weights. Suppose $\ind_Q u\in (X_j')_{W_j^{-1}}$ for all $u\in\mc{H}_j$ for $j=1,\ldots,m$, and suppose $\ind_{Q'} u\in X_{\mb{W}}$ for all $u\in\mc{H}$. If $S:Q'\times Q^m\to\C$ satisfies
\[
|S(x,\vec{y})|\leq |Q|^{-m},
\]
then the operator
\[
L(\vec{f})(x):=\ind_{Q'}(x)\int_{Q^m}S(x,\vec{y})\bigotimes_{j=1}^m f_j(y_j)\,\mathrm{d}\vec{y}
\]
satisfies
\[
\|L\|_{\vec{X}_{\vec{W}}\to X_{\mb{W}}}\leq C_{\vec{n},K_X}\|T_{Q,Q'}\|_{\vec{X}_{\vec{W}}\to X_{\mb{W}}}.
\]
\end{lemma}
\begin{proof}
Let $\vec{f}\in \vec{X}_{\vec{W}}$ of norm one, and let $\mc{B}_j$ be an orthonormal basis of $\mc{H}_j$ for $j=1,\ldots,m$. For any $x\in Q'$ we have
\begin{align*}
\Big\|\int_{Q^m}S(\cdot,\vec{y})&\bigotimes_{j=1}^m W_j(x)f_j(y_j)\,\mathrm{d}\vec{y}\Big\|_{\mc{H}}
\leq\prod_{j=1}^m\langle\|W_j(x)f_j\|_{\mc{H}_j}\rangle_Q\\
&\leq|Q|^{-m}\prod_{j=1}^m\|\|W_j(x)W_j^{-1}\|_{\mc{H}_j\to\mc{H}_j}\ind_Q\|_{X'_j}\|\|W_jf_j\|_{\mc{H}_j}\|_{X_j}\\
&\leq |Q|^{-m}\prod_{j=1}^m\sum_{e_j\in\mc{B}_j}\|\|W_j^{-1}W_j(x)e_j\|_{\mc{H}_j}\ind_Q\|_{X'_j}\\
&\leq |Q|^{-m}\prod_{j=1}^m\sum_{e_j\in\mc{B}_j}\|A_{(X_j')_{W_j^{-1}},Q}W_j(x)e_j\|_{\mc{H}_j}\\
&\leq n|Q|^{-m}\prod_{j=1}^m\|W_j(x)A_{(X_j')_{W_j^{-1}},Q}\|_{\mc{H}_j\to\mc{H}_j}\\
&= n|Q|^{-m}\Big\|\mb{W}\bigotimes_{j=1}^m A_{(X_j')_{W_j^{-1}},Q}\Big\|_{\mc{H}\to\mc{H}}.
\end{align*}
Thus, it follows from the ideal property of $X$ and Lemma~\ref{lem:reducing_operator_on_operator} that
\begin{align*}
\|L(\vec{f})\|_{X_{\mb{W}}}
&\lesssim_n|Q|^{-m}\Big\|\ind_{Q'}\Big\|\mb{W}\bigotimes_{j=1}^m A_{(X_j')_{W_j^{-1}},Q}\Big\|_{\mc{H}\to\mc{H}}\Big\|_X\\
&\eqsim_{n,K_{X}}|Q|^{-m}\Big\|A_{X_{\mb{W}},Q'}\bigotimes_{j=1}^m A_{(X_j')_{W_j^{-1}},Q}\Big\|_{\mc{H}\to\mc{H}}.
\end{align*}
Thus, the result follows from Proposition~\ref{prop:reducingmatrixavop}.
\end{proof}

\begin{proof}[Proof of Theorem~\ref{thm:nondeg}]
Set $\mb{X}_j:=(X_j)_{W_j}$, $\mb{X}:=X_{\mb{W}}$. We will first show that for all cubes $Q$ we have $\ind_Q u_j\in \mb{X}_j'$ for all $u_j\in\mc{H}_j$ for $j=1,\ldots,m$, and $\ind_Q u\in \mb{X}$ for all $u\in\mc{H}$. For the former result, note that per assumption
\[
\ind_Q \|W_j^{-1}u_j\|_{\mc{H}_j}\leq\ind_Q\|W_j^{-1}\|_{\mc{H}_j\to\mc{H}_j}\|u_j\|_{\mc{H}_j}\in X_j'
\]
and hence, by the ideal property of $X_j'$, $\ind_Q u_j\in\mb{X}_j'=(X_j')_{W_j^{-1}}$ for all $u_j\in\mc{H}_j$, $j=1,\ldots,m$. For the latter result, Let $u_j\in\mc{H}_j\backslash\{0\}$ for $j=1,\ldots,m$, set $u:=\bigotimes_{j=1}^m u_j$, and note that by \cite[Section~3.2]{Ni24b} the spaces $(\mb{X}_j)_{u_j}$ and $\mb{X}_u$, defined by the $h_j\in L^0(\R^d)$ for which $h_ju_j\in \mb{X}_j$ for $j=1,\ldots,m$ and $h\in L^0(\R^d)$ for which $hu\in\mb{X}$ respectively, are Banach function spaces over $\R^d$.

Noting that
\begin{equation}\label{eq:nondegmain1}
\widetilde{T}(h_1u_1,\ldots,h_mu_m)=T(\vec{h})u,
\end{equation}
the boundedness $\widetilde{T}:\vec{\mb{X}}\to\mb{X}$ implies that $T:\vec{\mb{X}}_{\vec{u}}\to \mb{X}_u$. Thus, since $T$ is non-degenerate by Lemma~\ref{lem:dirnondegimpliesnondeg}, it follows from Proposition~\ref{prop:scalarnondeg} that $\ind_Q\in \mb{X}_u$, i.e., $\ind_Q u\in\mb{X}$, for all cubes $Q$, as desired.

Next, since $\|\ind_E u\|_{\mb{X}}=\|\ind_E u\|_{\mb{X}_{\text{weak}}}$ for all $E\subseteq \R^d$ and $u\in\mc{H}$, it follows from \ref{it:dirnondegdef1} and \eqref{eq:nondegmain1} that for all $\vec{u}\in\vec{\mc{H}}$ and $0\leq h_1,\ldots,h_m\in L^\infty_c(\R^d)$ with $\prod_{j=1}^m h_j=\ind_{Q'}$ we have
\[
\ind_Q \bigotimes_{j=1}^m u_j\in C\mc{K}(\widetilde{T}(h_1u_1,\ldots,h_mu_m))
\]
a.e. and, hence, by the directional ideal property of $\mb{X}_{\text{weak}}$,
\begin{align*}
\Big\|\ind_Q \bigotimes_{j=1}^m u_j\Big\|_{\mb{X}}&
\leq C\|\widetilde{T}(h_1u_1,\ldots,h_m u_m)\|_{\mb{X}_{\text{weak}}}\\
&\leq C\|\widetilde{T}\|_{\vec{\mb{X}}\to \mb{X}_{\text{weak}}}\prod_{j=1}^m\|h_j\|W_j u_j\|_{\mc{H}_j}\|_{X_j}.
\end{align*}
By the ideal property of the $X_j$, this is true for all $0\leq h_j\|W_j u_j\|_{\mc{H}_j}\in X_j$ with $\ind_{Q'}\leq\prod_{j=1}^m h_j$. Hence, taking an infimum over all possible $h_j$, we conclude that
\begin{equation}\label{eq:nondegproof1}
\begin{split}
\Big\|\ind_Q \bigotimes_{j=1}^m u_j\Big\|_{\mb{X}}
&\leq C\|\widetilde{T}\|_{\vec{X}_{\vec{W}}\to \mb{X}_{\text{weak}}}\|\ind_{Q'}\prod_{j=1}^m\|W_j u_j\|_{\mc{H}_j}\|_X\\
&=C\|\widetilde{T}\|_{\vec{X}_{\vec{W}}\to \mb{X}_{\text{weak}}}\Big\|\ind_{Q'} \bigotimes_{j=1}^m u_j\Big\|_{\mb{X}}.
\end{split}
\end{equation}
Now, by \ref{it:dirnondegdef2}, for each $0<\eta<1$ we can find an $S:Q'\times Q^m\to\C$ with $|S(x,\vec{y})|\leq|Q|^{-m}$ a.e., and, defining $L$ as in Lemma~\ref{lem:boundedkernel},
\begin{equation}\label{eq:nondegproof2}
\begin{split}
\|T_{Q,Q'}(\vec{f})\|_{\mb{X}}
&\leq (1-\alpha)CK_X\|\widetilde{T}(\ind_Q\vec{f})\|_{\mb{X}}+\alpha K_X\|L(\ind_Q\vec{f})\|_{\mb{X}}\\
&\leq(1-\alpha)CK_X\|\widetilde{T}\|_{\vec{\mb{X}}\to \mb{X}}+\alpha K_X C_{\vec{n},K_X}\|T_{Q,Q'}\|_{\vec{\mb{X}}\to \mb{X}}
\end{split}
\end{equation}
for all $\vec{f}\in L^\infty_c(\R^d;\vec{\mc{H}})\cap \vec{\mb{X}}$ with $\|f_j\|_{\mb{X}_j}\leq 1$ for $j=1,\ldots,m$. For arbitrary $\vec{f}\in\vec{\mb{X}}$ of norm bounded by $1$, define
\[
h_j^N:=\ind_{\{x\in\R^d:|x|\leq N,\,\|f_j(x)\|_{\mc{H}_j}\leq N\}}
\]
and set $f_{j,N}:=h_j^N f_j\in L^\infty_c(\R^d;\mc{H}_j)\cap\mb{X}_j$ for positive integers $N$ and $j=1,\ldots,m$. Then
\[
\ind_{Q'}\prod_{j=1}^m\|W_j(\cdot)\langle f_{j,N}\rangle_Q\|_{\mc{H}_j}\leq\ind_{Q'}\prod_{j=1}^m\langle \|W_j(\cdot)f_j\|_{\mc{H}_j}\rangle_Q,
\]
so by the dominated convergence theorem and the Fatou property of $X=\prod_{j=1}^m X_j$ (see \cite{Sc10}), we may conclude that \eqref{eq:nondegproof2} actually holds for all $\vec{f}\in\vec{\mb{X}}$ of norm bounded by $1$, i.e.,
\[
\|T_{Q,Q'}\|_{\vec{\mb{X}}\to \mb{X}}\leq(1-\alpha)CK_X\|\widetilde{T}\|_{\vec{\mb{X}}\to \mb{X}}+\alpha K_X C_{\vec{n},K_X}\|T_{Q,Q'}\|_{\vec{\mb{X}}\to \mb{X}}.
\]
Choosing $\alpha=\tfrac{1}{2}K_X^{-1}C_{\vec{n},K_X}^{-1}$ and using \eqref{eq:nondegproof1}, we conclude that
\[
\|T_Q\|_{\vec{\mb{X}}\to\mb{X}}\leq C\|\widetilde{T}\|_{\vec{\mb{X}}\to \mb{X}_{\text{weak}}}\|T_{Q,Q'}\|_{\vec{\mb{X}}\to \mb{X}}\lesssim_{\vec{n},K_X} C^2\|\widetilde{T}\|_{\vec{\mb{X}}\to \mb{X}_{\text{weak}}}\|\widetilde{T}\|_{\vec{\mb{X}}\to \mb{X}}.
\]
Taking a supremum over all cubes $Q$ proves the assertion.
\end{proof}

\begin{proof}[Proof of Theorem~\ref{thm:nondeglebesgue}]
We apply Theorem~\ref{thm:nondeg} with $X_j=L^{p_j}(\R^d)$. The assertion then follows from the definition of $[\vec{W}]_{\vec{p}}$.
\end{proof}

As our non-degeneracy condition is quite convoluted, we provide an example.
\begin{example}\label{ex:riesznondeg}
Let $T$ be the $m$-linear Riesz transform of the first coordinate, i.e., the $m$-linear Calder\'on--Zygmund operator associated to the kernel
\[
K(x,\vec{y})=\frac{\sum_{j=1}^m (x^1-y_j^1)}{\Big(\sum_{j=1}^m|x-y_j|\Big)^{md+1}},
\]
where $x\in\R^d$, $\vec{y}=(y_1,\ldots,y_m)\in(\R^d)^m$, and where $x=(x^1,\ldots,x^d)$, $y_j=(y_j^1,\ldots,y_j^m)$ denote the coordinates of $x,y_j\in\R^d$, $j=1,\ldots,m$.

Given a cube $Q$ in $\R^d$, set
\[
Q':=Q+2m\ell(Q)e_1,
\]
where $e_1=(1,0,\ldots,0)\in\R^d$. To check \ref{it:dirnondegdef1}, let $0\leq h_1,\ldots,h_m\in L^\infty_c(\R^d)$ supported in $Q'$ with $\prod_{j=1}^m h_j=\ind_{Q'}$. Since for $x\in Q$ and $\vec{y}\in (Q')^m$ we have
\[
(2m-1)\ell(Q)\leq y_j^1-x^1\leq (2m+1)\ell(Q)
\]
and
\[
0\leq y_j^k-x^k\leq \ell(Q)
\]
for all $k\in\{2,\ldots,d\}$, $j\in\{1,\ldots,m\}$, it follows that for
\[
C_{m,d}:=\frac{(d-1+(2m+1)^2)^{\frac{1}{2}}}{m(2m-1)}
\]
we have
\begin{align*}
\Big(\sum_{j=1}^m|x-y_j|\Big)^{md+1}&\leq m(2m-1)C_{m,d}  \ell(Q)^{md+1}=C_{m,d}|Q|^m m(2m-1)\ell(Q)\\
&\leq C_{m,d}|Q|^m\sum_{j=1}^m (y_j^1-x^1).
\end{align*}
Thus, since $|Q|=|Q'|$,
\begin{align*}
C_{m,d} |T(\vec{h})(x)|&=C_{m,d} \int_{(Q')^m}\!\frac{\sum_{j=1}^m (y_j^1-x^1)}{\Big(\sum_{j=1}^m|x-y_j|\Big)^{md+1}}\prod_{j=1}^m h_j(y_j)\,\mathrm{d}\vec{y}\\
&\geq \prod_{j=1}^m\avint_{Q'}\!h_j(y)\,\mathrm{d}y\geq \Big(\avint_{Q'}\prod_{j=1}^m h_j(y)^{\frac{1}{m}}\,\mathrm{d}y\Big)^m=1
\end{align*}
for a.e. $x\in Q$, as desired.

For \ref{it:dirnondegdef2}, let $\vec{f}\in L^{\vec{\infty}}_c(\R^d;\vec{\mc{H}})$ be supported in $Q$ and let $0<\alpha<1$. Then we have
\[
\bigotimes_{j=1}^m\langle f_j\rangle_Q=(1-\alpha)C_{m,d} \widetilde{T}(\vec{f})(x)+\alpha\int_{Q^m}\!S(x,\vec{y})\bigotimes_{j=1}^mf_j(y_j)\,\mathrm{d}\vec{y},
\]
where
\[
S(x,\vec{y})=\alpha^{-1}\big(|Q|^{-m}-(1-\alpha)C_{m,d} K(x,\vec{y})\big).
\]
If $x\in Q'$ and $\vec{y}\in Q^m$, then by a similar computation as above,
\[
K(x,\vec{y})=\frac{\sum_{j=1}^m (x^1-y_j^1)}{\Big(\sum_{j=1}^m|x-y_j|\Big)^{md+1}}\geq C_{m,d} ^{-1} |Q|^{-m}.
\]
Hence, we have
\[
\alpha|S(x,\vec{y})|=|Q|^{-m}-(1-\alpha)C_{m,d} K(x,\vec{y})\leq(1-(1-\alpha))|Q|^{-m}=\alpha|Q|^{-m},
\]
proving the desired result.
\end{example}

\section{Proof of the main result}\label{sec:proofs}

\begin{proof}[Proof of Theorem~\ref{thm:A}]
The implications \ref{it:thmA3}$\Rightarrow$\ref{it:thmA1} \& \ref{it:thmA2} correspond to \ref{it:thmB1} \& \ref{it:thmB2} of Theorem~\ref{thm:B} respectively. The implication \ref{it:thmA2}$\Rightarrow$\ref{it:thmA3} follows from Proposition~\ref{prop:weaktypechar}, and the implication \ref{it:thmA1}$\Rightarrow$\ref{it:thmA3} from Theorem~\ref{thm:nondeglebesgue} applied with $T$ being the multilinear Riesz transform of Example~\ref{ex:riesznondeg}.
\end{proof}

\appendix

\section{}

Here we give the details of the proofs of Propositions~\ref{prop:weak_bound_multilinear_CZ},~\ref{prop:pointwise_bound_CZ_grand_maximal} and~\ref{prop:Cotlat_CZ_multilinear}. We consider a $m$-linear Calder\'{o}n--Zygmund operator with a kernel $K$ having a modulus of continuity $\omega$ satisfying the Dini condition. We denote by $C_{K}>0$ the best constant in the boundedness and smoothness conditions for the kernel $K$. Moreover, we denote by $D$ the doubling constant of $\omega$.

\begin{proposition}
    The operator $T$ admits a bounded extension $L^{\vec{1}}(\R^d)\to L^{\frac{1}{m},\infty}(\R^d)$ with
    \begin{equation*}
        \Vert T\Vert_{L^{\vec{1}}(\R^d)\to L^{\frac{1}{m},\infty}(\R^d)}\leq C(\Vert T\Vert_{L^{\vec{q}}(\R^d)\to L^{q}(\R^d)}+C_{K}\Vert\omega\Vert_{\mathrm{Dini}}), 
    \end{equation*}
    where the constant $C>0$ depends only on $d,m,\vec{q}$ and the doubling constant of $\omega$.
\end{proposition}

\begin{proof}
    We adapt the proof of \cite[Lemma 6.1]{Damian2018}. Let $\mc{D}$ be the standard dyadic grid in $\R^d$. Let $\vec{f}=(f_1,\ldots,f_m)$ be a $m$-tuple of functions that are finite linear combinations with complex coefficients of characteristic functions of cubes in $\mc{D}$. Let $\lambda>0$ be arbitrary. We show that
    \begin{equation*}
        \lambda|\{|T(\vec{f})|>\lambda\}|^{m}\leq C(\Vert T\Vert_{L^{\vec{q}}(\R^d)\to L^{q}(\R^d)}+C_{K}\Vert\omega\Vert_{\mathrm{Dini}}),
    \end{equation*}
    where the constant $C>0$ depends only on $d,m,\vec{q}$ and the doubling constant of $\omega$. Without loss of generality, we may assume that for every $j=1,\ldots,m$ it is not the case that $f_j=0$ a.e.~on $\R^d$.

    We consider positive real numbers $a_1,\ldots,a_m$ to be chosen below. Then, for each $j=1,\ldots,m$, we consider a Calder\'{o}n--Zygmund decomposition of $f_j$ at height $\alpha_j\lambda$. Thus, we obtain functions measurable functions $g_j,b_j:\R^d\to\C$, a finite family $\mc{F}_j$ of pairwise disjoint cubes in $\R^d$ and functions $b_{j,Q}:\R^d\to\C$, $Q\in\mc{F}_j$ with the following properties:
    \begin{gather*}
        f_j=g_j+b_j,\\
        b_j=\sum_{Q\in\mc{F}_j}b_{j,Q},\\
        \Vert g_j\Vert_{L^1(\R^d)}\leq\Vert f_j\Vert_{L^1(\R^d)},\quad\Vert g_j\Vert_{L^{\infty}(\R^d)}\leq 2^d\alpha_j\lambda,\\
        b_{j,Q}\text{ vanishes outside }Q\text{ and }\int_{Q}b_{j,Q}(x)\mathrm{d}x=0,~\Vert b_{j,Q}\Vert_{L^1(\R^d)}\leq 2^{d+1}\alpha_j\lambda|Q|\text{ for each }Q\in\mc{F},\\
        \sum_{Q\in\mc{F}_j}|Q|\leq\frac{1}{\alpha_j\lambda}\Vert f_j\Vert_{L^1(\R^d)}.
    \end{gather*}
    Set $\Omega_j:=\bigcup_{Q\in\mc{F}_j}(3Q)$ for each $j=1,\ldots,m$. Set also $E:=\R^d\setminus\Big(\bigcup_{j=1}^{m}\Omega_j\Big)$. Then, we have
    \begin{align*}
        |\{|T(\vec{f})|>\lambda\}|\leq\sum_{j=1}^{m}|\Omega_j|+\sum_{\vec{h}\in\{g,b\}^{m}}\Big|\Big\{x\in E:~|T(h_1,\ldots,h_m)|>\frac{\lambda}{2^m}\Big\}\Big|.
    \end{align*}
    We estimate each term separately. First of all, we have
    \begin{align*}
        \sum_{j=1}^{m}|\Omega_j|\leq\sum_{j=1}^{m}\sum_{Q\in\mc{F}_j}|3Q|\lesssim_{d}\sum_{j=1}^{m}\sum_{Q\in\mc{F}_j}|Q|\leq\frac{1}{\lambda}\sum_{j=1}^{m}\frac{1}{a_j}\Vert f_j\Vert_{L^1(\R^d)}.
    \end{align*}
    We turn our attention to the other $2^m$ terms. For the term where only ``good'' functions appear, we have
    \begin{align*}
        &\Big|\Big\{x\in E:~|T(\vec{g})|>\frac{\lambda}{2^m}\Big\}\Big|
        \leq\Big(\frac{2^m}{\lambda}\Big)^{q}\Vert T(\vec{g})\Vert_{L^q(\R^d)}^q\leq
        \Big(\frac{2^m}{\lambda}\Big)^{q}\Vert T\Vert_{L^{\vec{q}}(\R^d)\to L^q(\R^d)}^q\Vert\vec{g}\Vert_{L^{\vec{q}}(\R^d)}^q\\
        &\lesssim_{d,\vec{q}}\Big(\frac{2^m}{\lambda}\Big)^{q}\Vert T\Vert_{L^{\vec{q}}(\R^d)\to L^q(\R^d)}^q\Big(\prod_{j=1}^{m}(a_j\lambda)^{\frac{1}{q_j'}}\Big)\Big(\prod_{j=1}^{m}\Vert f_j\Vert_{L^1(\R^d)}^{\frac{1}{q_j}}\Big)^{q}.
    \end{align*}
    Next, we estimate the terms where there is a mixture of ``good'' and ``bad'' functions, or where there are only ``bad'' functions. Let $\vec{h}\in\{g,b\}^{m}$ such that it is not the case that $h_j=g_j$ for all $j=1,\ldots,m$. Let $B:=\{j\in\{1,\ldots,m\}:~h_j=b_j\}$ and $A:=\{1,\ldots,m\}\setminus B$. Consider enumerations $A:=\{s_1,\ldots,s_n\}$ and $B:=\{r_1,\ldots,r_{\ell}\}$. We abbreviate $\mathrm{d}y_{A}:=\mathrm{d}y_{s_1}\ldots\mathrm{d}y_{s_{n}}$ and $\mathrm{d}y_{B}:=\mathrm{d}y_{r_1}\ldots\mathrm{d}y_{r_{\ell}}$. Then, for each $x\in E$ we have
    \begin{align*}
        &|T(\vec{h})(x)|=\Big|\int_{(\R^d)^m}K(x,\vec{y})\prod_{j=1}^m h_j(y_j)\,\mathrm{d}\vec{y}\Big|\\
        &=
        \int_{(\R^d)^{A}}\Big(\sum_{\substack{Q_t\in\mc{F}_{r_t}\\t=1,\ldots,\ell}}\int_{Q_1\times\ldots\times Q_{\ell}}K(x,\vec{y})\prod_{j\in B}b_{j,Q_j}(y_j)\mathrm{d}y_{B}\Big)\prod_{j\in A}g_j(y_j)\mathrm{d}y_{A}\\
        &=\int_{(\R^d)^{A}}\bigg(\sum_{\substack{Q_t\in\mc{F}_{r_t}\\t=1,\ldots,\ell}}\int_{Q_1\times\ldots\times Q_{\ell}}(K(x,\vec{y})-K(x,y_1,\ldots,y_{r_1-1},c_{Q_{1}},y_{r_1+1},\ldots,y_m))\\
        &\times\prod_{j\in B}b_{j,Q_j}(y_j)\mathrm{d}y_{B}\bigg)\prod_{j\in A}g_j(y_j)\mathrm{d}y_{A}\\
        &\leq C_{K}\int_{(\R^d)^{A}}\bigg(\sum_{\substack{Q_t\in\mc{F}_{r_t}\\t=1,\ldots,\ell}}\int_{Q_1\times\ldots\times Q_{\ell}}\omega\Big(\frac{|y_{r_1}-c_{Q_1}|}{\sum_{j=1}^{m}|x-y_j|}\Big)\frac{1}{\Big(\sum_{j=1}^{m}|x-y_j|\Big)^{md}}\\
        &\times\prod_{j\in B}b_{j,Q_j}(y_j)\mathrm{d}y_{B}\bigg)
        \prod_{j\in A}g_j(y_j)\mathrm{d}y_{A}\\
        &=C_{K}\int_{(\R^d)^{A}}\bigg(\sum_{\substack{Q_t\in\mc{F}_{r_t}\\t=1,\ldots,\ell}}\int_{Q_1\times\ldots\times Q_{\ell}}\omega\Big(\frac{|y_{r_1}-c_{Q_1}|}{\sum_{j\in A}|y_j|+\sum_{j\in B}|x-y_j|}\Big)\\
        &\times\frac{1}{\Big(\sum_{j\in A}|y_j|+\sum_{j\in B}|x-y_j|\Big)^{md}}\prod_{j\in B}b_{j,Q_j}(y_j)\mathrm{d}y_{B}\bigg)\prod_{j\in A}g_j(y_j)\mathrm{d}y_{A}\\
        &\lesssim_{m,d,D}C_{K}\int_{(\R^d)^{A}}\bigg(\sum_{\substack{Q_t\in\mc{F}_{r_t}\\t=1,\ldots,\ell}}\int_{Q_1\times\ldots\times Q_{\ell}}\omega\Big(\frac{\ell(Q_{1})}{\sum_{j\in A}|y_j|+\sum_{j\in B}|x-c_{Q_j}|}\Big)\\
        &\times\frac{1}{\Big(\sum_{j\in A}|y_j|+\sum_{j\in B}|x-c_{Q_j}|\Big)^{md}}\prod_{j\in B}b_{j,Q_j}(y_j)\mathrm{d}y_{B}\bigg)\prod_{j\in A}g_j(y_j)\mathrm{d}y_{A}\\
        &\lesssim_{d,m}C_{K}\Big(\prod_{j\in A}a_j\Big)\lambda^{\#A}\int_{(\R^d)^{A}}\bigg(\sum_{\substack{Q_t\in\mc{F}_{r_t}\\t=1,\ldots,\ell}}\int_{Q_1\times\ldots\times Q_{\ell}}\omega\Big(\frac{\ell(Q_{1})}{\sum_{j\in A}|y_j|+\sum_{j\in B}|x-c_{Q_j}|}\Big)\\
        &\times\frac{1}{\Big(\sum_{j\in A}|y_j|+\sum_{j\in B}|x-c_{Q_j}|\Big)^{md}}\prod_{j\in B}b_{j,Q_j}(y_j)\mathrm{d}y_{B}\bigg)\mathrm{d}y_{A}\\
        &\lesssim_{d,m}C_{K}\Big(\prod_{j=1}^{m}a_j\Big)\lambda^{m}\int_{(\R^d)^{A}}\sum_{\substack{Q_t\in\mc{F}_{r_t}\\t=1,\ldots,\ell}}\Big(\prod_{k=1}^{\ell}|Q_k|\Big)\omega\Big(\frac{\ell(Q_{1})}{\sum_{j\in A}|y_j|+\sum_{j\in B}|x-c_{Q_j}|}\Big)\\
        &\times\frac{1}{\Big(\sum_{j\in A}|y_j|+\sum_{j\in B}|x-c_{Q_j}|\Big)^{md}}\mathrm{d}y_{A}\\
        &\leq C_{K}\Big(\prod_{j=1}^{m}a_j\Big)\lambda^{m}\int_{(\R^d)^{A}}\sum_{\substack{Q_t\in\mc{F}_{r_t}\\t=1,\ldots,\ell}}\Big(\prod_{k=1}^{\ell}|Q_k|\Big)\omega\Big(\frac{\ell(Q_{1})}{\sum_{j\in B}|x-c_{Q_j}|}\Big)\\
        &\times\frac{1}{\Big(\sum_{j\in A}|y_j|+\sum_{j\in B}|x-c_{Q_j}|\Big)^{md}}\mathrm{d}y_{A}\\
        &\lesssim_{m,d}C_{K}\Big(\prod_{j=1}^{m}a_j\Big)\lambda^{m}\sum_{\substack{Q_t\in\mc{F}_{r_t}\\t=1,\ldots,\ell}}\Big(\prod_{k=1}^{\ell}|Q_k|\Big)\omega\Big(\frac{\ell(Q_{1})}{\sum_{j\in B}|x-c_{Q_j}|}\Big)\frac{1}{\Big(\sum_{j\in B}|x-c_{Q_j}|\Big)^{(m-\#A)d}}\\
        &\lesssim_{m,d,D}C_{K}\Big(\prod_{j=1}^{m}a_j\Big)\lambda^{m}\sum_{\substack{Q_t\in\mc{F}_{r_t}\\t=1,\ldots,\ell}}\int_{Q_1\times\ldots\times Q_{\ell}}\omega\Big(\frac{\ell(Q_{1})}{\sum_{j\in B}|x-y_j|}\Big)\frac{1}{\Big(\sum_{j\in B}|x-y_j|\Big)^{\#Bd}}\mathrm{d}y_{B}\\
        &\leq C_{K}\Big(\prod_{j=1}^{m}a_j\Big)\lambda^{m}\sum_{k=1}^{\ell}\sum_{\substack{Q_t\in\mc{F}_{r_t}\\t=1,\ldots,\ell\\\ell(Q_k)=\max_{t=1,\ldots,\ell}\ell(Q_t)}}\int_{Q_1\times\ldots\times Q_{\ell}}\omega\Big(\frac{\ell(Q_{k})}{\sum_{j\in B}|x-y_j|}\Big)\\
        &\times\frac{1}{\Big(\sum_{j\in B}|x-y_j|\Big)^{\#Bd}}\mathrm{d}y_{B}\\
        &\leq C_{K}\Big(\prod_{j=1}^{m}a_j\Big)\lambda^{m}\sum_{k=1}^{\ell}\sum_{\substack{Q_t\in\mc{F}_{r_t}\\t=1,\ldots,\ell}}\int_{Q_1\times\ldots\times Q_{\ell}}\omega\Big(\frac{\ell(Q_{k})}{\sum_{j\in B}|x-y_j|}\Big)\frac{1}{\Big(\sum_{j\in B}|x-y_j|\Big)^{\#Bd}}\mathrm{d}y_{B}\\
        &\leq C_{K}\Big(\prod_{j=1}^{m}a_j\Big)\lambda^{m}\sum_{k=1}^{\ell}\sum_{Q_k\in\mc{F}_{r_k}}\int_{Q_k}\bigg(\int_{(\R^d)^{B\setminus\{r_k\}}}\omega\Big(\frac{\ell(Q_{k})}{|x-y_{r_k}|}\Big)\\
        &\frac{1}{\Big(\sum_{j\in B}|x-y_j|\Big)^{\#Bd}}\mathrm{d}y_{B\setminus\{r_k\}}\bigg)\mathrm{d}y_{r_k}\\
        &\lesssim_{m,d}C_{K}\Big(\prod_{j=1}^{m}a_j\Big)\lambda^{m}\sum_{k=1}^{\ell}\sum_{Q_k\in\mc{F}_{r_k}}\int_{Q_k}\omega\Big(\frac{\ell(Q_{k})}{|x-y_{r_k}|}\Big)\frac{1}{|x-y_{r_k}|^{d}}\mathrm{d}y_{r_k}.
    \end{align*}
    Integrating over $x\in E$ we obtain
    \begin{align*}
        &\int_{E}|T(\vec{h})(x)|\mathrm{d}x\\
        &\lesssim_{m,d,D}C_{K}\Big(\prod_{j=1}^{m}a_j\Big)\lambda^{m}\sum_{k=1}^{\ell}\sum_{Q_k\in\mc{F}_{r_k}}\int_{Q_k}\Big(\int_{E}\omega\Big(\frac{\ell(Q_{k})}{|x-y_{r_k}|}\Big)\frac{1}{|x-y_{r_k}|^{d}}\mathrm{d}x\Big)\mathrm{d}y_{r_k}\\
        &\lesssim_{d,D}C_{K}\Big(\prod_{j=1}^{m}a_j\Big)\lambda^{m}\sum_{k=1}^{\ell}\sum_{Q_k\in\mc{F}_{r_k}}\int_{Q_k}\Vert\omega\Vert_{\mathrm{Dini}}\mathrm{d}y_{r_k}\\
        &\leq C_{K}\Vert\omega\Vert_{\mathrm{Dini}}\lambda^{m-1}\sum_{k=1}^{\ell}\bigg(\prod_{\substack{j=1\\j\neq k}}^{m}a_j\bigg)\Vert f_k\Vert_{L^1(\R^d)}.
    \end{align*}
    Thus, we obtain
    \begin{align*}
         |\{|T(\vec{f})|>\lambda\}|&\lesssim_{m,d,D,\vec{q}}\frac{1}{\lambda}\sum_{j=1}^{m}\frac{1}{a_j}\Vert f_j\Vert_{L^1(\R^d)}\\
         &+\frac{1}{\lambda^{q}}\Vert T\Vert_{L^{\vec{q}}(\R^d)\to L^q(\R^d)}^q\Big(\prod_{j=1}^{m}(a_j\lambda)^{\frac{1}{q_j'}}\Big)\Big(\prod_{j=1}^{m}\Vert f_j\Vert_{L^1(\R^d)}^{\frac{1}{q_j}}\Big)^{q}\\
         &+C_{K}\Vert\omega\Vert_{\mathrm{Dini}}\lambda^{m-1}\sum_{k=1}^{\ell}\bigg(\prod_{\substack{j=1\\j\neq k}}^{m}a_j\bigg)\Vert f_k\Vert_{L^1(\R^d)}.
    \end{align*}
    We now choose
    \begin{equation*}
        a_j:=\frac{\Vert f_j\Vert_{L^1(\R^d)}^{1-\frac{1}{m}}}{\prod_{\substack{k=1\\k\neq j}}^{m}\Vert f_k\Vert_{L^1(\R^d)}^{\frac{1}{m}}}\cdot\frac{1}{\lambda^{1-\frac{1}{m}}}\cdot\frac{1}{(\Vert T\Vert_{L^{\vec{q}}(\R^d)\to L^q(\R^d)}+\Vert\omega\Vert_{\mathrm{Dini}})^{\frac{1}{m}}},\quad j=1,\ldots,m.
    \end{equation*}
    A direct computation shows that this choice yields the desired result.
\end{proof}

\begin{proposition}
    For any cube $Q_0\subseteq\R^d$ one has
    \begin{equation*}
        |T(\vec{f}\ind_{3Q_0})(x)|\leq c\Vert T\Vert_{L^{\vec{1}(\R^d)\to L^{\frac{1}{m},\infty}(\R^d)}}\prod_{j=1}^{m}|f_j(x)|+\mc{M}_{T,Q_0}(\vec{f})(x)
    \end{equation*}
    for a.e.~$x\in Q_0$, where the constant $c>0$ depends only on $d$.
\end{proposition}

\begin{proof}
    The proof given in \cite[Lemma 2.1]{Li18} for $m=2$ works mutatis mutandis. One has to replace $2$ by $m$, $\vec{q}$ by $\vec{1}$, $q$ by $1$, $(f_1\chi_{3Q_0},f_2\chi_{3Q_0})$ by $\vec{f}\ind_{3Q_0}$, $(f_1\chi_{3Q},f_2\chi_{3Q})$ by $\vec{f}\ind_{3Q}$, $(f_1,f_2)$ by $\vec{f}$, and, at the very end of the proof, let $x$ be a Lebesgue point of all $|f_1|,\ldots,|f_m|$.
\end{proof}

\begin{proposition}
    For all $\eta\in\Big(0,\frac{1}{m}\Big)$ we have
    \begin{equation*}
        \mc{M}_{T}(\vec{f})\leq C(C_{K}\Vert\omega\Vert_{\mathrm{Dini}}+\Vert T\Vert_{L^{\vec{1}}(\R^d)\to L^{\frac{1}{m},\infty}(\R^d)})\mc{M}(\vec{f})+M_{\eta}(|T(\vec{f})|)\quad\text{a.e.~on }\R^d,
    \end{equation*}
    where the constant $C>0$ depends only on $m,d,\eta$ and the doubling constant of $\omega$.
\end{proposition}

\begin{proof}
    We adapt and combine the proofs of \cite[Theorem 1.2]{Li18} and \cite[Theorem 6.6]{Damian2018}. Let $Q\subseteq\R^d$ be any cube. Then, for all $x,x',\xi\in Q$ we have
    \begin{align*}
        &|T(\vec{f})(\xi)-T(\vec{f}\ind_{3Q})(\xi)|=\Big|\int_{(\R^d)^{m}\setminus(3Q)^m}K(\xi,\vec{y})\Big(\prod_{j=1}^{m}f_j(y_j)\Big)\mathrm{d}\vec{y}\Big|\\
        &\leq\Big|\int_{(\R^d)^{m}\setminus(3Q)^m}(K(\xi,\vec{y})-K(x',\vec{y}))\Big(\prod_{j=1}^{m}f_j(y_j)\Big)\mathrm{d}\vec{y}\Big|
        +|T(\vec{f})(x')|+|T(\vec{f}\ind_{3Q})(x')|.
    \end{align*}
    We estimate
    \begin{align*}
        &\Big|\int_{(\R^d)^{m}\setminus(3Q)^m}(K(\xi,\vec{y})-K(x',\vec{y}))\Big(\prod_{j=1}^{m}f_j(y_j)\Big)\mathrm{d}\vec{y}\Big|\\
        &\leq\sum_{k=1}^{\infty}\int_{(2^k\cdot 3Q)^{m}\setminus(2^{k-1}\cdot 3Q)^{m}}
        |K(\xi,\vec{y})-K(x',\vec{y})|\Big(\prod_{j=1}^{m}|f_j(y_j)|\Big)\mathrm{d}\vec{y}\\
        &\leq C_{K}\sum_{k=1}^{\infty}\int_{(2^k\cdot 3Q)^{m}\setminus(2^{k-1}\cdot 3Q)^{m}}
        \omega\Big(\frac{|\xi-x'|}{\sum_{j=1}^{m}|\xi-y_j|+\sum_{j=1}^{m}|x'-y_j|}\Big)\\
        &\times\frac{1}{\Big(\sum_{j=1}^{m}|\xi-y_j|+\sum_{j=1}^{m}|x'-y_j|\Big)^{md}}
        \Big(\prod_{j=1}^{m}|f_j(y_j)|\Big)\mathrm{d}\vec{y}\\
        &\lesssim_{m,d,D}C_{K}\sum_{k=1}^{\infty}
        \omega\Big(\frac{1}{2^k}\Big)
        \Big(\prod_{j=1}^{m}\langle|f_j|\rangle_{2^k\cdot 3Q}\Big)
        \lesssim_{D}C_k\Vert\omega\Vert_{\mathrm{Dini}}\mc{M}(\vec{f})(x).
    \end{align*}
    Now, taking an $L^{\eta}$-average over $x'\in Q$, we obtain:
    \begin{align*}
        &|T(\vec{f})(\xi)-T(\vec{f}\ind_{3Q})(\xi)|\lesssim_{\eta,m,d,D}C_k\Vert\omega\Vert_{\mathrm{Dini}}\mc{M}(\vec{f})(x)+\langle|T(\vec{f})|^{\eta}\rangle_{Q}^{\frac{1}{\eta}}+\langle|T(\vec{f}\ind_{3Q})|^{\eta}\rangle_{Q}^{\frac{1}{\eta}}\\
        &\leq C_k\Vert\omega\Vert_{\mathrm{Dini}}\mc{M}(\vec{f})(x)+M_{\eta}(|T(\vec{f})|)(x)+\langle|T(\vec{f}\ind_{3Q})|^{\eta}\rangle_{Q}^{\frac{1}{\eta}}.
    \end{align*}
    Since
    \begin{equation*}
        L^{\frac{1}{m},\infty}\subseteq L^{\eta}
    \end{equation*}
    because $0<\eta<\frac{1}{m}$, we deduce
    \begin{align*}
        &\langle|T(\vec{f}\ind_{3Q})|^{\eta}\rangle_{Q}^{\frac{1}{\eta}}\lesssim_{\eta}\Vert T(\vec{f}\ind_{3Q})\Vert_{L^{\frac{1}{m},\infty}\Big(Q,\frac{\mathrm{d}x}{|Q|}\Big)}\\
        &\leq\frac{1}{|Q|^{m}}\Vert T\Vert_{L^{\vec{1}}(\R^d)\to L^{\frac{1}{m},\infty}(\R^d)}\prod_{j=1}^{m}\Vert f_j\ind_{3Q}\Vert_{L^1(\R^d)}\\
        &\lesssim_{d,m}\Vert T\Vert_{L^{\vec{1}}(\R^d)\to L^{\frac{1}{m},\infty}(\R^d)}\mc{M}(\vec{f})(x),
    \end{align*}
    concluding the proof.
\end{proof}

\bibliography{bieb}
\bibliographystyle{alpha}
\end{document}